\documentclass{amsart}
\usepackage[english]{babel}
\usepackage{amssymb}
\usepackage[usenames,dvipsnames]{xcolor}
\usepackage[bookmarksopen,bookmarksdepth=2]{hyperref}
\hypersetup{colorlinks=true,citecolor=NavyBlue,linkcolor=BrickRed,urlcolor=Green} 
\usepackage[font=footnotesize,labelfont=bf]{caption}
\usepackage[mathscr]{eucal}
\usepackage{mathrsfs,calligra}
\usepackage{tikz}
\usepackage{tikz-cd}
\usepackage{tikzsymbols} 
\usepackage{adjustbox}
\usepackage{graphicx}
\usepackage{multirow}
\usepackage{enumitem}
\usepackage{dsfont}
\usepackage{verbatim}
\usepackage{mathtools}
\usepackage{fancyhdr}
\usepackage{threeparttable}
\usepackage{slashbox}
\usepackage{microtype}
\DeclareUnicodeCharacter{2212}{-}

\newcommand{\heart}{\ensuremath\heartsuit}

\renewcommand{\mathbb}{\mathbf}

\usepackage{amssymb,amsmath,amsfonts,amsthm,epsfig,amscd}
\usepackage{stmaryrd}
\usepackage[all,cmtip,poly]{xy}
\usepackage{color}
\usepackage{xcolor}
\usepackage{array}   % for \newcolumntype macro
\usepackage{cleveref}
\crefname{equation}{}{} %for formatting of equation hypertexts
\usepackage{relsize}

\usepackage{float}
\newcolumntype{C}{>{$}c<{$}} % math-mode version of "l" column type

\newcounter{rownumber}[table]
\renewcommand{\therownumber}{(\roman{rownumber})}
\newcolumntype{N}{>{\refstepcounter{rownumber}\therownumber}c}

\usetikzlibrary{decorations.markings}

\makeatletter
\tikzcdset{
  open/.code     = {\tikzcdset{hook, circled};},
  closed/.code   = {\tikzcdset{hook, slashed};},
  circled/.code  = {\tikzcdset{markwith = {\draw (0,0) circle (.5ex);}};},
  slashed/.code  = {\tikzcdset{markwith = {\draw[-] (-.4ex,-.4ex) -- (.4ex,.4ex);}};},
  markwith/.code ={
    \pgfutil@ifundefined%
    {tikz@library@decorations.markings@loaded}%
    {\pgfutil@packageerror{tikz-cd}{You need to say %
      \string\usetikzlibrary{decorations.markings} to use arrows with markings}{}}{}%
    \pgfkeysalso{/tikz/postaction = {
      /tikz/decorate,
      /tikz/decoration={markings, mark = at position 0.4 with {#1}}}
    }
  },
}
\makeatother

\theoremstyle{plain}
\newtheorem{theorem}[equation]{Theorem}

\newtheorem{prop}[equation]{Proposition}
\newtheorem{proposition}[equation]{Proposition}

\newtheorem{lemma}[equation]{Lemma}
\newtheorem{cor}[equation]{Corollary}
\newtheorem{corollary}[equation]{Corollary}

\theoremstyle{definition}

\newtheorem{remark}[equation]{Remark}

\newtheorem{definition}[equation]{Definition}

\numberwithin{equation}{section}
\numberwithin{figure}{subsection}

\foreach \theoremenv in {theorem,thm, proposition,prop, lemma, lem, remark,rem, defn, corollary, definition, cor} {
  \AtBeginEnvironment{\theoremenv}{%
    \setlist[enumerate]{label={(\roman*)}}
  }
}

\newif\iffinalrun
\iffinalrun
\else
\fi

\iffinalrun
  \newcommand{\need}[1]{}
  \newcommand{\mar}[1]{}
\else
  \newcommand{\need}[1]{{\tiny *** #1}}
  \newcommand{\mar}[1]{\marginpar{\raggedright\tiny fixme #1}}
\fi

\newcommand{\F}{\FF}

\newcommand{\Q}{\QQ}

\newcommand{\Z}{\ZZ}

\renewcommand{\O}{\cO}

\renewcommand{\AA}{{\mathbb A}}

\newcommand{\FF}{{\mathbb F}}
\newcommand{\GG}{{\mathbb G}}

\newcommand{\PP}{{\mathbb P}}
\newcommand{\QQ}{{\mathbb Q}}

\newcommand{\ZZ}{{\mathbb Z}}

\newcommand{\cA}{{\mathcal A}}
\newcommand{\cB}{{\mathcal B}}
\newcommand{\cC}{{\mathcal C}}

\newcommand{\cE}{{\mathcal E}}
\newcommand{\cF}{{\mathcal F}}
\newcommand{\cG}{{\mathcal G}}

\newcommand{\cI}{{\mathcal I}}

\newcommand{\cK}{{\mathcal K}}
\newcommand{\cL}{{\mathcal L}}

\newcommand{\cM}{{\mathcal M}}

\newcommand{\cO}{{\mathcal O}}

\newcommand{\cU}{{\mathcal U}}

\newcommand{\cX}{{\mathcal X}}

\newcommand{\cZ}{{\mathcal Z}}

\newcommand{\gM}{\mathfrak{M}}

\newcommand{\gA}{\mathfrak{A}}
\newcommand{\gS}{\mathfrak{S}}

\newcommand{\Fbar}{\overline{\F}}

\newcommand{\Fp}{\F_p}

\newcommand{\Zp}{\Z_p}

\newcommand{\Qp}{\Q_p}

\DeclareMathOperator{\coker}{coker}

\DeclareMathOperator{\Gal}{Gal}
\DeclareMathOperator{\GL}{GL}

\DeclareMathOperator{\Hom}{Hom}
\DeclareMathOperator{\sheafHom}{\mathcal{H}\textit{\kern -1pt{om}}\,}
\DeclareMathOperator{\relSpec}{\mathcal{S}\textit{\kern -1pt{pec}}\,}

\DeclareMathOperator{\Mat}{Mat}

\DeclareMathOperator{\Proj}{Proj}

\DeclareMathOperator{\Spec}{Spec}

\DeclareMathOperator{\Spf}{Spf}

\DeclareMathOperator{\Sym}{Sym}

\newcommand{\id}{\operatorname{id}}
\newcommand{\nr}{\mathrm{nr}}
\newcommand{\sep}{\mathrm{sep}}

\newcommand{\ur}{\mathrm{ur}}

\newcommand{\rhobar}{\overline{\rho}}

\newcommand{\et}{\text{\'{e}t}}
\newcommand{\etale}{\'{e}tale~}

\newcommand{\onto}{\twoheadrightarrow}

\newcommand{\Gm}{\GG_m}

\newcommand{\<}{\langle}
\renewcommand{\>}{\rangle}

\newcommand{\LG}{L\mathrm{G}}
\newcommand{\LGp}{L^{+}\mathrm{G}}
\newcommand{\LGn}{L^{-}\mathrm{G}}
\newcommand{\Iw}{L^{+}\cG}
\newcommand{\LGtau}{L\mathrm{G}^{\tau}}
\newcommand{\LGbd}{L\mathrm{G}^{\mathrm{bd}, v v^{\mu}}}
\newcommand{\Kpr}{L^{+}_{r}\mathrm{G}} 
\newcommand{\KpI}{L^{+}_1\mathrm{G}} 
\newcommand{\KpII}{L^{+}_2\mathrm{G}}
\newcommand{\Kn}{L^{-}_1\mathrm{G}}
\newcommand{\Knr}{L^{-}_{r}\mathrm{G}}

\newcommand{\Grtau}{\mathrm{Gr}_1^{\tau}}
\newcommand{\Grbd}{\mathrm{Gr}_1^{\mathrm{bd}, v v^{\mu}}}

\newcommand{\Y}{Y^{\eta, \tau}}
\newcommand{\YF}{Y^{\eta, \tau}_{\F}}

\newcommand{\Uz}{\widetilde{U}(\widetilde{z})}

\newcommand{\tilz}{\widetilde{z}}
\newcommand{\tilw}{\widetilde{w}}
\newcommand{\till}{\widetilde{l}}
\newcommand{\tilr}{\widetilde{r}}

\newcommand{\tilB}{\widetilde{B}}
\newcommand{\tilY}{\widetilde{Y}^{\eta, \tau}}
\newcommand{\tilZ}{\widetilde{\cZ}^{\tau}}

\newcommand{\tilZnm}{\widetilde{\cZ}^{\tau, \mathrm{nm}}}
\newcommand{\tilYz}{\widetilde{Y}^{\eta, \tau}(\widetilde{z})}
\newcommand{\tilZz}{\widetilde{\cZ}^{\tau}(\widetilde{z})}
\newcommand{\tilZnmz}{\tilZnm(\tilz)}

\newcommand{\Baj}{{Ba}_j(\widetilde{z}_j)}
\newcommand{\tilBaj}{\widetilde{Ba}_j(\widetilde{z}_j)}
\newcommand{\tilBa}{\widetilde{Ba}(\widetilde{z})}

\newcommand{\ort}{\mathrm{or}}
\newcommand{\pr}{\mathrm{pr}}

\newcommand{\trn}{\mathfrak{T}}
\newcommand{\gr}{\mathfrak{t}}
\newcommand{\Ad}{\mathrm{Ad}}

\def\bal#1\nal{\begin{align*}#1\end{align*}}
\def\lbal#1\lnal{\begin{flalign*}#1\end{flalign*}}
\newcommand{\WjL}{(W_j)_{1,1}}
\newcommand{\WjR}{(W_j)_{2,2}}

\begin{document}
\selectlanguage{english}
\title[Non--generic components of the Emerton--Gee stack for $\GL_2$]{Non--generic components of the Emerton--Gee stack for $\GL_2$}
\author{Kalyani Kansal}
\address{Department of Mathematics, Imperial College London, London SW7 2AZ, UK}
\author{Ben Savoie}
\address{Department of Mathematics, Rice University, 6100 Main Street, Houston, TX 77005}

\begin{abstract}
    Let $K$ be a finite unramified extension of $\Qp$ with $p > 3$. We study the extremely non--generic irreducible components in the reduced part of the Emerton--Gee stack for $\GL_2$. We show precisely which irreducible components are smooth, which are normal, and which have Gorenstein normalizations. We show that the normalizations of the irreducible components admit smooth--local covers by resolution--rational varieties. We also determine the singular loci in the components, and use our results to update expectations about the conjectural categorical $p$--adic Langlands correspondence. 
\end{abstract}

\maketitle

\section{Introduction}

Let $p$ be a fixed prime. Let $K$ be a finite extension of $\Q_p$ with residue field $k$ of degree $f$ over $\Fp$, and absolute Galois group $G_K$. In \cite{emerton2022moduli}, Emerton and Gee studied the stack $\cX_{d}$ of \'etale $(\varphi, \Gamma)$--modules of rank $d$ defined over the formal spectrum of the ring of integers of a large finite extension of $\Qp$ with residue field $\F$. As discussed in \cite{EGHcategorical}, $\cX_d$ is expected to play the role of the stack of $L$--parameters in the so far conjectural categorical $p$--adic Langlands correspondence for $\GL_d(K)$. 

By \cite[Thm.~1.2.1]{emerton2022moduli}, $\cX_d$ is a Noetherian formal algebraic stack and its underlying reduced substack $\cX_{d, \text{red}}$ is an algebraic stack of finite type over $\F$. The irreducible components of $\cX_{d, \text{red}}$ admit a natural labelling by Serre weights, which are the irreducible representations of $\GL_d(k)$ with coefficients in $\F$. 
Each Serre weight for $\GL_2(k)$ is described by (ordered) pairs of $f$--tuples of integers $\mathbb{m} = (m_j)_j$ and $\mathbb{n} = (n_j)_j$ with $n_j \in [0, p-1]$ for each $j$, and is correspondingly denoted $\sigma_{\mathbb{m}, \mathbb{n}}$ (see Section \ref{subsec:serre-wt} for details). Let $\cX(\sigma_{\mathbb{m}, \mathbb{n}})$ be the irreducible component of $\cX_{2, \text{red}}$ labelled by $\sigma_{\mathbb{m}, \mathbb{n}}$.  The following is our main theorem, upgrading the main result of \cite{gkksw}.
\begin{theorem}[Theorems \ref{thm:serre-weight-sing}, \ref{thm:steinberg}]\label{thm:1} 
Let $p > 3$, $K$ unramified over $\Qp$ of degree $f$, and $\sigma_{\mathbb{m}, \mathbb{n}}$ a Serre weight. Then the following are true:
\begin{enumerate}
    \item The component $\cX(\sigma_{\mathbb{m}, \mathbb{n}})$ is not smooth if and only if either
    \begin{enumerate}[label=(\alph*)]
        \item $n_j = p-2$ for each $j \in \Z/f\Z$, or
        \item there exists a subset $\{i-k, \dots, i\} \subset \Z/f\Z$ with $n_{i-k} = 0$, $n_j = p-2$ whenever $j \in \{i-k, \dots, i\} \smallsetminus \{i-k, i\}$, and $n_{i} = p-1$.
    \end{enumerate}
     \item If (a) holds, $\cX(\sigma_{\mathbb{m}, \mathbb{n}})$ is not normal and its normalization admits a smooth--local cover by a resolution--rational variety that is not Gorenstein. The non--normal locus on $\cX(\sigma)$ has codimension $f$ and its complement is smooth.
    \item If (b) holds, $\cX(\sigma_{\mathbb{m}, \mathbb{n}})$ is normal and admits a smooth--local cover by a resolution--rational variety.  It is additionally Gorenstein, even lci, if and only if every subset $\{i-k, \dots, i\} \subset \Z/f\Z$ as in (b) has cardinality $2$. The singular locus in $\cX(\sigma)$ has codimension $\geq 2$.
\end{enumerate}
%In all cases, the ring of global functions on the normalization of $\cX(\sigma)$ is isomorphic to $\F[x,y][\frac{1}{y}]$.
\end{theorem}

Here, the definition of a resolution--rational variety is taken from \cite[Defn.~9.1]{kovacs2017rational} and is as follows.
A variety $Y$ is resolution--rational if it admits a resolution of singularities $f: X \to Y$ which is a cohomological equivalence. 
This means that $f_{*}\cO_X \cong \cO_Y$, $R^{i}f_{*} \cO_X = 0$
for $i >0$, $f_{*}\omega_{X} \cong \omega_{Y}$, and $R^{i}f_{*} \omega_X = 0$ for $i >0$, where $\omega_X$ and $\omega_Y$ are the dualizing sheaves on $X$ and $Y$ respectively (in particular, $Y$ is Cohen--Macaulay).

Fixing an algebraic closure $\Fbar$ of $\F$, $\cX_d(\Fbar) = \cX_{d, \text{red}}(\Fbar)$ is the groupoid of continuous two--dimensional representations of $G_K$ with coefficients in $\Fbar$. A study of the singular loci in the irreducible components of $\cX_{2, \text{red}}$ allows us to obtain the following two theorems, which can be viewed as partial generalizations of \cite[Thm.~1]{sandermultiplicities}.

\begin{theorem}[Theorem \ref{thm:non--normal-rep}]\label{thm:2}
 Let $p > 3$, $K$ unramified over $\Qp$, and $\sigma_{\mathbb{m}, \mathbb{n}}$ a Serre weight. The versal ring at $\rhobar \in \cX(\sigma_{\mathbb{m}, \mathbb{n}})(\Fbar)$ is not normal if and only if $n_j = p-2$ for each $j$ and, viewing $\rhobar$ as a $G_K$--representation,
 $$\rhobar \otimes \left( \prod_{j \in \Z/f\Z} \omega_j^{(1-m_j)} \right)$$ is unramified.
%  as a $G_K$--representation, $\rhobar$ is of the form
%  $$\left( \prod_{j \in \Z/f\Z} \omega_j^{(m_{j} - 1)} \right) \otimes \begin{pmatrix}
%     \ur_{\lambda'} & * \\
%     0 & \ur_{\lambda''}
% \end{pmatrix}$$
% where $\lambda'$ and $\lambda''$ are arbitrary units in $\Fbar$, and
% \begin{itemize}
%     \item $*$ is the vanishing class if $\lambda' \neq \lambda''$, and
%     \item $*$ lies in the $1$--dimensional space of extension classes that vanish after restriction to the inertia subgroup if $\lambda'=\lambda''$.
% \end{itemize}
\end{theorem}

\begin{theorem}[Theorem \ref{thm:non-CM-rep}]\label{thm:3}
    Let $p > 3$, $K$ unramified over $\Qp$, and $\sigma_{\mathbb{m}, \mathbb{n}}$ a Serre weight. Suppose there exists $i \in \Z/f\Z$ such that $n_{i+1} = 0$, 
$n_{i} = p-1$, and $n_j=p-2$ for $j \in \Z/f\Z \smallsetminus \{i, i+1\}$. The versal ring at $\rhobar \in \cX(\sigma_{\mathbb{m}, \mathbb{n}})(\Fbar)$ is not smooth if and only if, viewing $\rhobar$ as a $G_K$--representation, 
 $$\rhobar \otimes \left(\omega_{i+1}^{-m_{i+1}} \otimes \prod_{j \neq i+1} \omega_j^{1-m_j}\right)$$ is unramified.
% as a $G_K$--representation, $\rhobar$ is of the form
%  $$\left(\omega_{i+1}^{m_{i+1}} \otimes \prod_{j \neq i+1} \omega_j^{m_j - 1}\right) \otimes \begin{pmatrix}
%     \ur_{\lambda'} & * \\
%     0 & \ur_{\lambda''}
% \end{pmatrix}$$ where $\lambda'$ and $\lambda''$ are arbitrary units in $\Fbar$, and
% \begin{itemize}
%     \item $*$ is vanishing if $\lambda' \neq \lambda''$, and
%     \item $*$ lies in the $1$--dimensional space of extension classes that vanish after restriction to the inertia subgroup if $\lambda'=\lambda''$.
% \end{itemize}
In this case, the versal ring is normal and Cohen--Macaulay. It is Gorenstein, even lci, if and only if $f=2$.
\end{theorem}

Here, $\{\omega_j\}_{j}$ are choices of $G_K$--characters extending the $f$ distinct niveau $1$ fundamental characters of the inertia subgroup (see Section \ref{sec:notation} for details).
%, while $\ur_{\lambda'}$ and $\ur_{\lambda''}$ are the unramified characters mapping the geometric Frobenius to $\lambda'$ and $\lambda''$ respectively. 
Note that the hypothesis in the statement of Theorem \ref{thm:3} does not encapsulate all normal but non--smooth components unless $f=2$.

\subsection{Categorical \texorpdfstring{$p$}{p}--adic Langlands} In \cite[Conj.~6.1.14]{EGHcategorical}, Emerton, Gee and Hellmann conjecture the existence of an exact fully faithful functor 
$\gA$ from a certain derived category of so--called smooth representations of $\GL_d(K)$ to a certain derived category of quasicoherent sheaves on $\cX_d$, satisfying various properties. Without going into the details of what the appropriate categories are and the properties $\gA$ is expected to satisfy, we focus our attention on a few consequences of their conjecture laid out in Sections 6 and 7 of \textit{loc. cit.}, restricting to the case $d=2$. Let $\sigma_{\mathbb{m}, \mathbb{n}}$ be a Serre weight viewed as a representation of $\GL_2(\cO_K)$ via inflation, where $\cO_K$ is the ring of integers of $K$. Let $\cL(\sigma_{\mathbb{m}, \mathbb{n}})$ denote the conjectural sheaf 
$$ \mathfrak{A}\left(\text{c-Ind}_{\GL_2(\cO_K)}^{\GL_2(K)} \sigma_{\mathbb{m}, \mathbb{n}} \right).$$ This conjectural sheaf interpolates certain (actual, and not conjectural) patched modules (see \cite[Rmks.~6.1.26, 6.1.30]{EGHcategorical}). Some of the expectations about $\cL(\sigma_{\mathbb{m}, \mathbb{n}})$ are as follows:

\begin{enumerate}
    \item It is a coherent sheaf concentrated in degree $0$ by \cite[Rmk.~6.1.26]{EGHcategorical} and maximal Cohen--Macaulay on its support. The latter essentially follows from exactness of $\gA$ and Cohen--Macaulay nature of sheaves associated to locally algebraic types (see \cite[Rmk.~6.1.34]{EGHcategorical}). When $\sigma_{\mathbb{m}, \mathbb{n}}$ is non--Steinberg, i.e. $n_j < p-1$ for some $j$, expected compatibility with the geometric Breuil--M\'ezard conjecture implies that the support of $\cL(\sigma_{\mathbb{m}, \mathbb{n}})$ is  $\cX(\sigma_{\mathbb{m}, \mathbb{n}})$. 
    %and the main result of \cite{cegsA} on multiplicities of Breuil--M\'ezard cycles.
    \item Ignoring possible shifts of complexes, $\cL(\sigma_{\mathbb{m}, \mathbb{n}})$ is Grothendieck--Serre self--dual by \cite[Rmk.~6.1.35]{EGHcategorical}. Thus, when $K$ is an unramified non--trivial extension of $\Qp$, and $\sigma_{\mathbb{m}, \mathbb{n}}$ is non--Steinberg, Corollary \ref{cor:max-CM-normal} implies that $\cL(\sigma_{\mathbb{m}, \mathbb{n}})$ is the pushforward of a self--dual maximal Cohen--Macaulay sheaf on the normalization of $\cX(\sigma_{\mathbb{m}, \mathbb{n}})$.
    \item When $K$ is an unramified non--trivial extension of $\Qp$, $p>5$, and $\sigma_{\mathbb{m}, \mathbb{n}}$ is non--Steinberg, $\cL(\sigma_{\mathbb{m}, \mathbb{n}})$ has rank $1$ generically on $\cX(\sigma_{\mathbb{m}, \mathbb{n}})$. This follows from combining the data on codimension of non--normal locus in potentially Barsotti--Tate deformation rings given in the proof of \cite[Thm.~4.6.10]{lhmm}, conjecture about the rank of the generic fiber of sheaves corresponding to locally algebraic types in \cite[Rmk.~6.1.34]{EGHcategorical}, and compatibility with geometric Breuil--M\'ezard conjecture. The key point is that one can always find a non--scalar tame inertial type $\tau$ such that $\sigma_{\mathbb{m}, \mathbb{n}}$ appears in the Jordan--Holder decomposition of the $\GL_2(\cO_K)$--representation associated to $\tau$ by inertial local Langlands, and such that the potentially Barsotti--Tate deformation rings of type $\tau$ are regular in codimension $1$.
\end{enumerate}

Motivated by these expectations, we obtain the following theorem, whose proof uses the statement about the codimension of the singular locus in Theorem \ref{thm:1}.

\begin{theorem}[Corollary \ref{cor:line-bundle-pushfwd}]\label{thm:4}
    Let $p>3$, $K$ an unramified non--trivial extension of $\Qp$, and $\sigma_{\mathbb{m},\mathbb{n}}$ a Serre weight. Let $\iota: \cU \hookrightarrow \cX(\sigma_{\mathbb{m},\mathbb{n}})$ be the smooth open locus in $\cX(\sigma_{\mathbb{m},\mathbb{n}})$. Suppose $\cF$ is a finite type maximal Cohen--Macaulay sheaf on $\cX(\sigma_{\mathbb{m},\mathbb{n}})$ generically of rank $1$. The following are true:
    \begin{enumerate}
        \item The sheaf $\cF$ is isomorphic to the pushforward along $\iota$ of the invertible sheaf  $\iota^{*} \cF$ on $\cU$.
        \item If $\sigma_{\mathbb{m}, \mathbb{n}}$ is non--Steinberg and there does not exist $i$ such that $(n_{i-1}, n_i) = (0, p-1)$, then $\cF$ is the pushforward of a unique invertible sheaf on a smooth algebraic stack of Breuil--Kisin modules admitting a proper, birational map onto $\cX(\sigma_{\mathbb{m},\mathbb{n}})$.
    \end{enumerate}
\end{theorem}

In the setup of Theorem \ref{thm:4} above and assuming a reasonable notion of a dualizing complex on an algebraic stack, note that if $\cF$ is (Grothendieck--Serre) self--dual, then so is $\iota^{*} \cF$. As we will see later in the proof, $\cU$ is isomorphic to its preimage in the aforementioned stack of Breuil--Kisin modules. If the hypothesis in (ii) holds, then the codimension of the complement of the preimage of $\cU$ turns out to be $\geq 2$. Thus, by an argument involving the algebraic Hartog's Lemma on the dual of a line bundle on a smooth variety, $\cF$ is seen to be the (non--derived) pushforward of a self--dual invertible sheaf on the stack of Breuil--Kisin modules we are considering. 

When $p>5$, we therefore expect that one can uniquely characterize $\cL(\sigma_{\mathbb{m}, \mathbb{n}})$ for all non--Steinberg $\sigma_{\mathbb{m}, \mathbb{n}}$ as the pushforward of the unique self--dual invertible sheaf on $\cU$ (equivalently, on an appropriate stack of Breuil--Kisin modules when the hypothesis in (ii) holds), and indeed use this characterization as a key ingredient towards constructing $\gA$ (\textit{c.f.} \cite[Sec.~7.6.10]{EGHcategorical}).

%Notwithstanding our ignorance about constructing dualizing complexes on algebraic stacks and assuming one can do so, when the conditions in Theorem \ref{thm:4}(1) hold, $\cL(\sigma_{\mathbb{m}, \mathbb{n}})$ should be the pushforward along $\iota$ of an invertible sheaf that squares to the dualizing sheaf on $\cU$. By viewing $\cU$ as an open substack of a particular stack of Breuil--Kisin modules as in the proof of Theorem \ref{cor:line-bundle-pushfwd}, it is not hard to check that $\text{Pic}(\cU) \cong \Z^{f+1}$ taking squares is an injective operation in $\text{Pic}(\cU)$.

\subsection{Strategy and outline} 
We begin in Section \ref{sec:notation} by setting up some standard notation and definitions. In Section \ref{sec:charts}, we review many of the constructions from \cite{lhmm} that we use essentially. These constructions pertain to smooth--local charts on various closed substacks of $\cX_{2, \text{red}}$, as well as explicit auxiliary schemes through which maps from certain stacks of Breuil--Kisin modules to $\cX_{2, \text{red}}$ factor locally in the smooth topology. The constraints on $p$ and the requirement for $K$ to be unramified over $\Qp$ appear in this section, most critically in the proof of Proposition \ref{prop:factorization-through-bd}. We use the results of \cite{bellovin2024irregular} to impose ``shape" conditions that cut out irreducible components in these charts, thus setting up smooth--local charts for the non--Steinberg irreducible components of $\cX_{2, \text{red}}$.

 %The assumptions could potentially be relaxed to accommodate 
%smaller $p$ and 
%small ramification degrees for a study of a more restricted set of irreducible components, but we don't attempt that in this article. 
In Section \ref{sec:geometry}, we undertake a detailed study of the geometry of these charts. We start off by analyzing the additional relations on the auxiliary schemes that come from imposing shape conditions. Next, in Section \ref{subsec:Ba-product}, we make the crucial observation that the charts for the irreducible components of $\cX_{2, \text{red}}$ can be written as (typically non--trivial) products of varieties, with each factor in the product somewhat easier to study. Each factor admits a resolution of singularities by a smooth scheme, and we determine the locus where this resolution fails to be an isomorphism. The key argument is in Lemma \ref{lem:bad-sequence}. Next, by doing factor--wise computations, we obtain the cohomologies of the dualizing complex for the charts following the approach in \cite{lhmm}, which reduces the computations to those for the auxiliary schemes and Koszul resolutions. We also compute lower bounds on the ranks of these cohomology groups at the loci where the irreducible components of $\cX_{2, \text{red}}$ fail to be isomorphic to their resolutions of singularities. This allows us to infer that these loci are precisely the singular loci.

% and study its singular locus further by closely many of the approaches used in \cite{lhmm}, while also establishing additionally that most of the smooth--local charts under consideration break up into products of varieties 
% to focus attention on different irreducible components. We decompose the smooth charts into products of varieties that make it easier to study the geometry, compute cohomology of dualizing complexes on these charts, and describe singular loci as vanishing sets of certain ideals. 

Finally, we prove Theorems \ref{thm:serre-weight-sing} and \ref{cor:line-bundle-pushfwd} in Section \ref{sec:combinatorics} by doing various combinatorial calculations and applying the results from Section \ref{sec:geometry} to the irreducible components of $\cX_{2, \text{red}}$ labelled by specific non--Steinberg Serre weights. We handle the Steinberg components separately in Theorem \ref{thm:steinberg} by directly analyzing the points. We also determine explicitly the Galois representations corresponding to the finite type points of the singular loci in the setup of Theorems \ref{thm:non--normal-rep} and \ref{thm:non-CM-rep}.

\begin{remark}
    We use the definition of resolution--rational as given in \cite[Defn.~9.1]{kovacs2017rational}, as well as prove Corollary \ref{cor:ss} using \cite[Prop.~6.10]{kovacs2017rational}. While the paper \cite{kovacs2017rational} has now been retracted because of errors in the main statements, the definition of resolution--rational is still meaningful and the proof of \cite[Prop.~6.10]{kovacs2017rational} is still correct.
\end{remark}

\subsection{Acknowledgements} We are grateful to Brandon Levin for his continuous mentorship, insights and ideas. We would also like to thank Toby Gee for many helpful conversations at different stages of this project, as well as Ariane M\'ezard, Stefano Morra and Fred Diamond for discussions that improved the exposition. Finally, it is a pleasure to acknowledge the debt this paper owes to the work of Bao V. Le Hung, Ariane M\'ezard and Stefano Morra. This work was done while K.K. was supported by the National Science Foundation under Grant No. DMS--1926686.

% \begin{lemma}
% Let $\tau, S$ be such that $\cX(\sigma) = \cZ^{\tau}_S$.
% Suppose $\iota: \cU \hookrightarrow \cZ^{\tau, \mathrm{nm}}_S$ is an open immersion such that $\cU$ pulls back to an open subscheme of $\tilZnm$ that contains all codimension $1$ points. Then 
% $$\cL(\sigma) =\iota_{*} \iota^{*} \cL(\sigma).$$
% \end{lemma}
% \begin{proof}
%     Let $\widetilde{\cL(\sigma)}$ be the pullback of $\cL(\sigma)$ to $\tilZnm$ and let $\widetilde{\iota}: U \hookrightarrow \tilZnm$ be the pullback of $\iota$.
%     Applying \cite[Thm.~3.5]{hassett2004reflexive} to the structure map $\tilZnm \to \Spec \F$, we infer that $\widetilde{\cL(\sigma)} = \widetilde{\iota}_{*} \widetilde{\iota}^{*} \widetilde{\cL(\sigma)}$. Descent finishes the proof.
% \end{proof}

\section{Notation and background}\label{sec:notation} Fix a prime $p > 3$.
Let $K$ be a finite unramified extension of $\Q_p$ of degree $f$ with ring of integers $\cO_K$ and residue field $k$. 
%Thus, the ring of integers in $K$ is $W(k)$. 
Fix an algebraic closure $\overline{K}$ of $K$. For any algebraic extension $L$ of $K$ in $\overline{K}$, denote by $G_L$ the group $\Gal(\overline{K}/L)$. Denote by $I_L$ the inertia subgroup of $G_L$. Let $\pi' \in \overline{K}$ be a fixed $(p^{f}-1)$-th root of $p$. Let $K'$ be a tame extension of $K$ obtained by attaching $\pi'$.

Let $\F$ be a finite extension of $\Fp$ that is the residue field of the ring of integers $\cO$ of a finite field extension $E$ of $\Q_p$ with uniformizer $\varpi$. Denote by $\Fbar$ a fixed algebraic closure of $\F$. We take $\F$ to be large enough so that all embeddings $k \hookrightarrow \Fbar$ factor through $\F$. Fix an embedding $\sigma_0: k \hookrightarrow \F$ and let $$\sigma_{f-j} := \sigma_0^{p^{j}}.$$ The map $j \mapsto \sigma_j$ induces an identification of sets $$\Z/f\Z \xrightarrow{\sim} \Hom_{\Fp}(k, \F)= \Hom_{\Zp}(\cO_K, \cO).$$
For each $j \in \Z/f\Z$, let
$\omega_j: G_K \to \cO^{\times}$ be the character given by 
\begin{align*}
g \mapsto \sigma_j \left(\dfrac{g(\pi')}{\pi'}\right).\end{align*} 
Abusing notation, we will denote the mod $\varpi$ reduction of $\omega_j$ also by $\omega_j$ when it is clear that we are speaking of $\F$--coefficients. 
We will also denote the restriction $\omega_j|_{I_K}$ by $\omega_j$ when it is clear that we are speaking of $I_K$--representations. For $\lambda$ a nonzero element of $\Fbar$, let $$\ur_{\lambda}: G_K \to \Fbar^{\times}$$ be the unramified character mapping the geometric Frobenius element to $\lambda$.
\subsection{Tame inertial types}
A tame inertial type is the isomorphism class of a representation $\tau: I_K \to \GL_2(\cO)$ which has an open kernel, factors through the tame quotient of $I_K$, and extends to $G_K$. Such a representation is of the form $\tau \cong \eta_1 \oplus \eta_2$. We say that $\tau$ is a \textit{principal series} tame type if both $\eta_1$ and $\eta_2$ extend to characters of $G_K$, cuspidal otherwise. It is \textit{non--scalar} if $\eta_1 \neq \eta_2$. 
% Through standard arguments exploiting normality of $I_K$ in $G_K$ and conjugation by a Frobenius element, one notes that 
When $\tau$ is a principal series type, $\eta_1$ and $\eta_2$ factor through $I_K \onto \Gal(K'/K)$, see for e.g. \cite[Sec.~2]{breuil2007representations}.
In this article, it will suffice to restrict attention to principal series tame types, for which we now introduce notation from \cite{lhmm}.

Denote by $W = \{\id, w_0\}$ the Weyl group of $\GL_2$ (defined over $\Z$). Here $w_0 = (1\; 2)$ is the longest element of the Weyl group. Let
$B \subset \GL_2$ be the Borel subgroup of upper triangular matrices and $T \subset B$ the subgroup of diagonal matrices. We identify the group of its characters $X^{*}(T)$ with $\Z^{2}$ in the standard way. Let $\alpha$ denote the positive root of $\GL_2$, and let $\< \; , \; \>: X^{*}(T) \times X_{*}(T) \to \Z$ be the duality pairing where $X_{*}(T)$ is the group of cocharacters of $T$. The Weyl group $W$ acts naturally on $X^{*}(T)$. We extend this to a coordinate--wise action of $W^{\Z/f\Z}$ on $X^{*}(T)^{\Z/f\Z}$. Define 
\begin{align*}
\widetilde{W} &\stackrel{\text{def}}{=} X^{*}(T) \rtimes W
   % \underline{\widetilde{W}} &\stackrel{\text{def}}{=} X^{*}(T)^{\Z/f\Z} \rtimes W^{\Z/f\Z} \cong \widetilde{W}^{\Z/f\Z}.
\end{align*}
Denote by $t_{\nu}$ the image of $\nu \in X^{*}(T)$ under the obvious inclusion $X^{*}(T) \hookrightarrow \widetilde{W}$. 
%Abusing notation, we denote by $w_0$ also the element of $W^{\Z/f\Z} \subset \underline{\widetilde{W}}$ that is $w_0$ in each coordinate.

Suppose  $\mu = (\mu_j) \in X^{*}(T)^{\Z/f\Z}$ and $s = (s_j)_{j} \in W^{\Z/f\Z}$ satisfying $s_0 s_1 \dots s_{f-1}=1$. Let $\alpha_0 = \mu_0$ and $\alpha_j = s_{f-1}^{-1} s_{f-2}^{-1} \dots s_{f-j}^{-1}(\mu_{f-j})$ for $j \in \Z/f\Z \smallsetminus \{0\}$.
For each $j \in \Z/f\Z$, let $$\mathbb{a}^{(j)} = \sum_{i = 0}^{f-1} \alpha_{-j + i} p^{i} \in X^{*}(T)$$  where $\alpha_{l}$ for $0 \leq l \leq f-1$ is to be interpreted as $\alpha_{l \text{ mod } f}$. Viewing $\mathbb{a}^{(j)}$ as a cocharacter of the dual torus, we set $$\tau(s, \mu) \stackrel{\text{def}}{=} \mathbb{a}^{(j)} \omega_j$$ for any $j \in \Z/f\Z$, since this definition does not depend on $j$.
By \cite[Lem.~2.1.6]{lhmm}, for any principal series $\tau$, there exist $\mu$ and $s$ such that $\<\mu_j, \alpha^{\vee}\> \in [0, (p+1)/2]$ for each $j$, $s_j$ is $\id$ whenever $\<\mu_j, \alpha^{\vee}\>=0$, and $\tau \cong \tau(s, \mu)$. By Lem.~2.1.8 in \textit{loc. cit.}, whenever $\tau$ is non--scalar, there exists a unique $(s_{\ort, j})_j \in W^{\Z/f\Z}$
such that $\< s_{\ort, j}^{-1} \left(\mathbb{a}^{(j)}\right), \alpha^{\vee}\> > 0$. 

\begin{remark}
    Since our application does not require cuspidal types, we don't include notation needed to describe them. However, we note that cuspidal types can also be described using the data of suitable $s$ and $\mu$ where $s_0 s_1 \dots s_{f-1}$ equals $w_0$ (see for e.g. Section 2 in \cite{lhmm}). Furthermore, everything in Sections \ref{sec:notation}, \ref{sec:charts} and \ref{sec:geometry} can be generalized to include cuspidal types as well.
\end{remark}

\begin{definition}
    Let $\mu = (\mu_j)_j \in X^{*}(T)^{\Z/f\Z}$. We say $\mu$ is \textit{small} if for each $j$, $\<\mu_j, \alpha^{\vee}\> \in [0, (p+1)/2]$, and $s_j = \id$ whenever $\<\mu_j, \alpha^{\vee}\> = 0$.
\end{definition}

\begin{definition}
    Define $\eta \in X^{*}(T)$ to be the element $(1,0)$. Abusing notation, we also let $\eta \in X^{*}(T)^{\Z/f\Z}$ be the element that is $(1,0)$ in each coordinate.
\end{definition}
\subsection{Breuil--Kisin modules}
Let $\gS \stackrel{\text{def}}{=} W(k)[[u]]$, where $W(k)$ is the ring of Witt vectors of $k$. The ring $\gS_{K'}$ is equipped with a Frobenius endomorphism $\varphi$ that extends the usual arithmetic Frobenius on $W(k)$ lifted from the $p^{f}$-power map on $k$, and maps $u$ to $u^{p}$. It also admits an action of $\Gal(K'/K)$ extending the usual trivial action of $\Gal(K'/K)$ on $W(k)$, so that if $g \in \Gal(K'/K)$, then $$g(u) = \dfrac{g(\pi')}{\pi'} u.$$ Let $E(u)$ denote the minimal polynomial of $\pi'$ over $W(k)$. The subring $\gS^{0}$ of $\Gal(K'/K)$--invariants of $\gS$ is $W((k))[[v]]$ where $$v := u^{p^f - 1}.$$

For a $\cO/\varpi^{a}$--algebra $A$ where $a \geq 1$, let  $\gS_{A} \stackrel{\text{def}}{=} (W(k) \otimes_{\Zp} A)[[u]]$ and equip it with $A$--linear actions of $\varphi$ and $\Gal(K'/K)$ extended naturally from the $\varphi$ and $\Gal(K'/K)$ actions on $\gS$. The subring $\gS_{A}^{0}$ of $\Gal(K'/K)$--invariants of $\gS_{A}$ is $(W(k) \otimes_{\Zp} A)[[v]]$.
 Let $\tau$ be a principal series type.
 %Letting $\mathbb{a}^{(j)} = (\mathbb{a}^{(j)}_0, \mathbb{a}^{(j)}_1)$, set $$\eta_1 = \omega_j^{-\mathbb{a}^{(j)}_0}, \text{ and } \quad \eta_2 = \omega_j^{-\mathbb{a}^{(j)}_1}.$$
\begin{definition}
    A Breuil--Kisin module $\gM$ of rank $2$ with $A$--coefficients and descent data of type $\tau$ is a rank $2$ projective $\gS_{A}$--module $\gM$ together with
    \begin{itemize}
    \item a $\varphi$--semilinear map $\varphi_{\gM}: \gM \to \gM$ whose linearization is an isomorphism after inverting $u$, and
    \item a semilinear action of $\Gal(K'/K)$ on $\gM$ commuting with $\varphi_{\gM}$ such that Zariski-locally on $\Spec A$
    $$\gM \otimes_{k, \sigma_j} A \mod u \cong \tau^{\vee} \otimes_{\cO} A$$ as $\Gal(K'/K)$--representations.
\end{itemize}
We say that $\gM$ has \textit{height at most $h$} if the cokernel of $\Phi_{\gM}$ is annihilated by $E(u)$.
\end{definition}
Let $\gM$ be a Breuil--Kisin module of rank $2$ with $A$--coefficients and descent data of type $\tau$. For each $\sigma_j: W(k) \to \cO$, there is a corresponding idempotent $\mathfrak{e}_j \in W(k) \otimes_{\Zp} \cO$ such that $x \otimes 1$ and $1 \otimes \sigma_j(x)$ have the same action on $\mathfrak{e}_j(W(k) \otimes_{\Zp} \cO)$, a rank $1$ $\cO$--module. Set $\gM_j = \mathfrak{e}_j \gM$, a module over $A[[u]]$, and let 
$$\Phi_{\gM, j}: \varphi^{*} (\gM_{j-1}) \to \gM_j$$ be the map induced by $\Phi_{\gM}$. Each $\gM_j$ is Zariski locally on $\Spec A$ free as an $A[[u]]$ module by \cite[Lem.~2.3]{bellovin2024irregular}. 

Suppose $\tau \cong \tau(s, \mu)$ is non--scalar. For $j \in \Z/f\Z$, let $\mathbb{a}^{(j)}_0, \mathbb{a}^{(j)}_1 \in \Z$ be such that $\mathbb{a}^{(j)} = (\mathbb{a}^{(j)}_0, \mathbb{a}^{(j)}_1)$. By the argument in \cite[Lem.~2.1.3]{gkksw} for e.g., Zariski locally on $\Spec A$ one can choose an ordered basis $\beta_j = (e_j, f_j)$ of $\gM_j$  so that $\Gal(K'/K)$ acts on $e_j, f_j$ via $$\omega_j^{-\mathbb{a}^{(j)}_0}, \omega_j^{-\mathbb{a}^{(j)}_1}$$ respectively. Following convention, we call a $\Z/f\Z$--tuple $\beta = (\beta_j)_j$ of such ordered bases an \textit{eigenbasis} of $\gM$. Given an eigenbasis $\beta$ of $\gM$, let $C_{\gM, \beta}^{(j)}$ be the matrix of $\Phi_{\gM, j}$ with respect to the bases $\varphi^{*} (\beta_{j-1})$ and $\beta_j$. For each $j \in \Z/f\Z$, set 
$$A_{\gM, \beta}^{(j)} \stackrel{\text{def}}{=} \Ad \left(s_{\ort, j}^{-1} u^{-\mathbb{a}^{(j)}}\right) (C_{\gM, \beta}^{(j)}).$$
Here, the notation $\Ad \; A (B)$ means $A B A^{-1}$ and if $\nu = (\nu_0, \nu_1) \in X^{*}(T)$ and $x \in \gS_{A}$, then $x^{\nu}$ is the diagonal matrix $\begin{psmallmatrix}
    x^{\nu_0} & 0 \\
    0 & x^{\nu_1}
\end{psmallmatrix}$.

\begin{proposition}\label{prop:change-of-basis-BKmod}\cite[Prop.~5.1.8]{LLLM-models}
    Let $\gM$ be a Breuil--Kisin module of rank $2$ with $A$--coefficients and descent data of principal series non--scalar type $\tau \cong \tau(s, \mu)$. Let $\beta_1, \beta_2$ be two eigenbases of $\gM$ related via
    $$\beta_{2,j} D^{(j)} = \beta_{1,j}$$ with $D^{(j)} \in \GL_2(A[[u]])$ for each $j \in \Z/f\Z$. Set $$I^{(j)} \stackrel{\text{def}}{=} \Ad \left(s_{\ort, j}^{-1} u^{-\mathbb{a}^{(j)}}\right) \left(D^{(j)}\right).$$ Then $I^{(j)} \in \GL_2(A[[v]])$ is upper triangular mod $v$, and
    $$A_{\gM, \beta_2}^{(j)} = I^{(j)} A_{\gM, \beta_1}^{(j)} \Ad \left(s_{j}^{-1} v^{\mu_j}\right) \left(\varphi\left(I^{(j-1)}\right)\right)^{-1}.$$
    Furthermore, if $I^{(j)} \in \GL_2(A[[v]])$ upper triangular mod $v$ for each $j \in \Z/f\Z$, then $\Ad \left(u^{\mathbb{a}^{(j)}} s_{\ort, j} \right)\left(I^{(j)}\right) = D^{(j)} \in \GL_2(A[[u]])$ and for any eigenbasis $\beta$, $\left(\beta^{(j)} D^{(j)}\right)_{j}$ is also an eigenbasis.
    \end{proposition}

\begin{definition}
    Following \cite{lhmm}, let  $\Y$ be the \textit{fppf} stack over 
$\Spf(\cO)$ that assigns to an $\cO/\varpi^{a}$--algebra $A$ the groupoid of Breuil--Kisin modules 
of rank $2$ with $A$--coefficients, descent data of type $\tau$ and height at most $1$ satisfying the additional \textit{determinant condition}
$$\det(\Phi_{\gM}) \in vA[[v]]^{\times}.$$
\end{definition}
 Let  $\YF$ denote the special fiber of this stack.
\begin{remark}\label{rmk:compare-stacks}By \cite[Prop.~2.7]{bellovin2024irregular} and \cite[Cor.~4.5.3(2)]{cegsB}, $$\YF=\cC^{\tau^{\vee}, \text{BT}, 1}$$ where the right hand side is the stack studied in \cite{cegsC}. 
\end{remark}
We have the following description of the irreducible components of $\YF$.

\begin{theorem}\label{thm:irred-comp-Y}
There exists a bijective correspondence between $\{L, R\}^{\Z/f\Z}$ and the set of irreducible components of $\YF$ given in the following way: If $$S = (S_j)_j \in \{L, R\}^{\Z/f\Z},$$ then the corresponding irreducible component $\Y_S$ is the closed substack of $\YF$ obtained by imposing the condition that if $S_j = L$ (resp. $S_j = R$), then $\gM \in \Y_S(A)$ for a local $\F$--algebra $A$ if and only if $v$ divides the top left (resp. bottom right) entry of $A_{\gM, \beta}^{(j)}$ for some, equivalently any, choice of eigenbasis $\beta$ of $\gM$.
\end{theorem}
\begin{proof}
    Immediate from \cite[Thm.~3.16]{bellovin2024irregular}. 
\end{proof}

\subsection{\'Etale \texorpdfstring{$\varphi$}{phi}-modules and Galois representations}
Let $A$ be an $\cO/\varpi^{a}$--algebra for some $a \geq 1$. 
\begin{definition}
%  An \'etale $\varphi$--module $\cM'$ of rank $2$ with $A$--coefficients and descent data is a rank $2$ projective module over $\gS_{A}[1/u]$, together with 
% \begin{itemize}
%     \item a $\varphi$--semilinear map $\varphi_{\cM'}: \cM' \to \cM'$ whose linearization is an isomorphism, and
%     \item a semilinear action of $\Gal(K'/K)$ commuting $\varphi_{\cM'}$.
% \end{itemize}

   An \'etale $\varphi$--module $\cM$ of rank $2$ with $A$--coefficients is a rank $2$ projective module over $\gS_A^{0}[1/v]$, together with a $\varphi$--semilinear map $\varphi_{\cM}: \cM \to \cM$ whose
    linearization $\Phi_{\cM}: \varphi^{*} \cM \to \cM$ is an isomorphism.
\end{definition}
%One can pass from an \etale $\varphi$--module with descent data and finite type coefficients to one without descent data by taking $\Gal(K'/K)$--invariants and vice versa by sending $\cM$ to $\cM \otimes_{\gS^{0}[1/v]} \gS[1/u]$. 
%The composite of the two maps is identity in either direction, by \cite[Thm.~2.4.1,2.7.8]{emerton2022moduli}.
Let $\cM$ be an \etale $\varphi$--module of rank 2 with $A$--coefficients. As in the setting of Breuil--Kisin modules, we can decompose $\cM \cong \oplus_{j \in \Z/f\Z} \cM_j$ by setting $\cM_j := \mathfrak{e}_j \cM$. The map $\Phi_{\cM}$ induces maps $\Phi_{\cM, j}: \varphi^{*}\cM_{j-1} \to \cM_j$. 

\begin{definition}
    Let $\Phi\text{-Mod}^{\et, 2}_{K}$ denote the \textit{fppf} stack over $\Spf(\cO)$ that assigns to an $\cO/\varpi^{a}$--algebra $A$ the groupoid of rank $2$ \etale $\varphi$--modules of rank $2$ with $A$ coefficients. 
\end{definition}

Define a map 
\begin{align*}
\varepsilon_{\tau}: \Y \to \Phi\text{-Mod}^{\et, 2}_{K}
\end{align*}
by setting $\varepsilon_{\tau}(\gM) = \gM[1/u]^{\Gal(K'/K)}$.

\begin{remark}
    The map $\varepsilon_{\tau}$ is proper by \cite[Thm.~5.1.2]{cegsB} and by \cite[Thm.~4.5]{bellovin2024irregular}, the scheme--theoretic image $\cZ^{\tau}$ of $\varepsilon_{\tau}$ is isomorphic to the Emerton--Gee stack of potentially Barsotti--Tate representations of type $\tau^{\vee}$, described in \cite[Defn.~4.8.8]{emerton2022moduli}. Denote this locus by $\cX^{\tau^{\vee}, \mathrm{BT}}$.
    The scheme--theoretic image of $\YF$ is the reduced stack $\cZ^{\tau, 1}$. Denote by $\pi$ the induced map $\YF \to \cZ^{\tau, 1}$.
\end{remark}

\begin{definition}
    For $S \in \{L, R\}^{\Z/f\Z}$, denote by $\cZ^{\tau}_S$ the scheme--theoretic image of $\Y_S$ under $\pi$.
\end{definition}

\begin{proposition}\label{prop:phi-module-frob}\cite[Prop.~5.4.2]{LLLM-models} Let $\gM \in \Y(A)$ and $\beta$ an eigenbasis of $\gM$. Let $\tau = \tau(s, \mu)$. Then there exists a basis $\mathfrak{b}$ for $\varepsilon_{\tau}(\gM)$ such that the matrix of $\Phi_{\cM, j}$ with respect to $\mathfrak{b}$ is given by $A_{\gM, \beta}^{(j)}s_j^{-1} v^{\mu_j}$.
    
\end{proposition}

Finally, we describe how to assign a Galois representation to an \etale $\varphi$--module. Fix a compatible sequence $\{\pi_n\}_n$ of $p^{n}$-th roots of $p$ in $\overline{K}$, with $\pi_{n+1}^{p} = \pi_n$. Since $\gcd(e(K'/K), p)=1$, $\{\pi_n\}_{n=0}^{\infty}$ determines a compatible sequence $\{\pi'_n\}_{n=0}^{\infty}$ of $p^{n}$-th roots of $\pi'$ satisfying $(\pi'_n)^{e(K'/K)} = \pi_n$. Let $$K_{\infty} := \cup_n K(\pi_n), \text{ and} \quad K'_{\infty} := \cup_n K'(\pi'_n).$$ 
By Fontaine's theory of the field of norms, if $|A| < \infty$, then there exists a fully faithful functor $T$ from the \etale $\varphi$--modules with $A$--coefficients to $G_{K_{\infty}}$--representations with $A$--coefficients. To describe this functor, we first define $R := \lim_{x \mapsto x^{p}} \cO_{\overline{K}}/p$. Let $\underline{\pi'} = (\pi'_n)_n \in R$ and let $[\underline{\pi'}]$ be the canonical multiplicative lift of $\underline{\pi'}$ to the Witt vectors of $R$, $W(R)$. There exists a $\varphi$--equivariant inclusion $\gS \hookrightarrow W(R)$ over $W(k)$ that maps $u \to [\underline{\pi'}]$. This embedding extends to inclusions $$\cO_{\cE} \hookrightarrow W(\text{Frac}(R))$$ and $$\cE \hookrightarrow W(\text{Frac}(R))[1/p],$$ where $\cO_{\cE}$ is the $p$--adic completion of $\gS[1/u]$ and $\cE$ is its ring of fractions. The ring $\cO_{\cE}$ is a discrete valuation ring with uniformizer $p$ and residue field $k((u))$. Let $\cE^{\nr}$ be the maximal unramified extension of $\cE$ in $W(\text{Frac}(R))[1/p]$ with ring of integers $\cO_{\cE^{\nr}}$ and residue field $k((u))^{\sep}$, a separable closure of $k((u))$. Let $\cO_{\widehat{\cE^{\nr}}}$ be the $p$--adic completion of $\cO_{\cE^{\nr}}$.

\begin{definition}
    Let $|A| < \infty$. If $\cM \in \Phi\text{-Mod}^{\et, 2}_{K}(A)$, set
$$T(\cM) \stackrel{\text{def}}{=} \left(\cO_{\widehat{\cE^{\nr}}} \otimes_{\gS^{0}[1/v]} \cM \right)^{\varphi=1}$$ equipped with diagonal action of the group $G_{K_{\infty}}$.
We also define $T(\gM)$ for $\gM$ a Breuil--Kisin module with $A$--coefficients and descent data by setting 
$$T(\gM) \stackrel{\text{def}}{=} T\left(\varepsilon_{\tau}(\gM)\right).$$ 
\end{definition}
Using \cite[Thm.~2.4.1, 2.7.8]{emerton2022moduli} for e.g., we note that the right hand side above equals $$\left(\cO_{\widehat{\cE^{\nr}}} \otimes_{\gS[1/u]} \gM[1/u]\right)^{\varphi=1},$$ which agrees with the definition of $T(\gM)$ in \cite[Defn.~2.2.3]{cegsC}.

\subsection{Serre weights and the Emerton--Gee stack}\label{subsec:serre-wt}
A Serre weight is an isomorphism class of an irreducible $\F$--representation of $\GL_2(k)$. Such representations are precisely those of the form
$$\sigma_{\mathbb{m}, \mathbb{n}} := \bigotimes_{j \in \Z/f\Z} \left(\mathrm{det}^{m_j} \otimes \Sym^{n_j} k^{2} \right)\otimes_{k, \sigma_j} \F$$
where $k^{2}$ denotes the standard two--dimensional representation of $\GL_2(k)$ and $0 \leq n_j \leq p-1$ for each $j$. The representation is non--Steinberg if for some $j$, $n_j < p-1$.

If $\sigma_{\mathbb{m}, \mathbb{n}}$ is a Serre weight,
the finite type points of the corresponding irreducible component $\cX(\sigma_{\mathbb{m}, \mathbb{n}})$ of $\cX_{2, \text{red}}$ admit crystalline lifts with Hodge--Tate weights $\{-n_j-m_j, -m_j+1\}$ in the $j$-th embedding. When $\sigma_{\mathbb{m}, \mathbb{n}}$ is non--Steinberg, the main result of \cite{cegsA} shows that this description uniquely characterizes $\cX(\sigma_{\mathbb{m}, \mathbb{n}})$.

% by t If $\sigma_{\mathbb{m}, \mathbb{n}}$ is Steinberg, then each finite type point of $\cX(\sigma_{\mathbb{m}, \mathbb{n}})$ admits a crystalline lift with Hodge--Tate weights $\{-n_j-m_j, -m_j+1\}$ in the $j$-th embedding. 

Here, we are normalizing Hodge--Tate weights so that all Hodge--Tate weights of the cyclotomic character are equal to $-1$. 

% Extend this definition to Breuil--Kisin modules with descent data by defining $T(\gM)$ to be $T(\gM[1/u])$.

% As noted in \cite[Sec.~2.3]{bellovin2023irregular}, if $\cM'$ is an \etale $\varphi$--module with descent data and $A$--coefficients, then $\cM := (\cM')^{\Gal(K'/K)} \in \Phi\text{-Mod}^{\et, 2}_{K}(A)$ and $T(\cM') = T^{0}(\cM)$. 

\section{smooth--local charts}\label{sec:charts}
\subsection{Loop groups and torsors over \texorpdfstring{$\YF$}{Y-eta-tau} and \texorpdfstring{$\cZ^{\tau, 1}$}{Z-tau-1}}

From now onwards, we will work entirely over $\F$--coefficients, although as described in detail in \cite{lhmm}, many of the descriptions below extend to $\cO$--coefficients. Our primary objective in this subsection is to construct smooth--local affine charts on $\cZ^{\tau, 1}$ following closely various constructions in \cite[Sec.~3]{lhmm}.

Let $\tau = \tau(s,\mu)$ be a fixed non--scalar principal series tame type with $\mu = (\mu_j)_j$ small. Assume that $p - 2 > \max_j \<\mu_j, \alpha^{\vee}\> $. This is true if $p>3$ and $\max_j \<\mu_j, \alpha^{\vee}\> \leq (p-1)/2$.
Let $\mu_j = (\mu_{j, 0}, \mu_{j, 1})$ for each $j$. Define a functor $\LG$ by setting for an $\F$--algebra $A$ $$\LG(A) := \GL_2(A((v)))$$ as well as various subfunctors
\begin{align*}
    \LGp(A)&:= &&\GL_2(A[[v]]) \\
    \LGn(A) &:= &&\GL_2(A[1/v]) \\
    \Iw(A) &:= &&\{P \in \LGp(A) \mid P \text{ is upper triangular mod } v\} \\
    \Kpr(A) &:= &&\{P \in \Iw(A) \mid P \equiv \id \text{ mod } v^{r} \} \\
    \Knr(A) &:= &&\{P \in \LGn(A) \mid P \equiv \id \text{ mod } 1/v^{r} \} \\
    \cA(\eta)(A) &:= &&\{P \in \Mat_2(A[[v]]) \mid P \text{ is upper triangular mod } v, \det P \in vA[[v]]^{\times}\}
\end{align*}
where $r$ is any positive integer in the definition of $\Kpr$ and $\Knr$.
Define $$\LGbd(A) \subset \LG(A)^{\Z/f\Z}$$ to be the set of $(P_j)_j \in \LG(A)^{\Z/f\Z}$ satisfying, for each $j$,
\begin{itemize}
    \item $P_j \in v^{\mu_{j, 1}} \Mat_2(A[[v]])$, and
    \item $\det P_j \in v^{\mu_{j, 0} + \mu_{j, 1} + 1} A[[v]]^{\times}$.
\end{itemize}
\begin{remark}
    We are allowing nonzero values of $\mu_{j, 1}$, and so, this definition of $\LGbd$ is slightly different from the mod $p$ version of the one in \cite{lhmm}. 
\end{remark}
Define 
$$\LGtau(A) := \{(W_j s_j^{-1}v^{\mu_j})_j \in \LG(A)^{\Z/f\Z} \; | W_j \in \cA(\eta)(A) \text{ for each } j\} \subset \LGbd.$$

\begin{definition}
    We define \textit{shifted $\varphi$--conjugation} to be a left action of $\LG^{\Z/f\Z}$ (and its various subfunctors) on itself denoted by $\cdot_{\varphi}$ and given by setting
$$(P_j)_j \cdot_{\varphi} (Q_j)_j \stackrel{\text{def}}{=} \left(P_j Q_j \varphi(P_{j-1})^{-1}\right).$$
\end{definition}

\begin{lemma}\cite[Lem.~3.2.1]{lhmm}\label{lem:identification-Gr}
    Let $\gM \in \YF(A)$ for an an $\F$--algebra $A$. Then, Zariski-locally on $\Spec A$, $\gM$ has an eigenbasis $\beta$. The assignment $$\gM \mapsto \left(A_{\gM, \beta}^{(j)} s_j^{-1}v^{\mu_j}\right)_{j \in \Z/f\Z}$$ defines an isomorphism of algebraic stacks over $\F$
    \begin{align}\label{eqn:presentation-Y}
    \YF \xrightarrow{\sim} \left[\LGtau \Big{/}_{\varphi} \left(\Iw^{\Z/f\Z}\right)\right] \end{align} and hence a morphism
    \begin{align}\label{eqn:map-to-bd}
        \YF \to \left[\LGbd \Big{/}_{\varphi} \left(\LGp^{\Z/f\Z}\right)\right]. 
    \end{align}
\end{lemma}

Define a morphism
\begin{align}\label{eqn:map-to-phimod}
    \left[\LGbd \Big{/}_{\varphi} \left(\LGp^{\Z/f\Z}\right)\right] \to \Phi\text{-Mod}^{\et, 2}_{K}
\end{align}
by sending the class of $P = (P_j)_j$ to the \etale $\varphi$--module $\iota(P)$ which is free of rank $2$ and such that $\Phi_{\iota(P), j}: \varphi^{*} \left(\iota(P)_{j-1}\right) \to \iota(P)_{j}$ has matrix $P_j$ in the standard basis.

\begin{prop}\cite[Prop.~3.2.4]{lhmm}\label{prop:factorization-through-bd}
    The map (\ref{eqn:map-to-phimod}) is a closed immersion, and the map $\varepsilon_{\tau}|_{\YF}$ factors as
    \begin{align}\label{eqn:factor-I}
    \YF \xrightarrow{\pi} \cZ^{\tau, 1} \hookrightarrow \left[\LGbd \Big{/}_{\varphi} \left(\LGp^{\Z/f\Z}\right)\right] \xhookrightarrow{\text{(\ref{eqn:map-to-phimod})}} \Phi\text{-}\mathrm{Mod}^{\et, 2}_{K}
    \end{align}
    where the composite of the first two arrows is the map (\ref{eqn:map-to-bd}). 
    % $$\YF \xrightarrow{\text{\Cref{eqn:map-to-bd}}} \left[\LGbd \Big{/}_{\varphi} \left(\LGp^{\Z/f\Z}\right)\right] \xhookrightarrow{\text{\Cref{eqn:map-to-phimod}}} \Phi\text{-}\mathrm{Mod}^{\et, 2}_{K}.$$
    % In particular, the map \Cref{eqn:map-to-bd} factors as $$\YF \to \cZ^{\tau, 1} \hookrightarrow \left[\LGbd \Big{/}_{\varphi} \left(\LGp^{\Z/f\Z}\right)\right]$$
    % where the first arrow is the map induced by $\varepsilon_{\tau}$ and the second arrow is a closed immersion.
\end{prop}

\begin{definition}
    Define a quotient sheaf
\begin{align*}
    \Grbd := \quad &\left[\left(\KpI^{\Z/f\Z}\right) \Big{\backslash} \LGbd \right] \end{align*} and its closed subsheaf
\begin{align*}
\Grtau := \quad &\left[\left(\KpI^{\Z/f\Z}\right) \Big{\backslash} \LGtau \right],
\end{align*}
where the quotient is for action by left multiplication.
\end{definition}
\begin{remark}
    The sheaves $\Grbd$ and $\Grtau$ are represented by finite type schemes. Indeed, the sheaf $\Grbd$ is a torsor for an affine scheme over the quotient $$\left[ \left(\LGp^{\Z/f\Z}\right) \Big{\backslash} \LGbd \right], $$ which in turn is a closed subscheme of a finite type scheme by the argument in \cite[Lem.~1.1.5]{zhu2016introduction}. Therefore, by \cite[\href{https://stacks.math.columbia.edu/tag/0245}{Tag 0245}]{stacks-project}, $\Grbd$ is representable.
\end{remark}

\begin{lemma}\label{lem:straighten}
    There exists an isomorphism $$\Grbd \cong \left[\LGbd \Big{/}_{\varphi} \left(\KpI^{\Z/f\Z}\right)\right]$$ which induces an isomorphism
    $$Y^{\eta, \tau}_{\F} \cong \left[ \Grtau \Big{/}_{\varphi} \left(B^{\Z/f\Z}\right) \right]$$
\end{lemma}
\begin{proof}
    By \cite[Lem.~3.3.7]{lhmm}.
\end{proof}

\begin{remark}
     The proofs of Proposition \ref{prop:factorization-through-bd} and Lemma \ref{lem:straighten}  critically use the assumption that $p-2 > \max_j \<\mu_j, \alpha^{\vee}\>$.
\end{remark}

\begin{remark}
    If $(P_j)_j$ represents a point of $\Grbd$ and $(g_j)_{j}$ of
$$\GL_2^{\Z/f\Z} \cong \left(\LGp/\KpI\right)^{\Z/f\Z},$$ then we have $$(g_j)_j \cdot_{\varphi} (P_j)_j = (g_j P_{j} g_{j-1}^{-1}).$$ Therefore, $\varphi$--action induces simply a \textit{shifted conjugation} action of $\GL_2^{\Z/f\Z}$ and its subgroups on $\Grbd$.
\end{remark}

\begin{lemma}\label{lem:smooth-irred}
    The stack $\Y_S$ is smooth.
\end{lemma}
\begin{proof}
 The inclusion $$\Ad \begin{pmatrix}
     1 &  \\
            & v
 \end{pmatrix} \left(\KpII\right) \subset \KpI$$ allows the construction of a map
\begin{align}\label{eqn:smoothcover-step-YS}
\prod_{j \in \Z/f\Z} \left(\KpII \backslash \LGp\right) \to \Grtau
\end{align}
given in the following way: If $(A_j)_j \in \LGp^{\Z/f\Z}$ 
represents an object on the left, then its image is the object represented by $(B_j)_j \in \LGtau$ where 
\begin{align*}
     B_j = 
     \begin{cases}
        A_j  \begin{pmatrix}
            v &  \\
            & 1
        \end{pmatrix}  &\text{ if } S_j = L, \vspace{0.2cm}\\
        \begin{pmatrix}
            1 &  \\
            & v
        \end{pmatrix} A_j  &\text{ if } S_j = R.
     \end{cases}
 \end{align*}

%Suppose $\gM$ is a point of $\Y_S$ admitting an eigenbasis $\beta$.
 By Theorem \ref{thm:irred-comp-Y} and \cite[(3.14)]{bellovin2024irregular}, one can specify a point $\gM$ of $\Y_S$ with an eigenbasis $\beta$ by specifying
 \begin{align*}
     A_{\gM, \beta}^{(j)} \in 
     \begin{cases}
        \LGp  \begin{pmatrix}
            v &  \\
            & 1
        \end{pmatrix}  &\text{ if } S_j = L, \vspace{0.2cm}\\
        \begin{pmatrix}
            1 &  \\
            & v
        \end{pmatrix} \LGp  &\text{ if } S_j = R.
     \end{cases}
 \end{align*}
Therefore, there exists a map 
\begin{align}\label{eqn:smoothcover-YS}
    \prod_{j \in \Z/f\Z} \left(\KpII \backslash \LGp\right) \to \Y_S
\end{align} fitting into a commutative diagram
\begin{equation*}
\begin{tikzcd}
    \prod_{j \in \Z/f\Z} \left(\KpII \backslash \LGp\right) \arrow[r, "\text{(\ref{eqn:smoothcover-step-YS})}"] \arrow[d, "\text{(\ref{eqn:smoothcover-YS})}"]&\Grtau \arrow[d] \\
    \Y_S \arrow[r, hook, "\text{closed}"] & \Y_{\F}    
\end{tikzcd}
\end{equation*}
where the right vertical arrow is induced from the second isomorphism in Lemma \ref{lem:straighten}. The top horizontal and right vertical arrows are representable, and so the same is true for the map in (\ref{eqn:smoothcover-YS}). Since an eigenbasis always exists Zariski--locally, the map in (\ref{eqn:smoothcover-YS}) is surjective on points valued in local rings and therefore, is formally smooth. In particular, it is a smooth surjective map with smooth domain, and so, the codomain is also smooth.
\end{proof}

Let $\tilz = (\tilz_j)_j \in \widetilde{W}^{\Z/f\Z}$. As described in \cite[Sec.~3.3]{lhmm}, there exists an open immersion 
$$\widetilde{U}(\tilz) := \left[ \prod_{j \in \Z/f\Z}\LGn \; \tilz_j \right] \cap \Grbd \hookrightarrow \Grbd.$$ 
% of schemes and we have an equality
% \begin{align}\label{equality-affine-open}
%     \bigcup_{\tilz \in \widetilde{W}^{\Z/f\Z}} \widetilde{U}(\tilz) = \Grbd.
% \end{align}

The scheme $\widetilde{U}(\tilz)$ has a natural structure as a product of schemes $\prod_j \widetilde{U}(\tilz_j)$. 

Next, we construct the following commutative diagram for $S \in \{L, R\}^{\Z/f\Z}$ and $\tilz \in \widetilde{W}^{\Z/f\Z}$:

\begin{equation}\label{diagram:define-covers}
\begin{tikzcd}[row sep = 4ex, column sep=1.7ex]
    \tilYz_S \arrow[rr] \arrow[dd, open] \arrow[dr, hook, "\mathit{cl.}"] & &\tilZz_S \arrow[hook, open]{dd} \arrow[dr, hook, "\mathit{cl.}"]  \\
     &\widetilde{Y}^{\eta, \tau}(\tilz) \arrow[rr]  \arrow[dd, open]  & & \widetilde{\cZ}^{\tau, 1}(\tilz) \arrow[dd, open] \arrow[r, hook, "\mathit{cl.}"] \arrow[ddr, phantom, "\square"] & \widetilde{U}(\tilz) \arrow[dd, open] \\
    \tilY_S \arrow[rr, "\widetilde{\pi}_S" {xshift=15pt}] \arrow[dd] \arrow[dr, hook, "\mathit{cl.}"] & &\tilZ_S \arrow[dd] \arrow[dr, hook, "\mathit{cl.}"]  \\
     &\widetilde{Y}^{\eta, \tau} \arrow[rr,"\widetilde{\pi}" {xshift=-11pt}]  \arrow[dd]  & & \widetilde{\cZ}^{\tau, 1} \arrow[dd] \arrow[r, hook, "\mathit{cl.}"] \arrow[ddr, phantom, "\square"]  & \Grbd \arrow[dd] \\
     \Y_S \arrow[rr, "{\pi}_S" {xshift=15pt}] \arrow[dr, hook, "\mathit{cl.}"]  & &\cZ^{\tau}_S \arrow[dr, hook, "\mathit{cl.}"]  \\
     &\YF \arrow{rr}{\pi} & & \cZ^{\tau, 1} \arrow[r, hook, "\mathit{cl.}"] & \left[\LGbd \Big{/}_{\varphi} \left(\LGp^{\Z/f\Z}\right)\right]
\end{tikzcd}
\end{equation}
where 
\begin{itemize}
    \item the bottom two arrows are respectively the first two arrows in (\ref{eqn:factor-I});
    \item the vertical arrow in the bottom right 
    $$\Grbd \to \left[\LGbd \Big{/}_{\varphi} \left(\LGp^{\Z/f\Z}\right)\right]$$ is the composition of the first isomorphism in Lemma \ref{lem:straighten} with the quotient under shifted conjugation by $\GL_2^{\Z/f\Z}$;
    \item all hooked arrows annotated with \textit{cl.} are closed immersions, while all those marked with a circle are open immersions;
    \item the schemes $\widetilde{\cZ}^{\tau, 1}$ and $\widetilde{\cZ}^{\tau, 1}$ are defined so that the right most squares (marked with a square symbol in the center) are pullback squares; and
    \item the stacks  $\tilYz_S$, $\tilYz$, $\widetilde{Y}^{\eta, \tau}_S$ and $\widetilde{Y}^{\eta, \tau}$, and the schemes $\tilZz_S$ 
     and $\widetilde{\cZ}^{\tau}_S$, are defined so that the front, back and side faces (but not necessarily the top and bottom faces!) of the two cubes are pullback squares.
\end{itemize}

\begin{remark}
We make the following note for translation between the paper \cite{lhmm} and this article. In the paper \cite{lhmm}, there are various pairs of different but related notions that are identified mod $p$, namely, $Y^{\mathrm{mod}, \eta, \tau}$ and $\Y$, $\widetilde{Y}^{\mathrm{mod}, \eta, \tau}$ and $\tilY$, and $\widetilde{Z}^{\mathrm{mod}, \tau}$ and $\tilZ$. Since we are working entirely mod $p$, we do not define $Y^{\mathrm{mod}}$ ,  $\widetilde{Y}^{\mathrm{mod},\eta, \tau}$ and $\widetilde{\cZ}^{\mathrm{mod}, \tau}$ at all.
\end{remark}

\begin{prop}\label{prop:irred-Gr}
    \begin{enumerate}
        \item The stack $\tilY$ is a scheme identifying with the closed subscheme of $$\Grbd \times \left(B \backslash \GL_2\right)^{\Z/f\Z}$$ consisting of pairs $\left((P_j)_j, (l_j)_j\right)$ such that if $\till_j$ is a lift of $l_j$ to $\GL_2$ and $\widetilde{P}_j$ a lift of $P_j$ to $\LG$, then $$\left(\till_j \widetilde{P}_j \till_{j-1}^{-1}\right)_{j} \in \LGtau.$$ Equivalently, for each $j$, $$\till_j \widetilde{P}_j \till_{j-1}^{-1} v^{-\mu_j}s_j \in \cA(\eta).$$
        \item For $\tilz = (\tilz_j)_j \in \widetilde{W}^{\Z/f\Z}$, the open subscheme $\tilYz \subset \tilY$ identifies with the open subscheme consisting of those pairs $\left((P_j)_j, (l_j)_j\right)$ which, in addition to satisfying (1) above, have the property that for each $j$, $P_j$ is represented (uniquely) by a matrix of the form \begin{align*}
            \kappa_j X_j \tilz_j
        \end{align*} where $\kappa_j$ is a point of $\GL_2$ and $X_j$ of $\Kn$.
        \item For $S = (S_j)_j \in \{L, R\}^{\Z/f\Z}$, the closed subscheme $\tilYz_S \subset \tilYz$ further identifies with the subscheme consisting of those pairs $$\left((P_j)_j, (l_j)_j\right)$$ which, in addition to satisfying (2) above, have the property that if $\till_j$ is any lift of $l_j$ to $\GL_2$ and $\widetilde{P}_j$ is a matrix lifting $P_j$, then $$\widetilde{l}_j \widetilde{P}_j \till_{j-1}^{-1} v^{-\mu_j}s_j$$ has top left (resp. bottom right) entry divisible by $v$ if $S_j = L$ (resp. $S_j = R$) for each $j$.
    \end{enumerate}
   
\end{prop}
\begin{proof}
    The first part of the statement is proven in \cite[Prop.~3.3.1]{lhmm}. The second part follows from the definition of $\widetilde{U}(\tilz)$. The third part is immediate from Theorem \ref{thm:irred-comp-Y}.
\end{proof}

\begin{lemma}
The proper maps $\widetilde{\pi}$ and $\widetilde{\pi}_S$ are scheme--theoretically dominant.
\end{lemma}
\begin{proof}
    Being smooth torsors over the reduced stacks $\cZ^{\tau, 1}$ and $\widetilde{\cZ}^{\tau}_S$ respectively, $\widetilde{\cZ}^{\tau, 1}$ and $\tilZ_S$ are reduced algebraic stacks. Being pullbacks of $\pi$ and $\pi_S$ respectively, the maps $\widetilde{\pi}$ and $\widetilde{\pi}_S$ are proper and surjective. The universal property of scheme--theoretic images finishes the proof.
\end{proof}

\begin{lemma}
    A Zariski open cover of $\tilY$ (resp. $\tilZ$) is given by the set of all $\tilYz$ (resp. $\tilZz$) such that $\tilz = (\tilz_j)_{j\in \Z/f\Z}$ with $\tilz_j = \tilw_j s_j^{-1}v^{\mu_j}$ where $\tilw_j \in \{w_0 t_{\eta}, t_{w_0 (\eta)}\}$ for each $j$.
\end{lemma}
\begin{proof}
    The statement is true if $\tilw_j \in \{t_{\eta}, w_0 t_{\eta}, t_{w_0, \eta}\}$ for each $j$ by \cite[Lem.~3.3.5]{lhmm}. We claim that $$\widetilde{U}(t_{\eta}s_j^{-1} v^{\mu_j}) = \widetilde{U}(w_0 t_{\eta}s_j^{-1} v^{\mu_j}).$$ Indeed, this is immediate because $\LGn t_{\eta} s_j^{-1} v^{\mu_j} = \LGn w_0 t_{\eta} s_j^{-1} v^{\mu_j}$.
\end{proof}

\subsection{Auxiliary schemes}
Next, we construct certain auxiliary schemes through which the map $\tilY_S \to \tilZ_S$ factors and that will make it easier to study the geometry of $\tilZ_S$ later. The constructions in this section closely follow those in \cite[Sec.~4.1]{lhmm} with some variants that allow a detailed study of various irreducible components of $\tilZ$.

We fix $\tau = \tau(s, \mu)$ non--scalar principal series tame type with $\mu = (\mu_j)_j$ small; $\tilz = (\tilz_j)_j \in \widetilde{W}^{\Z/f\Z}$ with each $\tilz_j = \tilw_j s_j^{-1}v^{\mu_j}$ for some $\tilw_j \in \{w_0 t_{\eta}, t_{w_0 (\eta)}\}$; and $S = (S_j)_j \in \{L, R\}^{\Z/f\Z}$. Assume $p-2 > \max_j \<\mu_j, \alpha^{\vee} \>$.
\begin{definition}
    Define $\tilBaj$ to be the closed $\F$--subscheme of $\PP^{1} \times \GL_2 \times \Kn \times \PP^{1}$ parameterizing tuples $$(l_j, \kappa_j, X_j, r_j) \in \PP^{1} \times \GL_2 \times \Kn \times \PP^{1}$$ such that  if $\till_j, \tilr_j$ are lifts to $\GL_2$ of $l_j, r_j \in \PP^{1}$ respectively (under the obvious map $\GL_2 \to B\backslash \GL_2 \cong \PP^{1}$) and \begin{align}\label{eqn:W_j}
    W_j \stackrel{\text{def}}{=} \widetilde{l}_j \kappa_j X_j \tilw_j s_j^{-1} v^{\mu_j} \widetilde{r}_j^{-1} v^{-\mu_j} s_j,
\end{align}then $W_j \in \cA(\eta)$. 

Define a closed subscheme $\tilBaj_{S_j}$ of $\tilBaj$ by further requiring that the top left entry of $W_j$ is divisible by $v$ if $S_j = L$ and the bottom right entry is divisible by $v$ if $S_j=R$.
Since $\mu_j$ is dominant, these criteria are independent of the choice of lifts $\till_j$ and $\tilr_j$.
\end{definition}

\begin{definition}
  Define $\Baj$ to be the closed subscheme of $\tilBaj$ obtained by setting $\kappa_j = 1$. The scheme $\Baj$ naturally admits a monomorphism to $\PP^{1} \times \Kn \times \PP^{1}$.
  Define a closed subscheme $\Baj_{S_j}$ of $\Baj$ as the fiber product $$\tilBaj_{S_j} \times_{\tilBaj} \Baj.$$ 
\end{definition}

We define a map
\begin{align*}
    \widetilde{\pr}_j: \tilBaj \to \widetilde{U}(\tilz_j)
\end{align*}
by mapping $(l_j, \kappa_j, X_j, r_j) \mapsto \kappa_j X_j \tilz_j$. Denote by $\pr_j$ the restriction of $\widetilde{\pr}_j$ to $\Baj$. Denote by $p_j, q_j: \Baj \to \PP^{1}$ the obvious projections to the first and last coordinates, respectively.

\begin{lemma}\label{lem:projective-pr}
    The maps $\widetilde{\pr}_j$ and $\pr_j$ are projective.
\end{lemma}
\begin{proof}
It suffices to prove the statement for $\widetilde{\pr_j}$.
The map $\widetilde{\pr}_j$ sits in the following commutative diagram
\begin{equation}\label{diag:projective-prj}
    \begin{tikzcd}[row sep = 4ex, column sep=4ex]
     &\tilBaj \arrow[r, hook, "\mathit{cl.}"] \arrow[d, "\widetilde{\pr}_j"] &\PP^{1} \times \GL_2 \times \Kn \times \PP^{1} \arrow{d} \\
    &\widetilde{U}(\tilz_j) \arrow[hook]{r} &\LGn 
\end{tikzcd}
\end{equation}
where the bottom arrow is the map sending $(P\tilz_j) \mapsto P$ and is a monomorphism; the rightmost projective surjection is given by mapping $(l_j, \kappa_j, X_j, r_j) \mapsto \kappa_j X_j$; and the top horizontal arrow is a closed immersion. Therefore, the composition of the top horizontal arrow with the right vertical arrow is projective, and so is $\widetilde{\pr}_j$ by cancellation. 
\end{proof}
Define an isomorphism
    \begin{align}\label{eqn:Baj-decomp-std}
        \tilBaj &\xrightarrow{\sim} \Baj \times \GL_2 \\
    (l_j, \kappa_j, X_j, r_j) &\mapsto (l_j\kappa_j, X_j, r_j), \kappa_j. \nonumber
    \end{align} Here, $l_j \kappa_j$ is to be understood as the image of $\till_j \kappa_j$ in $B \backslash \GL_2 \cong \PP^{1}$ for some lift $\till_j \in \GL_2$ of $l_j$. The isomorphism in (\ref{eqn:Baj-decomp-std}) induces an isomorphism
    \begin{align}\label{eqn:tilBaj-decom}
        \tilBaj_{S_j} \cong \Baj_{S_j} \times \GL_2.
    \end{align}

\begin{lemma}\label{lem:im-is-product}
    The scheme--theoretic image of $\tilBaj$ (resp. $\tilBaj_{S_j}$) under $\widetilde{\pr}_j$ is isomorphic to the product of $\GL_2$ with the scheme--theoretic image of $\Baj$ (resp. of $\Baj_{S_j}$). 
\end{lemma}
 \begin{proof}
     Since $\LGn \cong \GL_2 \times \Kn$, pullback along the bottom arrow of (\ref{diag:projective-prj}) induces an isomorphism $$\widetilde{U}(\tilz) \cong U(\tilz) \times \GL_2 $$ for an appropriate closed subscheme $U(\tilz) \subset \widetilde{U}(\tilz)$. Via (\ref{eqn:Baj-decomp-std}) and the isomorphism above, the map $\widetilde{\pr}_j$ can be viewed as a map
     $$\Baj \times \GL_2 \to U(\tilz) \times \GL_2$$ given by sending 
     $\left((l_j\kappa_j, X_j, r_j), \kappa_j\right) \mapsto \left(X_j \tilz_j, \kappa_j\right)$. Therefore $\widetilde{\pr}_j$ is the product of a map $\pr'_j: \Baj \to U(\tilz)$ and $\id: \GL_2 \to \GL_2$, implying that the scheme--theoretic image of $\widetilde{\pr}_j$ is isomorphic to a product of the scheme--theoretic image of $\pr'_j$ with $\GL_2$. Finally, the observation that the scheme--theoretic image of $\pr'_j$ is isomorphic to that of $$\pr_j: \Baj \xrightarrow{\widetilde{\pr}'_j} U(\tilz) \xhookrightarrow{\text{closed}}  \widetilde{U}(\tilz)$$ finishes the proof. \end{proof}

We let $Z_j$ denote the scheme--theoretic image of $\Baj_{S_j}$ under $\pr_j$.
Define
$$\tilBa \stackrel{\text{def}}{=} \prod_{j \in \Z/f\Z} \tilBaj$$ and let $\tilBa_S \subset \tilBa$ be the closed subscheme $\prod_{j \in \Z/f\Z} \tilBaj_{S_j}$. 
\begin{align*}
    \pr_{\tilB}: \tilBa \longrightarrow \widetilde{U}(\tilz)
\end{align*}
by mapping $(l_j, \kappa_j, X_j, r_j)_j \mapsto (\kappa_j X_j \tilz_j)_j$. Denote by $\mathrm{Im}(\pr_{\tilB})$ the scheme--theoretic image of $\pr_{\tilB}$, admitting a closed immersion into $\widetilde{U}(\tilz)$. By Lemma \ref{lem:im-is-product}, the scheme--theoretic image of $\tilBa_S$ under $\pr_{\tilB}$ is isomorphic to $$\prod_j \left(\GL_2 \times Z_j\right),$$ which therefore admits a closed immersion into $\mathrm{Im}(\pr_{\tilB})$.
\begin{lemma}
    Setting $l_{j-1} = r_{j}$ for each $j \in \Z/f\Z$ cuts out closed subschemes $\tilYz \xrightarrow{\Delta} \tilBa$ and $\tilYz_S \xhookrightarrow{\Delta_S} \tilBa_S$.
\end{lemma}
\begin{proof}
    Obvious from the definitions of $\tilBa$ and $\tilBa_S$, and Proposition \ref{prop:irred-Gr}.
\end{proof}

We obtain the following commutative diagram
\begin{equation}\label{diagram:define-covers-z}
\begin{tikzcd}[row sep = 4ex, column sep=1.7ex]
\tilYz_S \arrow[hook]{rr}[xshift=11pt]{\Delta_S}[swap]{\mathit{cl.}} \arrow[hook]{dd}{\mathit{cl.}} \arrow[dr, two heads] & &\tilBa_S \arrow[hook]{dd}[yshift=12pt]{\mathit{cl.}} \arrow[dr, two heads]  \\
     &\tilZz_S \arrow[hook]{rr}[swap,xshift=14pt]{\mathit{cl.}}  \arrow[hook]{dd}[yshift=12pt]{\mathit{cl.}}  & & \prod_j \left(\GL_2 \times Z_j \right) \arrow[hook]{dd}{\mathit{cl.}} \\
\tilYz \arrow[hook]{rr}[xshift=17pt]{\Delta}[swap,{xshift=14pt}]{\mathit{cl.}} \arrow[dr, two heads] & &\tilBa \arrow[dr, two heads]  \\
     &\tilZz \arrow[hook]{rr}{\mathit{cl.}}  & & \mathrm{Im}(\pr_{\tilB}) \arrow[hook]{r}{\mathit{cl.}} &\widetilde{U}(\tilz) 
\end{tikzcd}
\end{equation}
where
\begin{itemize}
    \item except for the nodes $\tilBa_S$, $\tilBa$, $\prod_j \left(\GL_2 \times Z_j \right)$ and $\mathrm{Im}(\pr_{\tilB})$, this diagram is a subdiagram of (\ref{diagram:define-covers});
    \item all two-headed arrows are proper and scheme--theoretically dominant; and
    \item all hooked arrows annotated with \textit{cl.} are closed immersions.
\end{itemize}

\section{Local geometry}\label{sec:geometry}

We fix $\tau = \tau(s, \mu)$ non--scalar principal series tame type with $\mu = (\mu_j)_j$ small; $\tilz = (\tilz_j)_j \in \widetilde{W}^{\Z/f\Z}$ with each $\tilz_j \in \{w_0 t_{\eta}, t_{w_0 (\eta)}\} s_j^{-1}v^{\mu_j}$; and $S = (S_j)_j \in \{L, R\}^{\Z/f\Z}$. Assume $p - 2 > \max_j \<\mu_j, \alpha^{\vee}\> $.

The schemes $\Baj$ are described in detail in \cite[Table~3]{lhmm} (in \textit{loc. cit.}, everything is defined over $\cO$, and so, after setting $p=0$, we obtain the $\Baj$'s considered in this article). We now add to those descriptions to further describe the closed subschemes $\Baj_{S_j}$ explicitly.

\begin{lemma}\label{lemma:shape-conditions}
    For each $j$, the ideals cutting out $\Baj_{L}$ and $\Baj_{R}$ in $\Baj$ as well as the scheme--theoretic images of $\Baj_{L}$ and $\Baj_{R}$ under $\pr_j$ are given by Tables \ref{table:shape-L} and \ref{table:shape-R} respectively. Here, 
    % $I$ denotes the ideal, $Z$ denotes the scheme--theoretic image of $\Baj_{L}$ in Table \ref{table:shape-L} and of $\Baj_R$ in Table \ref{table:shape-R}, 
    $[x:y]$ are the coordinates of $l_j \kappa_j$, $[x':y']$ are the coordinates of $r_j$, and the rest of the variables are in terms of \cite[Table~3]{lhmm}.
\end{lemma} 

\begin{center}
 \begin{table}[ht]
 \begin{threeparttable}
 %\Rotatebox{90}{%
 \resizebox{\textwidth}{!}{%
{\renewcommand{\arraystretch}{2}
\begin{tabular}{|c|c|| c|c|}
\hline
\multicolumn{1}{|c|}{\backslashbox{$\langle \mu_j, \alpha^{\vee} \rangle$}{$\tilw_j$}} & \multicolumn{1}{c||}{} & 
%\multicolumn{1}{c|}{$t_{\eta}$} & 
\multicolumn{1}{c|}{$w_0 t_{\eta}$} & \multicolumn{1}{c|}{$t_{w_0 (\eta)}$}\\
\hline
& $s_j$ &  & \\
\hline
\hline
\multirow{2}{*}{$>1$} & $w_0$ & 
% {\color{Brown} \begin{tabular}{l}
% $I = (B)$ \\
% $\Baj_L = \Proj \F [x,y] \times \Spec \F[C']$\\
% $Z= \Spec \F[C']$
% \end{tabular}} &
{\color{Brown}\begin{tabular}{l}
 $I = (B)$ \\
$\Baj_L = \Proj \F [x,y] \times \Spec \F[C']$\\
$Z_j= \Spec \F[C']$
\end{tabular}} & \begin{tabular}{l}
$I = (1)$ \\
$\Baj_L = Z_j =\varnothing$
\end{tabular} \\
\cline{2-4}
  & $\id$ & 
  % {\color{Brown}\begin{tabular}{l}
%  $I = (C)$ \\
% $\Baj_L = \Proj \F [x,y] \times \Spec \F[C']$\\
% $Z= \Spec \F[C']$
% \end{tabular}} & 
{\color{Brown}\begin{tabular}{l}
 $I = (C)$ \\
$\Baj_L = \Proj \F [x,y] \times \Spec \F[C']$\\
$Z_j= \Spec \F[C']$
\end{tabular}} & 
\begin{tabular}{l}
 $I = (1)$ \\
$\Baj_L = Z_j= \varnothing$
\end{tabular}\\
\hline
\hline

\multirow{2}{*}{$=1$} & $w_0$ & 
% {\color{Plum} \begin{tabular}{l}
% $I = (B, C, D)$ \\
% $\Baj_L = \Proj \F [x,y] \times \Proj \F[x', y']$\\
% $Z= \Spec \F$
% \end{tabular}} & 
{\color{Plum}\begin{tabular}{l}
$I = (B, C, D)$ \\
$\Baj_L = \Proj \F [x,y] \times \Proj \F[x', y']$\\
$Z_j= \Spec \F$
\end{tabular}} & \begin{tabular}{l}
$I = (1)$\\
$\Baj_L = Z_j = \varnothing$
\end{tabular} \\
\cline{2-4}
 & $\id$ & 
%  {\color{Brown}\begin{tabular}{l}
% $I = (C)$ \\
% %Note that $\Baj = D(y')$ \\
% $\Baj_L = \Proj \F[x, y] \times \Spec \F[C']$ \\
% $Z_j \Spec \F[C']$
% \end{tabular}} &
{\color{Brown}\begin{tabular}{l}
$I = (C)$ \\
$\Baj_L = \Proj \F[x, y] \times \Spec \F[C']$ \\
$Z_j = \Spec \F[C']$ 
\end{tabular}} & \begin{tabular}{l}
$I = (1)$ \\
$\Baj_L = Z_j= \varnothing$
\end{tabular} \\
\hline
\hline

$=0$ & $\id$ &
% {\color{Brown}\begin{tabular}{l}
% $I = (x'-y'C)$\\
% $\Baj_L = \Proj \F[x,y] \times \Spec \F[C]$ \\
% $Z_j \Spec \F[C]$
% \end{tabular}} &
{\color{Brown}\begin{tabular}{l}
$I = (x'-y'C)$\\
$\Baj_L = \Proj \F[x,y] \times \Spec \F[C]$ \\
$Z_j= \Spec \F[C]$
\end{tabular}} & {\color{Brown}\begin{tabular}{l}
$I = (y' - x'B)$ \\
$\Baj_L = \Proj \F[x, y] \times \Spec \F[B]$ \\
$Z_j= \Spec \F[B]$
\end{tabular}} \\
\hline
\hline

\end{tabular}}
}
%}%
\caption{Ideals $I$ cut out $\Baj_L$ in $\Baj$ and the scheme $Z_j$ is the scheme--theoretic image of $\Baj_L$ under $\pr_j$. Different font colors correspond to different isomorphism classes of the tuple $(\Baj_L, p_j|_{\Baj_L}, q_j|_{\Baj_L}, \pr_j|_{\Baj_L})$.}\label{table:shape-L}
\end{threeparttable}
\end{table}
\end{center}

\begin{center}
 \begin{table}[ht]
 \begin{threeparttable}
 %\Rotatebox{90}{%
 \resizebox{\textwidth}{!}{%
 {\renewcommand{\arraystretch}{2}
\begin{tabular}{|c|c|| c|c|}
\hline
\multicolumn{1}{|c|}{\backslashbox{$\langle \mu_j, \alpha^{\vee} \rangle$}{$\tilw_j$}} & \multicolumn{1}{c||}{} & 
%\multicolumn{1}{c|}{$t_{\eta}$} & 
\multicolumn{1}{c|}{$w_0 t_{\eta}$} & \multicolumn{1}{c|}{$t_{w_0 (\eta)}$}\\
\hline
& $s_j$ & & \\
\hline
\hline
\multirow{2}{*}{$>1$} & $w_0$ & 
% {\color{Magenta} \begin{tabular}{l}
% $I = (y)$ \\
% $\Baj_R =  Z_j= \Spec \F[B, C']$ \\
% \end{tabular}} &
{\color{Magenta}\begin{tabular}{l}
$I = (x)$ \\
$\Baj_R =  Z_j= \Spec \F[B, C']$ \\
\end{tabular}} & {\color{Magenta} \begin{tabular}{l}
$I = (0)$ \\
$\Baj_R =  Z_j= \Spec \F[C, C']$ \\
\end{tabular}} \\
\cline{2-4}
 & $\id$ & 
%  {\color{Magenta}\begin{tabular}{l}
% $I = (y)$ \\
% $\Baj_R =  Z_j= \Spec \F[C, C']$ \\
% \end{tabular}} &
{\color{Magenta} \begin{tabular}{l}
$I = (x)$ \\
$\Baj_R =  Z_j= \Spec \F[C, C']$ \\
\end{tabular}} & {\color{Magenta} \begin{tabular}{l}
$I = (0)$ \\
$\Baj_R =  Z_j= \Spec \F[C, C']$ \\
\end{tabular}} \\
\hline
\hline

\multirow{2}{*}{$=1$} & $w_0$ &
% {\color{ForestGreen}\begin{tabular}{l} 
% $I = (x x' - y y')$\\
% $\Baj_R =$ \\
% \quad $\Proj \F[x,y]/(Cx-Dy, Dx+By)$\\
% \quad $\times \Spec \F[B, C, D]/(D^2 + BC)$\\
% $Z=\Spec \F[B, C, D]/(D^2 + BC)$
% \end{tabular}} &
{\color{ForestGreen} \begin{tabular}{l}
 $I = (xy' - yx')$\\
 $\Baj_R = $\\
 \quad $\Spec \F[B, C, D] $\\
 \quad $\times \Proj \F[x,y]/(Dx-Cy, Bx+Dy)$\\
 $Z_j= \Spec \F[B, C, D]/(D^2 + BC)$
\end{tabular}} & {\color{Magenta} \begin{tabular}{l}
$I = (0)$\\
$\Baj_R =Z_j= \Spec \F[C, C']$ \\
\end{tabular}} \\
\cline{2-4}
 & $\id$ & 
%  {\color{Magenta} \begin{tabular}{l}
%  $I = (y)$ \\
%  $\Baj_R = Z_j= \Spec \F[C, C']$
% \end{tabular}} &
{\color{Magenta}\begin{tabular}{l}
$I = (x)$ \\
$\Baj_R = Z_j= \Spec \F[C, C']$
\end{tabular}} & {\color{ForestGreen} \begin{tabular}{l}
$I = (0)$ \\
$\Baj_R = $\\
\quad $\Spec \F[B, C, D]$\\
 \quad $\times \Proj \F[x,y]/(Dx-Cy, Bx+Dy)$\\
$Z_j= \Spec \F[B, C, D]/(D^2 + BC)$
\end{tabular}} \\
\hline
\hline

$=0$ & $\id$ & 
% {\color{MidnightBlue} \begin{tabular}{l}
% $I = (y)$ \\
% $\Baj_R = \Spec \F[C] \times \Proj \F[x', y']$ \\
% $Z_j= \Spec \F[C]$
% \end{tabular}} &
{\color{MidnightBlue} \begin{tabular}{l}
$I = (x)$ \\
$\Baj_R = \Spec \F[C] \times \Proj \F[x', y']$ \\
$Z_j= \Spec \F[C]$
\end{tabular}} & {\color{MidnightBlue} \begin{tabular}{l}
$I = (x)$ \\
$\Baj_R = \Spec \F[B] \times \Proj \F[x', y']$ \\
$Z_j= \Spec \F[B]$
\end{tabular}} \\
\hline
\hline
\end{tabular}}}
\caption{Ideals $I$ cut out $\Baj_R$ in $\Baj$ and the scheme $Z_j$ is the scheme--theoretic image of $\Baj_R$ under $\pr_j$. Different font colors correspond to different isomorphism classes of the tuple $(\Baj_R, p_j|_{\Baj_R}, q_j|_{\Baj_R}, \pr_j|_{\Baj_R})$. %Note that when $\<\mu_j, \alpha^{\vee}\>= 1$ and $(s_j, \tilw_j) \in \{((1, 2), t_{\eta}), ((1, 2), w_0 t_{\eta}), (\id, t_{w_0 (\eta)})\}$, the map $\pr_j$ induces an isomorphism $\Baj_R \smallsetminus V(B, C, D) \cong Z \smallsetminus V(B, C, D)$.
}\label{table:shape-R}
\end{threeparttable}
\end{table}
\end{center}

\begin{proof} Let $(l_j\kappa_j, X_j, r_j)$ be a point of $\Baj$ and let $\till_j$ and $\tilr_j$ be lifts to $\GL_2$ of $l_j$ and $r_j$ respectively. Let
$$\till_j \kappa_j = \begin{pmatrix}
u & z \\
x & y
\end{pmatrix} \quad \text{ and } \quad \tilr_j =\begin{pmatrix}
t & -s \\
x' & y'
\end{pmatrix},$$ where $u, z, s, t$ can be freely chosen subject to the restriction that $\till_j, \tilr_j$ are invertible. Upon restricting to the distinguished opens $D(x, x'), D(y, x'), D(x, y'), D(y, y')$, we will make the following choices for $u, z, s, t$: 
\begin{align*}
    u=0, \quad z=1, \quad s=1, \quad t=0 \quad &\text{ on } D(x, x'), \\
    u=1, \quad z=0, \quad s=1, \quad t=0 \quad &\text{ on } D(y, x'), \\
    u=0, \quad z=1, \quad s=0, \quad t=1 \quad &\text{ on } D(x, y'), \\
    u=1, \quad z=0, \quad s=0, \quad t=1 \quad &\text{ on } D(y, y').
\end{align*}
We now deal with the different cases separately, taking the descriptions of $$s_j w_j^{-1} X_j w_j s_j^{-1}$$ along with the relations the various variables appearing in $X_j$ satisfy from \cite[Tables~2,\;3]{lhmm}. In the following, let $k_j := \langle \mu_j, \alpha^{\vee}\rangle$.

\begin{enumerate}
% \item Let $k_j >1, (s_j, \tilw_j) = (w_0, t_{\eta}).$ Then
% $$\hspace{1cm}X_j = w_j s_j^{-1} \begin{pmatrix}
%     1 + BC'v^{-k_j} & Bv^{-1} \\
%     C' v^{-k_j + 1} & 1
% \end{pmatrix} s_j w_j = \begin{pmatrix}
%     1 & C' v^{-k_j + 1}  \\
%     Bv^{-1} & 1 + BC'v^{-k_j}
% \end{pmatrix}$$ 
% with the variables $B, C'$ satisfying $x' - y'C' = yB = 0$. Thus, $\Baj = D(y')$, $(s, t) = (0, 1)$ and
%   $$\hspace{1cm} (\det \tilr_j) W_j = \begin{pmatrix}
%                 u & z \\
%                 x & y
%             \end{pmatrix} 
%             \begin{pmatrix} 
%             1& C' v^{-k_j+1}\\  
%             B v^{-1}& 1 + B C' v^{-k_j}
%             \end{pmatrix}
%             \begin{pmatrix}
%                 v & 0 \\
%                 0 & 1
%             \end{pmatrix} \begin{pmatrix}
%                  1 & -x' v^{-k_j} \\
%                  0 & y'
%             \end{pmatrix}.$$
% Computing the product shows that $v \mid \WjL$ if and only if $z B = 0$, and $v \mid \WjR$ if and only if $y y' =0$. The condition $z B = 0$ holds if and only if $B = 0$, since $z=1$ on $D(x,y')$ and $B=0$ on $D(y,y')$. We have $y y'=0$ if and only if $y=0$ since we are on $D(y')$.

\item Let $k_j > 1, (s_j, \tilw_j) = (w_0, w_0 t_\eta).$ Then
$$\hspace{1cm} X_j = w_j s_j^{-1} \begin{pmatrix}
    1 + BC'v^{-k_j} & Bv^{-1} \\
    C' v^{-k_j + 1} & 1
\end{pmatrix} s_j w_j = \begin{pmatrix}
    1 + BC'v^{-k_j} & Bv^{-1} \\
    C' v^{-k_j + 1} & 1
\end{pmatrix}$$ 
with the variables $B, C'$ satisfying $x' - y'C'  = xB = 0$. Thus, $\Baj = D(y')$, $(s, t) = (0,1)$ and
    $$\hspace{1cm} (\det \tilr_j) W_j = 
            \begin{pmatrix}
                u & z \\
                x & y
            \end{pmatrix} 
            \begin{pmatrix}
                1 + BC'v^{-k_j} & Bv^{-1} \\
                C' v^{-k_j + 1} & 1
            \end{pmatrix}
            \begin{pmatrix}
                0 & 1 \\
                v & 0
            \end{pmatrix} \begin{pmatrix}
                 1 & -x' v^{-k_j} \\
                 0 & y'
            \end{pmatrix}.$$ 
This shows that $v \mid \WjL$ if and only if $u B = 0$, and $v \mid \WjR$ if and only if $x y' =0$. Note that $u B = 0$ if and only if $B = 0$, since $u=1$ on $D(y,y')$ and $B=0$ on $D(x,y')$. We have $x y'=0$ if and only if $x=0$ since we are on $D(y')$.

\item Let $k_j \geq 1, (s_j, \tilw_j) = (w_0, t_{w_0(\eta)}).$ Then
$$\hspace{1cm} X_j = w_j s_j^{-1} \begin{pmatrix}
    1  & 0 \\
    Cv^{-1} + C' v^{-k_j - 1} & 1
\end{pmatrix} s_j w_j = \begin{pmatrix}
    1 & Cv^{-1} + C' v^{-k_j - 1} \\
    0 & 1
\end{pmatrix}$$ with the variables $C, C'$ satisfying $x' - y'C' = x = 0$. Thus $\Baj = D(y, y')$, $(u, z, s, t) = (1, 0, 0, 1)$ and
   $$\hspace{1cm} (\det \tilr_j) W_j =  \begin{pmatrix}
                1 & 0 \\
                0 & y
            \end{pmatrix} 
             \begin{pmatrix} 
             1 &  C v^{-1} + C' v^{-(k_j+1)}\\
             0& 1
             \end{pmatrix}
            \begin{pmatrix}
                1 & 0 \\
                0 & v
            \end{pmatrix} \begin{pmatrix}
                 1 & -x' v^{-k_j} \\
                 0 & y'
            \end{pmatrix}.$$ 
This shows that $v \nmid \WjL$ and $v \mid \WjR$.

% \item Let $k_j \geq 1, (s_j, \tilw_j) = (\id , t_\eta).$ Then 
% $$\hspace{1cm} X_j = w_j s_j^{-1} \begin{pmatrix}
%     1  & 0 \\
%     Cv^{-1} + C' v^{-k_j - 1} & 1
% \end{pmatrix} s_j w_j = \begin{pmatrix}
%     1  & 0 \\
%     Cv^{-1} + C' v^{-k_j - 1} & 1
% \end{pmatrix}$$
% with the variables $C, C'$ satisfying $x' - y'C' = yC = 0$. Thus, $\Baj = D(y')$, $(s, t)=(0, 1)$ and 
%    $$\hspace{1cm} (\det \tilr_j) W_j =  \begin{pmatrix}
%                 u & z \\
%                 x & y
%             \end{pmatrix} 
%              \begin{pmatrix}
%                 1  & 0 \\
%                 Cv^{-1} + C' v^{-k_j - 1} & 1
%             \end{pmatrix}
%             \begin{pmatrix}
%                 v & 0 \\
%                 0 & 1
%             \end{pmatrix} \begin{pmatrix}
%                  y' & 0 \\
%                  -x' v^{-k} & 1
%             \end{pmatrix}.$$
% This shows that, $v \mid \WjL$ if and only if $z y' C = 0$, and $v \mid \WjR$ if and only if $y =0$. We have $z y' C = 0$ if and only if $z C =0$ since we are on $D(y')$. Moreover, $z C = 0$ if and only if $C = 0$, since $z = 1$ on $D(x,y')$ and $C=0 $ on $D(y,y')$.

\item Let $k_j \geq 1, (s_j, \tilw_j) = (\id , w_0 t_\eta).$ Then
$$\hspace{1cm} X_j = w_j s_j^{-1} \begin{pmatrix}
    1  & 0 \\
    Cv^{-1} + C' v^{-k_j - 1} & 1
\end{pmatrix} s_j w_j = \begin{pmatrix}
    1 & Cv^{-1} + C' v^{-k_j - 1} \\
    0 & 1
\end{pmatrix}$$ with the variables $C, C'$ satisfying $x' - y'C' = xC = 0$. Thus, $\Baj = D(y')$, $(s,t) =(0,1)$ and 
$$ \hspace{1cm} (\det \tilr_j) W_j = \begin{pmatrix}
                u & z \\
                x & y
            \end{pmatrix} 
             \begin{pmatrix} 
              1 & Cv^{-1} + C' v^{-k_j - 1} \\
                0 & 1
             \end{pmatrix}
            \begin{pmatrix}
                0 & 1 \\
                v & 0
            \end{pmatrix} \begin{pmatrix}
                 y' & 0 \\
                 -x' v^{-k_j} & 1
            \end{pmatrix}.$$
This shows that $v \mid \WjL$ if and only if $u y' C = 0$, and $v \mid \WjR$ if and only if $x =0$. The equality $u y' C = 0$ holds if and only if $u C =0$ since we are on $D(y')$. Moreover, $u C = 0$ if and only if $C = 0$, since $u = 1$ on $D(y,y')$ and $C=0 $ on $D(x,y')$.

\item Let $k_j > 1, (s_j, \tilw_j) = (\id ,  t_{w_0(\eta)}).$ Then 
$$\hspace{1cm} X_j = w_j s_j^{-1} \begin{pmatrix}
    1 + BC'v^{-k_j} & Bv^{-1} \\
    C' v^{-k_j + 1} & 1
\end{pmatrix} s_j w_j = \begin{pmatrix}
    1 + BC'v^{-k_j} & Bv^{-1} \\
    C' v^{-k_j + 1} & 1
\end{pmatrix}$$ with the variables $B, C'$ satisfying $x' - y'C' = x = 0$. Therefore, $\Baj = D(y,y')$, $(u, z, s, t) = (1, 0, 0, 1)$
and
$$\hspace{1cm} (\det \tilr_j) W_j =  \begin{pmatrix}
                1 & 0 \\
                0 & y
            \end{pmatrix} 
            \begin{pmatrix}
                1 + BC'v^{-k_j} & Bv^{-1} \\
                C' v^{-k_j + 1} & 1
            \end{pmatrix}
            \begin{pmatrix}
                1 & 0 \\
                0 & v
            \end{pmatrix} \begin{pmatrix}
                 y' & 0 \\
                 -x' v^{-k_j} & 1
            \end{pmatrix}$$
Hence, $v \nmid \WjL$ and $v \mid \WjR$. 

% \item Let $k_j =1, (s_j, \tilw_j) = (w_0, t_{\eta}).$ Then 
% $$\hspace{1cm} X_j = w_j s_j^{-1} \begin{pmatrix}
%     1 - Dv^{-1} & Bv^{-1} \\
%     Cv^{-1}& 1 + Dv^{-1}
% \end{pmatrix} s_j w_j = \begin{pmatrix}
%     1 + Dv^{-1} & Cv^{-1} \\
%     Bv^{-1}& 1 - Dv^{-1}
% \end{pmatrix}$$ with the variables $B, C, D$ satisfying
% $x'D - y'C = x'B + y'D = xC - yD = xD + yB = 0$. Thus,
%     $$\hspace{1cm} (\det \tilr_j)W_j = \begin{pmatrix}
%                 u & z \\
%                 x & y
%             \end{pmatrix} 
%             \begin{pmatrix} 1 + D_j v^{-1}& C_j v^{-1}\\ B_j v^{-1}& 1 - D_j v^{-1}\end{pmatrix}
%             \begin{pmatrix}
%                 v & 0 \\
%                 0 & 1
%             \end{pmatrix} \begin{pmatrix}
%                  t & -x' v^{-1} \\
%                  s v & y'
%             \end{pmatrix}$$
% This shows that $v \mid \WjL$ if and only if $t z B + s u C + (t u - s z )D  = 0$, and $v \mid \WjR$ if and only if $y y'- x x' = 0$. Plugging in the values of $u, z, s, t$ on each chart, we find that $t z B + s u C + (t u - s z )D  = 0$ if and only if $B = C= D = 0$. 

\item Let $k_j =1, (s_j, \tilw_j) = (w_0, w_0 t_{\eta}).$ Then 
$$\hspace{1cm} X_j = w_j s_j^{-1} \begin{pmatrix}
    1 - Dv^{-1} & Bv^{-1} \\
    Cv^{-1}& 1 + Dv^{-1}
\end{pmatrix} s_j w_j = \begin{pmatrix}
    1 - Dv^{-1} & Bv^{-1} \\
    Cv^{-1}& 1 + Dv^{-1}
\end{pmatrix}$$ with the variables $B, C, D$ satisfying
$x'D - y'C = x'B + y'D = xD - yC = xB + yD = 0$. Thus, 
    $$\hspace{1cm} (\det \tilr_j) W_j = \begin{pmatrix}
                u & z \\
                x & y
            \end{pmatrix} 
            \begin{pmatrix}
                1 - Dv^{-1} & Bv^{-1} \\
                Cv^{-1}& 1 + Dv^{-1}
            \end{pmatrix}
            \begin{pmatrix}
                0 & 1 \\
                v & 0
            \end{pmatrix} \begin{pmatrix}
                 t & -x' v^{-1} \\
                 s v & y'
            \end{pmatrix}$$
    This shows that $v \mid \WjL$ if and only if $t u B + s z C + (t z - s u )D = 0$ and $v \mid \WjR$ if and only if $ x y' - y x' = 0$. Checking on each of the charts, we find that $t u B + s z C + (t z - s u )D  = 0$ if and only if $B = C = D = 0$.

\item Let $k_j=1, (s_j, \tilw_j) = (\id, t_{w_0(\eta)}).$ Then 
$$\hspace{1cm} X_j = w_j s_j^{-1} \begin{pmatrix}
    1 - Dv^{-1} & Bv^{-1} \\
    Cv^{-1}& 1 + Dv^{-1}
\end{pmatrix} s_j w_j = \begin{pmatrix}
    1 - Dv^{-1} & Bv^{-1} \\
    Cv^{-1}& 1 + Dv^{-1}
\end{pmatrix}$$ with the variables $B, C, D$ satisfying
$x'D - y'C = x'B + y'D = xy' - yx' = 0$. Thus,
$$ \hspace{1cm} (\det \tilr_j) W_j = \begin{pmatrix}
                u & z \\
                x & y
            \end{pmatrix} 
           \begin{pmatrix}
            1 - Dv^{-1} & Bv^{-1} \\
            Cv^{-1}& 1 + Dv^{-1}
           \end{pmatrix}
            \begin{pmatrix}
                1 & 0 \\
                0 & v
            \end{pmatrix} \begin{pmatrix}
                 y' & s v \\
                -x' v^{-1}  & t
            \end{pmatrix}.$$
This shows that $v \mid \WjL$ if and only if $u y' - z x'= 0$, and $v \mid \WjR$ if and only if $t x B + s y C + ( t y - s x ) D= 0$. The relation $x y' -y x' = 0$ implies that $\Baj = D(x,x') \cup D(y,y')$. Since $u y' - z x' = -x'$ on $D(x,x')$ and $u y' - z x' = y'$ on $D(y,y')$, $u y' - z x'$ vanishes nowhere on $\Baj$. Further, since $t x B + s y C + ( t y - s x ) D$  equals $y C - x D$ on $D(x,x')$ and $x B + y D$ on $D(y, y')$, $t x B + s y C + ( t y - s x ) D$ vanishes on $\Baj$.

% \item Let $k_j=0, (s_j, \tilw_j) = (\id, t_\eta).$ Then 
% $$X_j = w_j s_j^{-1} \begin{pmatrix}
%     1 & 0 \\
%     Cv^{-1}& 1
% \end{pmatrix} s_j w_j = \begin{pmatrix}
%     1 & 0 \\
%     Cv^{-1}& 1
% \end{pmatrix}$$ with the variable $C$ satisfying the relation $yx' - yy'C =0$. Thus, 
%     $$(\det \tilr_j) W_j = \begin{pmatrix}
%                 u & z \\
%                 x & y
%             \end{pmatrix} 
%            \begin{pmatrix}
%            1 & 0\\ 
%            C v^{-1}& 1 
%            \end{pmatrix}
%             \begin{pmatrix}
%                 v & 0 \\
%                 0 & 1
%             \end{pmatrix} \begin{pmatrix}
%                  y' & s  \\
%                 -x' & t
%             \end{pmatrix}$$
%     Hence, $v \mid \WjL$ if and only if $z(x' - C y') = 0$, and $v \mid \WjR$ if and only if $ y(t + s C)= 0$. Note that $z(x' - y'C) = 0$ if and only if $x' - y'C=0$, since $z=1$ on $D(x)$ and we already know $x' - y' C=0$ on $D(y)$. Further, $y(t+sC) = 0$ if and only if $y=0$, since $y(t+sC) = y$ on $D(y')$, whereas on $D(x')$, $y$ is a multiple of $y(t+sC) = yC$. 
\item Let $k_j=0, (s_j, \tilw_j) = (\id, w_0 t_\eta).$ Then 
$$X_j = w_j s_j^{-1} \begin{pmatrix}
    1 & 0 \\
    Cv^{-1}& 1
\end{pmatrix} s_j w_j = \begin{pmatrix}
    1 & Cv^{-1} \\
    0 & 1
\end{pmatrix}$$ with the variable $C$ satisfying $xx' - xy'C = 0$.
Thus,
    $$(\det \tilr_j) W_j = \begin{pmatrix}
                u & z \\
                x & y
            \end{pmatrix} 
           \begin{pmatrix}
           1 & C v^{-1}\\ 
           0& 1 
           \end{pmatrix}
            \begin{pmatrix}
                0 & 1 \\
                v & 0
            \end{pmatrix} \begin{pmatrix}
                 y' & s  \\
                -x' & t
            \end{pmatrix}.$$
    Hence, $v \mid \WjL$ if and only if $u(x' - y'C)= 0$, and $v \mid \WjR$ if and only if $ x(t + s C)= 0$. Note that $u(x' - y'C)=0$ if and only if $x'-  y'C=0$, since $u = 1$ on $D(y)$ and we already know that $x' - y'C = 0$ on $D(x)$. Further, $x(t+s C)  = 0$ if and only if $x = 0$, since $x(t+sC) = x$ on $D(y')$, whereas on $D(x')$, $x$ is a multiple of $x(t+sC) = xC$. 

\item Let $k_j=0, (s_j, \tilw_j) = (\id, t_{w_0(\eta)}).$ Then 
$$X_j = w_j s_j^{-1} \begin{pmatrix}
    1 & Bv^{-1} \\
    0& 1
\end{pmatrix} s_j w_j = \begin{pmatrix}
    1 & Bv^{-1} \\
    0& 1
\end{pmatrix}$$ with the variable $B$ satisfying $xx'B - xy' = 0$. Thus,
$$ (\det \tilr_j) W_j = \begin{pmatrix}
                u & z \\
                x & y
            \end{pmatrix} 
           \begin{pmatrix}
           1 & B v^{-1}\\ 
           0& 1 
           \end{pmatrix}
            \begin{pmatrix}
                1 & 0 \\
                0 & v
            \end{pmatrix} \begin{pmatrix}
                 y' & s  \\
                -x' & t
            \end{pmatrix}.$$
        Hence, $v \mid \WjL$ if and only if $u(y' - x'B)= 0$, and $v \mid \WjR$ if and only if $ x(s + t B)= 0$. As in the previous case, $u(y' - x'B)=0$ if and only if $y'-x'B=0$, and $x(s + t B)  = 0$ if and only if $x = 0$.
\end{enumerate} 
In order to compute the scheme--theoretic image $Z_j$ of $\Baj_{S_j}$ under $\pr_j$, we note that if a reduced scheme $Z$ fits into a commutative diagram
\begin{equation*}
    \begin{tikzcd}
    \Baj_{S_j} \arrow[rr, "\pr_j"] \arrow[two heads]{dr}[swap,yshift=3]{\text{surjection}} && \widetilde{U}(\tilz) \\
    &Z \arrow[hook]{ur}[swap,yshift=2]{\text{monomorphism}}
\end{tikzcd} \quad ,
\end{equation*}
then Lemma \ref{lem:projective-pr} implies that the map $Z \hookrightarrow \widetilde{U}(\tilz)$ is a closed immersion, and so $Z_j =Z$. The descriptions of $Z_j$ are now immediate except possibly the relations satisfied by $B, C, D$ when $S_j=R$, $k_j = 1$ and $(s_j, \tilw_j) \in \{(w_0, w_0 t_{\eta}), (\id, t_{w_0(\eta)})$. The proof of Lemma \ref{lem:Rpr-dualizing}(iii) shows that the only relations $B, C, D$ satisfy with respect to each other are given by setting $(D^{2} + BC)$ equal to $0$.
\end{proof}

\subsection{Classification of \texorpdfstring{$\Baj_{S_j}$}{Baj}}\label{subsec:classfn}
We will henceforth restrict attention to the schemes $\Baj_{S_j}$ for $S_j \in \{L, R\}$ and their scheme--theoretic images. To simplify notation, we will denote the maps $p_j|_{\Baj_{S_j}}, q_j|_{\Baj_{S_j}}, \pr_j|_{\Baj_{S_j}}$ simply as $p_j, q_j, \pr_j$ respectively. Using Lemma \ref{lemma:shape-conditions}, we find that when non--empty, $\Baj_{S_j}, p_j, q_j, \pr_j$ admit one of the following descriptions up to isomorphism:

\begin{enumerate}
    \item\label{Baj-P1XP1} The scheme $\Baj_{S_j}$ is isomorphic to $\Proj \F[x,y] \times \Proj \F[x',y']$. The maps $p_j$ and $q_j$ are projections to the $[x:y]$ and $[x':y']$ coordinates respectively, while $\pr_j$ is a constant map. This happens when
    \begin{center}$S_j = L, \quad \tilw_j = w_0 t_\eta, \quad \text{and}\quad (\langle \mu_j, \alpha^{\vee}\rangle, s_j) = (1, w_0).$\end{center}
    \item\label{Baj-P1XA1} The scheme $\Baj_{S_j}$ is isomorphic to $\Proj \F[x,y] \times \AA^1$. The map $p_j$ is projection to the $[x:y]$ coordinate, the map $q_j$ is projection to the $\AA^1$ factor followed by inclusion into $\PP^{1}$ given by mapping $C \mapsto [C:1]$, and the map $\pr_j$ is the projection to the $\AA^1$ factor. This happens when
    \begin{itemize}
    \item $S_j = L, \quad \tilw_j = w_0 t_\eta, \quad \hspace{0.15cm}\text{and}\quad (\langle \mu_j, \alpha^{\vee}\rangle, s_j) \neq (1, w_0)$; or
    \item $S_j = L, \quad \tilw_j = t_{w_0(\eta)}, \quad \text{and}\quad (\langle \mu_j, \alpha^{\vee}\rangle, s_j) = (0, \id)$.
    \end{itemize}

    \item\label{Baj-blowup} The scheme $\Baj_{S_j}$ is isomorphic to $\Spec \F[B, C, D] \times \Proj \F[x,y]/(Dx-Cy, Bx+Dy)$. The maps $p_j$ and $q_j$ are the same and given by projection to the $[x:y]$ coordinates, while $\pr_j$ extracts the variables $B, C, D$. This happens when
    \begin{itemize}
    \item $S_j = R, \quad \tilw_j = w_0 t_\eta, \quad \hspace{0.15cm}\text{and}\quad (\langle \mu_j, \alpha^{\vee}\rangle, s_j) = (1, w_0)$; or
    \item $S_j = R, \quad \tilw_j = t_{w_0(\eta)}, \quad \text{and}\quad (\langle \mu_j, \alpha^{\vee}\rangle, s_j) = (1, \id)$.    
    \end{itemize}

    \item \label{Baj-A2} The scheme $\Baj_{S_j}$ is isomorphic to $\AA^2$. The map $p_j$ is a constant map, $q_j$ is given by $(C, C') \mapsto [C': 1]$, and $\pr_j$ extracts the variables $C, C'$. This happens when
    \begin{itemize}
    \item $S_j = R, \quad \text{and } \langle \mu_j, \alpha^{\vee}\rangle >1$; or
    \item $S_j = R, \quad \tilw_j = w_0 t_\eta, \quad \hspace{0.15cm}\text{and}\quad (\langle \mu_j, \alpha^{\vee}\rangle, s_j) = (1, \id)$; or
    \item $S_j = R, \quad \tilw_j = t_{w_0(\eta)}, \quad \text{and}\quad (\langle \mu_j, \alpha^{\vee}\rangle, s_j) = (1, w_0)$.    
    \end{itemize}

    \item\label{Baj-P1XA1-constant-pj} The scheme $\Baj_{S_j}$ is isomorphic to $\AA^1 \times \Proj \F[x',y']$. The map $p_j$ is constant, the map $q_j$ is projection to $[x':y']$ coordinates, and the map $\pr_j$ is the projection to the $\AA^1$ factor. This happens when 
    \begin{center}$S_j = R,  \quad \text{and}\quad (\langle \mu_j, \alpha^{\vee}\rangle, s_j) = (0, \id).$\end{center}
\end{enumerate}

We use the above classification to define the class $T_j$ of $\Baj_{S_j}$. We will say that $T_j$ equals $1, 2, 3, 4 \text{ or } 5$ if $\Baj_{S_j}$ is of the form described in (\ref{Baj-P1XP1}), (\ref{Baj-P1XA1}), (\ref{Baj-blowup}), (\ref{Baj-A2}) or (\ref{Baj-P1XA1-constant-pj}) respectively.

\subsection{A different decomposition of \texorpdfstring{$\tilBa_S$}{Ba}}\label{subsec:Ba-product} In order to make certain cohomological computations easier, we now consider $\tilBa_S$ as a product of schemes in a slightly different way and set up related notations. For each $j$ with $T_j \neq 3$, define schemes $\tilBaj_{S_j}^{\mathrm{I}}$ and $\tilBaj_{S_j}^{\mathrm{II}}$ via the following isomorphism:
\begin{align}\label{eqn:Baj-IandII}
    \tilBaj_{S_j} &\xrightarrow{\sim} \tilBaj_{S_j}^{\mathrm{I}} \times \tilBaj_{S_j}^{\mathrm{II}} \\
    (l_j, \kappa_j, X_j, r_j) &\mapsto (l_j\kappa_j, \kappa_j), (X_j, r_j). \nonumber
\end{align}
The map $p_j: \Baj_{S_j} \to \PP^{1}$ (resp. $q_j$) induces a map $\tilBa_S \to \PP^{1}$ via projection to the $\Baj_{S_j}$--factor under the isomorphism in (\ref{eqn:tilBaj-decom}). We denote this map by $p_j$ (resp. $q_j$) as well. Abusing notation further, via (\ref{eqn:Baj-IandII}), we let $p_j$ also denote the map $\tilBaj_{S_j}^{\mathrm{I}} \to \PP^{1}$ and $q_j$ the map $\tilBaj_{S_j}^{\mathrm{II}} \to \PP^{1}$.

% \mar{Indexing set is a bit confusing here} The indexing set for all quantities above is $\Z/f\Z$, but we will frequently take the indexing set to be $\Z$ via the natural projection $\Z \to \Z/f\Z = \Z/f\Z$. 
\begin{definition}
    Given a class tuple $T = (T_j)_{j \in \Z/f\Z}$, let $\trn$ be the set of sequences $\gr = (i-k, i-k+1, \dots, i)$ of elements in $\Z/f\Z$ where 
    \begin{itemize}
        \item the length of the sequence, denoted $l(\gr)$, is $\geq 2$,
        \item $T_{i-k}, T_{i} \neq 3$, and
        \item whenever $l(\gr) \geq 3$, $T_{i-k+1} = \dots = T_{i-1} = 3$.
    \end{itemize}  
    
%      such that modulo the equivalence relation where t
    
%     pairs
% $$((T_{i-k}, T_{i-k+1}, \dots, T_{i}), \{i-k, i-k+1, \dots, i\} \text{ mod } f)$$
% where $k \in [1, f]$, $T_{i-k}, T_{i} \neq 3$ and if $k \neq 1$, then $T_{i-k+1} = \dots = T_{i-1} = 3$. 
\end{definition}
% In other words, $\trn$ is the set of certain subsequences of $T$ along with the data of the indices involved.
% Unless $T_j = 3$ for each $j$, $\trn \neq \varnothing$ and the subsequences in $\trn$ cover $T$. For ease of notation, while describing an element of $\trn$, we will not explicit mention the data of the indices. 
It is evident from the definition that if $T_j \neq 3$ for some $j \in \Z/f\Z$, then there exists a sequence in $\trn$ that starts with $j$ and a sequence that ends with $j$. Furthermore, whenever $\trn \neq \varnothing$, every $j \in \Z/f\Z$ shows up in at least one sequence in $\trn$.
Note that $\trn$ depends on the data of $\tilz, s, \mu, S$ since it depends on the the class tuple.

\begin{definition}\label{defn:trn-schemes}
    Suppose $\gr = (i-k, \dots, i) \in \trn$.
    \begin{enumerate}
        \item Define $B(\gr)$ to be the scheme 
        \begin{multline*}
        \qquad \quad \widetilde{Ba}_{i-k}(\widetilde{z}_{i-k})_{S_{i-k}}^{\mathrm{I}} \times \widetilde{Ba}_{i-k+1}(\widetilde{z}_{i-k+1})_{S_{i-k+1}} \times \\
        \dots \times \widetilde{Ba}_{i-1}(\widetilde{z}_{i-1})_{S_{i-1}} \times \widetilde{Ba}_{i}(\widetilde{z}_{i})_{S_{i}}^{\mathrm{II}}.
        \end{multline*}
        \item Define $Y(\gr)$ to be the closed subscheme of $B(\gr)$ obtained by setting $l_{j-1} = r_{j}$ for each $j \in \{i-k+1, \dots, i\}$. Denote by $\Delta(\gr)$ the map $Y(\gr) \hookrightarrow B(\gr)$.
        \item Let 
$$Z_{\tilB}(\gr) \stackrel{\text{def}}{=} \GL_2 \times (Z_{i-k+1} \times \GL_2) \times \dots \times (Z_{i-1} \times \GL_2) \times Z_{i}.$$ Recall that $Z_j$ is the scheme--theoretic image of $\Baj_{S_j}$ under $\pr_j$, and by Lemma \ref{lem:im-is-product}, $Z_j \times \GL_2$ is isomorphic to the the scheme--theoretic image of $\tilBaj_{S_j}$ under $\widetilde{\pr}_j$.
    \end{enumerate}
    \end{definition}
    
    It is clear from the definitions that as long as $\trn$ is non--empty, there exist obvious isomorphisms
\begin{align}
    \tilBa_S \xrightarrow{\sim} \prod_{\gr \in \trn} B(\gr), \\
    \tilYz_S \xrightarrow{\sim} \prod_{\gr \in \trn} Y(\gr) \label{eqn:Y-as-prod}
\end{align}
which induce an identification of the scheme--theoretic image of $\tilBa_S$ under $\pr_{\tilB}$, with the scheme $\prod_{\gr \in \trn} Z_{\tilB}(\gr)$. Thus, the map $\pr_{\tilB}$ induces proper and scheme--theoretically dominant maps $\pr(\gr): B(\gr) \to Z_{\tilB}(\gr)$. Define $\cZ(\gr)$ to be the scheme--theoretic image of $Y(\gr)$ under $\pr(\gr)$. There exists an obvious identification of $\prod_{\gr \in \trn} \cZ(\gr)$ with $\tilZz_S$.

\begin{definition}\label{defn:I}
    For $j$ with $T_j = 3$, let $N_j$ be the ideal of $\Gamma(\tilBaj_{S_j})$ generated by the functions $B, C, D$. There exists an obvious inclusion $$\Gamma(\tilBaj_{S_j}) \hookrightarrow \Gamma(\tilBa_S)$$ and if $\gr = (i-k, \dots, i) \in \trn$ and $j \in \{i-k, \dots, i\} \smallsetminus \{i-k, i\}$, then also an inclusion $$\Gamma(\tilBaj_{S_j}) \hookrightarrow \Gamma(B(\gr)).$$
    Abusing notation, we let the ideals generated by the image of $N_j$ under these inclusions also be denoted $N_j$.
    
    Let $N$ (resp. $N(\gr)$) be the minimal ideal of $\Gamma(\tilBa_S)$ (resp. of $\Gamma(B(\gr))$) containing $N_j$ for each $j$ with $T_j=3$ (resp. for each $j \in \{i-k, \dots, i\} \smallsetminus \{i-k, i\}$). We also let the image of $N(\gr)$ under the obvious inclusion $$\Gamma(B(\gr)) \hookrightarrow \Gamma(\tilBa_S)$$ be denoted by $N(\gr)$.    
    % , which we also denote by $N_j$, for each $j$ with $T_j = 3$ under the obvious inclusion
    % $$\Gamma(\tilBaj_{S_j}) \hookrightarrow \Gamma(\tilBa_S).$$
    % Suppose $\gr = (i-k, \dots, i) \in \trn$. Let $N(\gr)$ %(resp. $P(\gr)$), 
    % denote the minimal %(resp. maximal) 
    % ideal of $\Gamma(B(\gr))$ that contains %(resp. is contained in) 
    % the ideal gener3.11ated by the image of $N_j$, again also denoted by $N_j$, for each $j \in \{i-k, \dots, i\} \smallsetminus \{i-k+1, i\}$ under the obvious inclusion  
    % $$\Gamma(\tilBaj_{S_j}) \hookrightarrow \Gamma(B(\gr)).$$ Abusing notation, we let the ideal of $\Gamma(\tilBa_S)$ generated by the image of $N(\gr)$ %and $P(\gr)$ 
    % under the inclusion $$\Gamma(B(\gr)) \hookrightarrow \Gamma(\tilBa_S)$$ also be denoted 
    % by $N(\gr)$.
    % and $P(\gr)$ respectively.
    % $\Gamma(\tilBa_S)$ (resp. $\Gamma(B(\gr))$ for $\gr = (i-k, \dots, i) \in \trn$) generated by pullbacks of the functions $B, C, D$ along the projection to $\tilBaj_{S_j}$ (resp. if $j \in \{i-k+1, \dots, i-1\}$). Let $N$ (resp. $N(\gr)$) be the minimal ideal containing each $N_j$. Abusing notation, we denote the ideal generated by the image of $N(\gr)$ under the obvious map $$\Gamma(B(\gr)) \to \Gamma(\tilBa_S)$$ also by $N(\gr)$.
\end{definition}

% \begin{remark}\label{rem:affine}
%     When $T_j =3$, the vanishing locus of $N_j$ in $\tilBaj_{S_j}$ is the vanishing locus of $B$ in $D(y)$ and of $C$ in $D(x)$. It follows that the open immersions
%     \begin{align*}
%         &\tilBa_{S} \smallsetminus V(N) \hookrightarrow \tilBa_{S} \quad \text{ and} \\
%         &\tilYz_S \smallsetminus V(N) \hookrightarrow \tilYz_S
%     \end{align*}
%     are relatively affine.
% \end{remark}

\begin{lemma}\label{lem:bad-sequence}
    Let $\gr = (i-k, \dots, i) \in \trn$. 
    \begin{enumerate}
        \item If $(T_{i-k}, T_i) \not\in \{(1,1), (1, 5), (2, 1), (2, 5)\}$, then the map $\pr(\gr)$ induces an isomorphism of $Y(\gr)$ with $\cZ(\gr)$. 
    \item Suppose $T_{i-k} \in \{1, 2\}$ and $T_{i} \in \{1, 5\}$. If $l(\gr) = 2$, then the map $Y(\gr) \to \cZ(\gr)$ induced by $\pr(\gr)$ is a $\PP^{1}$--torsor. On the other hand, if $l(\gr) > 2$, then $\dim \Y(\gr) = \dim \cZ(\gr)$, the restriction of $\pr(\gr)$ to $\pr(\gr)^{-1} \left(\cZ(\gr) \smallsetminus V(N(\gr))\right)$ is a monomorphism, and $\pr(\gr)$ induces a birational map from $Y(\gr)$ to $\cZ(\gr)$. 
    % the dimension of $Y(\gr)$ and $\cZ(\gr)$ are the same if and only if $l(\gr) > 2$. When that happens, the restriction of $\pr(\gr)$ to $\pr(\gr)^{-1} \left(\cZ(\gr) \smallsetminus V(N(\gr))\right)$ is a monomorphism, and $\pr(\gr)$ induces a birational map from $Y(\gr)$ to $\cZ(\gr)$.
    \end{enumerate} 
\end{lemma}
\begin{proof}
Since $\pr(\gr)$ is proper, it suffices to show that it restricts to a monomorphism on $Y(\gr)$ unless $T_{i-k} \in \{1, 2\}$ and $T_{i} \in \{1, 5\}$. Unpacking the definitions, we find that $\pr(\gr)$ restricts to a monomorphism if and only if given the data of $\{X_j\}_{j = i-k+1}^{i}, \{\kappa_j\}_{j=i-k}^{i-1}$, the tuples $\{l_j\}_{j = i-k}^{i-1}$ and $\{r_j\}_{j = i-k+1}^{i}$ are uniquely determined after imposing the conditions $l_{j-1} = r_{j}$ for $j \in \{i-k+1, \dots, i\}$. We make the following observations:
\begin{itemize}
    \item When $T_j = 1$, the data of $X_j, \kappa_j$ imposes no constraints on $l_j, r_j$, which are unrelated to each other.
    \item When $T_j = 2$, the data of $X_j, \kappa_j$ imposes no constraints on $l_j$, but completely determines $r_j$.
    \item When $T_j=3$, $X_j, \kappa_j$ do not completely determine $l_j$ or $r_j$ (except away from the vanishing locus of $N_j$), but $l_j$ and $r_j$ determine each other completely.
    \item When $T_j = 4$, $\kappa_j$ determines $l_j$ and $X_j$ determines $r_j$.
    \item When $T_j = 5$, $\kappa_j$ determines $l_j$ but the data of $X_j, \kappa_j$ imposes no constraints on $r_j$.
\end{itemize}
From these, the first part of the statement of the Lemma follows immediately. For the second part, suppose $\gr = (i-k, \dots, i)$ with $(T_{i-k}, T_{i}) \in \{(1,1), (1,5), (2, 1), (2,5)\}$. If $l(\gr)=2$, then since $l_{i-k} = r_{i} \in \PP^{1}$ can take any value, $Y(\gr)$ is a $\PP^{1}$--torsor over $\cZ(\gr)$. Finally, the third bullet above implies that the restriction of $\pr(\gr)$ to $\pr(\gr)^{-1}\left(\cZ(\gr) \smallsetminus V(N(\gr))\right)$ is a monomorphism. If $l(\gr) > 2$, $\cZ(\gr) \smallsetminus V(N(\gr))$ is readily seen to be non--empty, implying that $\pr(\gr))$ induces a birational map from $Y(\gr)$ to $\cZ(\gr)$.
\end{proof}

By the same argument as above, we also obtain the following lemma.
\begin{lemma}\label{lem:birational-T3}
    Suppose $T_j =3$ for each $j \in \Z/f\Z$. The map $\widetilde{\pi}_S$ restricts to a monomorphism on $\widetilde{\pi}^{-1}_S\left(\tilZz_S \smallsetminus V(N) \right)$.
\end{lemma}

\begin{definition} Let $\trn^{*} \subset \trn$ be the set of those $\gr = (i-k, \dots, i) \in \trn$ satisfying $T_{i-k} \in \{1, 2\}$ and $T_{i} \in \{1, 5\}$.
\end{definition}

\begin{lemma}\label{lem:Baj}(Version of \cite[Lem.~4.2.3]{lhmm}) Let $S \in \{L, R\}^{\Z/f\Z}$ and $\gr \in \trn$.
\begin{enumerate}
        \item  The schemes $\Baj_{S_j}$ and $B(\gr)$ are local complete intersections over $\F$. Whenever non--empty, $\Baj_{S_j}$ has dimension $2$. 
    \item When $T_j = 3$, the dualizing sheaf of $\Baj_{S_j}$ is $\cO_{\Baj_{S_j}}$. Thus, if $\gr = (i-k, \dots, i)$, the dualizing sheaf of $B(\gr)$ is $$p_{i-k}^{*} \cO(-2) \otimes q_{i}^{*} \cO(-2).$$
    %\item The map $\O_{\Uzj} \to \pr_{j*} \O_{\Baj_{S_j}}$ is surjective. 
    % \item  The schemes $\Baj_{S_j}$ and $B(\gr)$ are local complete intersections over $\F$. Whenever non--empty, $\Baj_{S_j}$ has dimension $2$. 
    % \item When $T_j = 3$, the dualizing sheaf of $\Baj_{S_j}$ is $p^*_j \O(-1) \otimes q^{*}_j \O(-1)$. The dualizing sheaf of $B(\gr)$ is $$
    
    % In all other cases, the dualizing sheaf of $\Baj_{S_j}$ is $ p^*_j \O(-2) \otimes q^{*}_j \O(-2)$.
    % \item The map $\O_{\Uzj} \to \pr_{j*} \O_{\Baj_{S_j}}$ is surjective. 
\end{enumerate}
\end{lemma}
\begin{proof}
First, suppose $T_j = 3$. Let $\iota \colon \Baj_{S_j} \hookrightarrow \PP^1 \times \AA^3 $ be the embedding given by $([x:y], (B, C, D)) \mapsto ([x:y], (B, C, D))$. Set $t = x/y$. Below is an open cover of $\PP^1 \times \AA^3$ along with local generators of the ideal sheaf $\cI$ whose vanishing locus gives $\Baj_{S_j}$:
\begin{align*}
    &D(y) =  \Spec \F[B, C,D,t], &&\cI(D(y)) = (D t - C, B t + D);\\
    &D(x) = \Spec \F[B, C, D, t^{-1}], && \cI(D(x))=(D - C t^{-1}, B + D t^{-1}). 
\end{align*}

The given generators clearly describe a regular sequence on each chart, and so the first statement holds in this case. The sheaf $\bigwedge^{2} \iota^{*} \cI$ is an invertible sheaf freely generated on $D(y) \cap \Baj_{S_j}$ by $(D t - C) \wedge (B t + D)$, and on $D(x) \cap \Baj_{S_j}$ by $(D - C t^{-1})\wedge (B + D t^{-1})$. With these generators, the transition map from $D(y) \cap \Baj_{S_j}$ to $D(x) \cap \Baj_{S_j}$ is given by multiplication by $t^2$. Hence, $\wedge^{2} (\iota^{*} \cI)^{\vee} \cong p^*_j \O(2)$, implying that the dualizing sheaf of $\Baj_{S_j}$ is $p^*_j \O(2)  \otimes p^*_j \O(-2)  \cong \cO_{\Baj_{S_j}}$.

% The first and second statements are immediate when $T_j \in \{1, 2, 4, 5\}$. When $T_j = 3$, let $\iota \colon \Baj_{S_j} \hookrightarrow \PP^1 \times \AA^3 $ be the embedding given by $([x:y], (B, C, D)) \mapsto ([x:y], (B, C, D))$. Set $s = x/y$. Below is an open cover of $\PP^1 \times \AA^3$ along with local generators of the ideal sheaf $\cI$ whose vanishing locus gives $\Baj_{S_j}$:
% \begin{align*}
%     &D(y) =  \Spec \F[B, C,D,s], &&\cI(D(y)) = (D s - C, B s + D);\\
%     &D(x) = \Spec \F[B, C, D, s^{-1}], && \cI(D(x))=(D - C s^{-1}, B + D s^{-1}). 
% \end{align*}

% The given generators clearly describe a regular sequence on each chart, and so the first statement holds in this case. The sheaf $\bigwedge^{2} \iota^{*} \cI$ is an invertible sheaf freely generated on $D(y) \cap \Baj_{S_j}$ by $(D s - C) \wedge (B s + D)$, and on $D(x) \cap \Baj_{S_j}$ by $(D - C s^{-1})\wedge (B + D s^{-1})$. With these generators, the transition map from $D(y) \cap \Baj_{S_j}$ to $D(x) \cap \Baj_{S_j}$ is given by multiplication by $s^2$. Hence, $\wedge^{2} (\iota^{*} \cI)^{\vee} \cong \O(2) \cong p^*_j\O(1) \otimes q^*_j \O(1)$. Thus, we obtain that the dualizing sheaf of $\Baj_{S_j}$ is $p^*_j \O(1)  \otimes q^*_j \O(1) \otimes p^*_j \O(-2)  \otimes q^*_j \O(-2) =  p^*_j \O(-1)  \otimes q^*_j \O(-1)$ as desired.
The rest of the first and second statements is now immediate.
%For the third statement, we note from Table \ref{table:shape-L,table:shape-R} that $\pr_{j*}\O_{\Baj_{S_j}} = \O_{Z}$, where $Z$ is the scheme--theoretic image of $\Baj_{S_j}$. 
\end{proof}

\begin{lemma}\label{lem:Rpr-dualizing}(Version of \cite[Lem.~4.2.5]{lhmm})
Let $\epsilon_j, \delta_j \in \Z$ and $$\cF :=  p_j^{*} \O_{\PP^1}(-1)^{\delta_j} \otimes_{\O_{\Baj_{S_j}}} q_j^{*} \O_{\PP^1}(-1)^{\epsilon_j} $$ be a sheaf on $\Baj_{S_j}$.

Suppose $T_j \neq 3$ and $\epsilon_j, \delta_j \in \{0, 1\}$. Then the cohomology groups of $\cF$ admit the following descriptions as $\Gamma(\Baj_{S_j})$--modules:
\begin{enumerate}
    \item If $T_j = 1$, then 
    \begin{align*}
        R^{i} \Gamma \cF \cong \begin{cases}
            \Gamma(\Baj_{S_j}) &\text{ if } i=0, \epsilon_j=\delta_j=0; \\
            % \Gamma(\Baj_{S_j}) &\text{ if } i=1; (\epsilon_j, \delta_j) \in \{(0, 2), (2,0)\}\\
            % \Gamma(\Baj_{S_j}) &\text{ if } i=2; \epsilon_j=\delta_j=2 \\
            0 &\text{ otherwise}.
        \end{cases}
    \end{align*}
    \item If $T_j = 2$, then
    \begin{align*}
        R^{i} \Gamma \cF \cong \begin{cases}
            \Gamma(\Baj_{S_j}) &\text{ if } i=0, \epsilon_j \in \{0, 1\}, \delta_j=0; \\
            % \Gamma(\Baj_{S_j}) &\text{ if } i=1; \epsilon_j \in \{0, 1, 2\}, \delta_j=2, \\
            0 &\text{ otherwise}.
        \end{cases}
    \end{align*}
    \item If $T_j = 4$, then
    \begin{align*}
        R^{i} \Gamma \cF \cong \begin{cases}
            \Gamma(\Baj_{S_j}) &\text{ if } i=0, \epsilon_j, \delta_j \in \{0, 1\}; \\
            0 &\text{ otherwise}.
        \end{cases}
    \end{align*}
        \item If $T_j = 5$, then
    \begin{align*}
        R^{i} \Gamma \cF \cong \begin{cases}
            \Gamma(\Baj_{S_j}) &\text{ if } i=0, \epsilon_j=0, \delta_j \in \{0, 1\}; \\
            % \Gamma(\Baj_{S_j}) &\text{ if } i=1; \epsilon_j=2, \delta_j \in \{0, 1, 2\}, \\
            0 &\text{ otherwise}.
        \end{cases}
    \end{align*}
\end{enumerate}
Suppose $T_j = 3$ and $\epsilon_j, \delta_j \in \{-1, 0, 1\}$. The following are true:
\begin{enumerate}
    \item If $\epsilon_j + \delta_j = 0$, then $R^{i}\Gamma\cF = 0$ for $i \neq 0$, and
        $$R^{0}\Gamma \cF \cong
            \Gamma(\Baj_{S_j}). $$
    \item  If $\epsilon_j + \delta_j \in \{-1, 1\}$, then $R^{i}\Gamma \cF = 0$ for  $i \neq 0$, and
 $$R^{0}\Gamma \cF \cong
            \Gamma(\Baj_{S_j})[e_1, e_2]/(De_1 - Ce_2, Be_1 + De_2).$$
    \item If $\epsilon_j + \delta_j =-2$, then $R^{i}\Gamma \cF = 0$ for  $i \neq 0$, and $R^{0}\Gamma \cF \cong$
    $$\Gamma(\Baj_{S_j})[e_1, e_2, e_3]/(De_1 - Ce_2, Be_1 + De_2, Be_1 + Ce_3).$$
    \item Finally, if $\epsilon_j + \delta_j =2$, then 
    \begin{align*}
        R^{i}\Gamma \cF \cong \begin{cases}
            \Gamma(\Baj_{S_j}) &\text{ if } i =0, \\
            \Gamma(\Baj_{S_j})/N_j &\text{ if } i =1, \\
            0 &\text{ if } i \not\in \{0,1\}.
        \end{cases}
    \end{align*}
\end{enumerate}

\end{lemma}
\begin{proof}
The proof for classes \ref{Baj-P1XP1}, \ref{Baj-P1XA1}, \ref{Baj-A2} and \ref{Baj-P1XA1-constant-pj} follows from the following observations:
\begin{itemize}
    \item $R\Gamma (\AA^{n}, \cO)$ is concentrated in degree $0$.
    \item $R\Gamma (\PP^{1}, \cO(-n))$ is zero if $n=1$ and free of rank $1$ over $\Gamma(\PP^{1}, \cO) = \F$ concentrated in degree $0$ if $n=0$.
    % non--zero concentrated in degree $0$ if $n=0$ and zero if $n=1$, and non--zero concentrated in degree $1$ if $n=2$. Whenever non--zero, the cohomology is free of rank $1$ over $\Gamma(\PP^{1}, \cO) = \F$.
\end{itemize}
In particular, by K\"unneth formula, we have: 
\begin{itemize}
    \item When $T_j = 1$, $
        R^{i} \Gamma \cF = \oplus_{m+n=i} H^{m}(\PP^{1}, \cO(-\delta_j)) \otimes H^{n}(\PP^{1}, \cO(-\epsilon_j))$.
    \item When $T_j = 2$, $
         R^{i} \Gamma \cF = \oplus_{m+n=i} H^{m}(\PP^{1}, \cO(-\delta_j)) \otimes H^{n}(\AA^{1}, \cO).$
    \item When $T_j = 4$, $
         R^{i} \Gamma \cF = \oplus_{m+n=i} H^{m}(\AA^{1}, \cO) \otimes H^{n}(\AA^{1}, \cO).$
    \item When $T_j = 5$,
         $R^{i} \Gamma \cF = \oplus_{m+n=i} H^{m}(\AA^{1}, \cO) \otimes H^{n}(\PP^{1}, \cO(-\epsilon_j)).$
\end{itemize}

Finally, suppose $T_j = 3$. Letting $t$ denote $x/y$, $\Baj_{S_j}$ admits an open cover by schemes $\Spec \F[t^{-1}, C]$ and $\Spec \F[t, B]$. 
Let $\iota_1: \F[t^{-1}, C] \hookrightarrow \F[t^{\pm}, C]$ be the obvious inclusion, and let $\iota_2: t^{\epsilon + \delta} \F[t, B] \hookrightarrow \F[t^{\pm}, C]$ be given by $t^{n} \mapsto t^{n}, B \mapsto -Ct^{-2}$.
Therefore, $R \pr_{j*}\cF \cong R \pr_{j*} (p_j^{*} \O_{\PP^1}(-1)^{\delta_j + \epsilon_j})$ is computed by the \v{C}ech complex
    $$\F[t^{-1}, C] \oplus t^{\delta_j + \epsilon_j}\F[t, B] \to \F[t^{\pm}, C]$$ where the differential maps $(f, g) \mapsto \iota_1(f) - \iota_2(g)$.  This complex is quasi-isomorphic to the following complex of $\F[B, C, D]/(D^{2} + BC)$ modules:
    \begin{align}\label{eqn:T3-coh-complex}
        t^{\delta_j + \epsilon_j} \F[t, Ct^{-2}] \to \F[t^{\pm}, C]/\F[t^{-1}, C],
    \end{align} where $B$ acts via multiplication by $-Ct^{-2}$, $C$ acts via multiplication by $C$, and $D$ acts via multiplication by $Ct^{-1}$. Note that, as $\F$--vector spaces, 
    \begin{align*}
        t^{\epsilon_j + \delta_j}\F[t, Ct^{-2}] &\cong \bigoplus_{\substack{m \geq 0, \\ n \geq -2m + \delta_j + \epsilon_j}} \F C^{m}t^{n}, \text{ and} \\
         \F[t^{\pm}, C]/\F[t^{-1}, C] &\cong \qquad \bigoplus_{\substack{m \geq 0, \\ n > 0}} \F C^{m}t^{n}.
    \end{align*}
We make the following observations about the kernel and cokernel of the map in (\ref{eqn:T3-coh-complex}).
\begin{itemize}
    \item  When $\delta_j + \epsilon_j = 0$, the kernel is free of rank $1$ over $\F[B, C, D]/(D^{2} + BC)$, which thus gives the global functions, while the cokernel is $0$. 
    \item  When $\delta_j + \epsilon_j = -1$, the kernel is generated over $\F[B, C, D]/(D^{2} + BC)$ by $\{1, t^{-1}\}$ which satisfy the relations $D \cdot 1 - C \cdot t^{-1} = 0$ and $B \cdot 1 + D \cdot t^{-1} = 0$. The cokernel is clearly $0$.
    \item  When $\delta_j + \epsilon_j = 1$, the kernel is generated over $\F[B, C, D]/(D^{2} + BC)$ by $\{C, Ct^{-1}\}$ which satisfy the relations $D \cdot C - C \cdot Ct^{-1} = 0$ and $B \cdot C + D \cdot Ct^{-1}$. The cokernel is clearly $0$. 
    \item  When $\delta_j + \epsilon_j = -2$,  the kernel is generated over $\F[B, C, D]/(D^{2} + BC)$ by $\{1, t^{-1}, t^{-2}\}$ which satisfy the relations $D \cdot 1 - C \cdot t^{-1} = 0$, $B \cdot 1 + D \cdot t^{-1} = 0$ and $B \cdot 1 + C \cdot t^{-2} = 0$. The cokernel is clearly $0$.
    \item When $\delta_j + \epsilon_j = 2$, the kernel is free of rank $1$ over $\F[B, C, D]/(D^{2} + BC)$ with generator $C$. A generator of the cokernel is given by the class of $t$ while $B \cdot t, C \cdot t, D \cdot t$ give the zero class in the cokernel. In particular, the cokernel is isomorphic to $\F$ and supported at the vanishing locus of $N_j$.
\end{itemize}
The assertions in the Lemma follow immediately.
\end{proof}

\begin{remark}\label{rem:identify-sections}
    In the proof of Lemma \ref{lem:Rpr-dualizing} above when $T_j=3$, one verifies immediately that when $\delta_j + \epsilon_j = -1$, the pullbacks of the sections $x$ and $y$ of $\cO_{\PP^{1}}(1)$ along $p_j$ ($= q_j$) are $1$ and $t^{-1}$ respectively. When $\delta_j + \epsilon_j = -2$, the pullbacks of the sections $x^2$, $xy$ and $y^{2}$ of $\cO_{\PP^1}(2)$ along $p_j$ are $1$, $t^{-1}$ and $t^{-2}$ respectively.
\end{remark}

\begin{lemma}\label{lem:sch-th-image}
    The scheme--theoretic image of $\tilBa_S$ is a closed subscheme of $\widetilde{U}(\tilz)$ that identifies naturally with $$\relSpec \pr_{\tilB *}(\cO_{\tilBa_S}) \cong \Spec \Gamma(\tilBa_S).$$ Similarly, for $\gr \in \trn$, $Z_{\tilB}(\gr)$ identifies naturally with $$\relSpec \pr(\gr)_{*} \cO_{B(\gr)} \cong \Spec \Gamma(B(\gr)).$$ 
\end{lemma}
\begin{proof}
%Since the codomains of $\pr_{\tilB}$ and $\pr(\gr)$ are affine, higher pushforwards along these maps are computed by higher cohomologies for the global sections functor.
The only thing to check is that the global functions on the scheme--theoretic images of $\tilBa_S$ and $B(\gr)$ are the same as the global functions on $\tilBa_S$ and $B(\gr)$ respectively. Tables \ref{table:shape-L} and \ref{table:shape-R} give descriptions of global functions on the scheme--theoretic image of $\Baj_{S_j}$ under $\pr_j$ and Lemma \ref{lem:Rpr-dualizing} gives a description of $\Gamma({\Baj_{S_j}})$. The descriptions match and so, an application of Lemma \ref{lem:im-is-product} finishes the proof. 
%that allow computations of the global functions on the scheme--theoretic image of $\tilBa_S$. Lemma \ref{lem:Rpr-dualizing} gives a description of $\pr_{\tilB*} \O_{\tilBa_S}$. Comparing the two shows that the natural map from the global functions on the scheme--theoretic image of $\tilBa_S$ to $\pr_{\tilB*} \O_{\tilBa_S}$ is an isomorphism
\end{proof}

% More statements from \cite[Sec.~4.2]{lhmm} apply nearly identically when working mod $p$ and attaching a subscript for shape. We restate these without proof, because the proofs are exactly those presented in \cite{lhmm}.
\begin{lemma}\label{lem:structure-sh-coh-vs-I}(Version of \cite[Cor.~4.2.10]{lhmm})
    Denoting the restriction of $\pr_{\tilB}$ to $\tilBa_S$ also by $\pr_{\tilB}$, consider the commutative diagram
\begin{center}
\begin{tikzcd}
    \tilYz_{S} \arrow[rd, "\pr"] \arrow[r, hook, "\Delta_S"] & \tilBa_S \arrow[d, "\pr_{\tilB}"] \\
    & \Uz
\end{tikzcd}
\end{center}
The following are true:
\begin{enumerate}
    \item Whenever $i>0$, $R^{i} \pr_{\tilB *} \O_{\tilBa_S} = 0$. 
    \item The map $\O_{\Uz} \to \pr_{\tilB*} \O_{\tilBa_S}$ is a surjection.
    \item If $\cI(\tilz)$ is the ideal sheaf defining the closed immersion $\Delta_S$, then
    \begin{itemize}
    \item $\coker(\O_{\Uz} \to \pr_{*}\O_{\tilYz_S}) = R^{1}\pr_{\tilB *} \cI(\tilz)$, and
    \item $R^{i}\pr_{*} \O_{\tilYz_S} = R^{i+1} \pr_{\tilB*} \cI(\tilz)$ for $i>0$.
\end{itemize}
\end{enumerate}

% , and let $\tilZ_S$ be the scheme--theoretic image of $\tilY_S$ under $\pr$. The following are true:

\end{lemma}
\begin{proof}
    Vanishing of $R^{i} \pr_{\tilB *} \O_{\tilBa_S} = 0$ if $i > 0$ follows from Lemma \ref{lem:Rpr-dualizing}. %Table \ref{table:shape-L,table:shape-R} give descriptions of the global functions on the scheme--theoretic image of $\tilBa_S$ and Lemma \ref{lem:Rpr-dualizing} gives a description of $\pr_{\tilB*} \O_{\tilBa_S}$. Comparing the two shows that the natural map from the global functions on the scheme--theoretic image of $\tilBa_S$ to $\pr_{\tilB*} \O_{\tilBa_S}$ is an isomorphism. 
    The map $\O_{\Uz} \to \pr_{\tilB*} \O_{\tilBa_S}$ is a surjection by Lemma \ref{lem:sch-th-image}. Finally, the long exact sequence in cohomology obtained from the short exact sequence
    $$0 \to \cI(\tilz) \to \cO_{\tilBa_S} \to (\Delta_S)_{*}\cO_{\tilYz_S} \to 0$$ of sheaves on $\tilBa_S$ implies the remaining statements.
\end{proof}

\begin{lemma}\label{lem:koszul}(Version of \cite[Lem.~4.2.11]{lhmm})
 For each $j \in \Z/f\Z$, there exists a map $$\mathfrak{s}_j: q_j^{*} \O_{\PP^1}(-1) \otimes_{\O_{\tilBa_S}} p_{j-1}^{*} \O_{\PP^1}(-1) \longrightarrow \O_{\tilBa_S},$$
so that $\tilYz_S$ is a complete intersection defined by the zero locus of $\{\mathfrak{s}_j\}_{j \in \Z/f\Z}$.
If $\gr = (i-k, \dots, i) \in \trn$, $\mathfrak{s}_{j}$ for $j \in \{i-k+1, \dots, i\}$ can be viewed as a map $$q_j^{*} \O_{\PP^1}(-1) \otimes_{\O_{B(\gr)}} p_{j-1}^{*} \O_{\PP^1}(-1) \longrightarrow \O_{B(\gr)}$$ and $Y(\gr)$ is a complete intersection defined by the zero locus of $\mathfrak{s}_{i-k+1}, \dots, \mathfrak{s}_{i}$.
\end{lemma}
\begin{proof}
    Proof for $\tilYz_S$ is identical to that of \cite[Lem.~4.2.11]{lhmm}. Since $\dim \tilYz_S = \sum_{\gr \in \trn} \dim Y(\gr)$ whenever $\trn$ is non--empty, $Y(\gr)$ is forced to be a complete intersection (in fact, smooth by Lemma \ref{lem:smooth-irred}) as well.
\end{proof}

\begin{corollary}\label{cor:ss}
Let $$\cE \stackrel{\text{def}}{=}\bigoplus_{j \in \Z/f\Z} q_j^{*} \O_{\PP^1}(-1) \otimes_{\O_{\tilBa_S}} p_{j-1}^{*} \O_{\PP^1}(-1)$$ and for $\gr = (i-k, \dots, i) \in \trn$, let $$\cE(\gr) \stackrel{\text{def}}{=}\bigoplus_{j = i-k+1}^{i} q_j^{*} \O_{\PP^1}(-1) \otimes_{\O_{B(\gr)}} p_{j-1}^{*} \O_{\PP^1}(-1).$$
The Koszul resolutions $\mathrm{Kos}_{\bullet}\bigl(\cE, (\mathfrak{s}_j)_{j \in \Z/f\Z}\bigr)$ and $\mathrm{Kos}_{\bullet}\bigl(\cE(\gr), (\mathfrak{s}_j)_{j = i-k+1}^{i}\bigr)$ yield exact sequences 
\begin{align*}
&0 \to \bigwedge^{f} \cE \to \bigwedge^{f-1} \cE \dots \to \bigwedge^{1} \cE \to \cI(\tilz) \to 0,\\
&0 \to \omega_{\tilBa_S} \to \omega_{\tilBa_{S}} \otimes \bigwedge^{1} \cE^{\vee} \to \dots \to \omega_{\tilBa_S} \otimes \bigwedge^{f} \cE^{\vee} \to (\Delta_S)_{*}\omega_{\tilYz_S} \to 0,\\
&0 \to \omega_{B(\gr)} \to \omega_{B(\gr)} \otimes \bigwedge^{1} \cE(\gr)^{\vee} \to \dots \to \omega_{B(\gr)} \otimes \bigwedge^{l(\gr)-1} \cE(\gr)^{\vee} \to \Delta(\gr)_{*}\omega_{Y(\gr)} \to 0.\end{align*}
In particular, there exist three cohomological spectral sequences living in the second quadrant given by 
\begin{align*}
    &E_1^{p, q} = R^{q} \Gamma \left( 
        \bigwedge^{1-p} \cE \right) \Longrightarrow R^{p+q} \Gamma \left(\cI(\tilz)\right), \\
    &E_1^{p, q} = R^{q} \Gamma \left( 
        \omega_{\tilBa_S} \otimes 
        \bigwedge^{f+p} \cE^{\vee} \right) \Longrightarrow R^{p+q} \Gamma \left(\omega_{\tilYz_S}\right), \text{ and}\\
    &E_1^{p, q} = R^{q} \Gamma \left( 
        \omega_{B(\gr)} \otimes 
        \bigwedge^{l(\gr)-1+p} \cE(\gr)^{\vee} \right) \Longrightarrow R^{p+q} \Gamma \left(\omega_{Y(\gr)}\right).
\end{align*}

% \begin{align*}
%     &R^{k+i-1} \pr_{\tilB *} \left( 
%         \bigwedge^{k} \cE \right) \Longrightarrow R^{i} \pr_{\tilB *} \cI(\tilz), \text{ and}\\
%     &R^{f-k+i} \pr_{\tilB *} \left( 
%         \omega_{\tilBa_S} \otimes 
%         \bigwedge^{k} \cE^{\vee} \right) \Longrightarrow R^{i} \pr_{*} \omega_{\tilYz_S}.
% \end{align*}

\end{corollary}
\begin{proof}
    The first exact sequence follows immediately from Lemma \ref{lem:koszul}. For the second and third, note that for each $j \in \Z/f\Z$, resp. each $j \in \{i-k+1, \dots, i\}$, $q_j^{*} \O_{\PP^1}(-1) \otimes p_{j-1}^{*} \O_{\PP^1}(-1)$ is the ideal sheaf of an effective Cartier divisor. Further, $\tilYz_S$ is irreducible being a $\GL_2^{\Z/f\Z}$--torsor over the irreducible component $\Y_S$ of $\Y_{\F}$, and therefore, so is $Y(\gr)$. Thus, \cite[Prop.~6.10(iii)]{kovacs2017rational} gives the remaining exact sequences. \end{proof}

\begin{lemma}\label{lem:coh-L1}
    The following are true:
    \begin{enumerate}
        \item (Version of \cite[Cor.~4.2.12]{lhmm}) For $0 < a \leq f$, $b \geq a$, $R^{b} \Gamma \left( 
        \bigwedge^{a} \cE \right) \neq 0$ if and only if $a=b=f$ and $T_j =3$ for each $j \in \Z/f\Z$. In this special case, $$R^{f} \Gamma \left( 
        \bigwedge^{f} \cE \right) \cong \Gamma(\tilBa_S)/N.$$

        \item Suppose $T_j = 3$ for each $j \in \Z/f\Z$. Then 
            \begin{align*}
                R^{b} \Gamma \left( 
        \omega_{\tilBa_S} \otimes 
        \bigwedge^{a} \cE^{\vee} \right) \neq 0 \iff 0 \leq a \leq f, b=0. \end{align*}
        In particular, $R^{0} \Gamma(\omega_{\tilYz_S})$ is the cokernel of the map
     \begin{align}\label{eqn:E1-map-T3}
        R^{0}\Gamma 
        \left( \bigwedge^{f-1} \cE^{\vee} \right) \longrightarrow R^{0}\Gamma
        \left( \bigwedge^{f} \cE^{\vee}\right)
    \end{align}
        in the $E_1$ page of the second spectral sequence in Corollary \ref{cor:ss}.
        
        \item Suppose $\gr = (i-k, \dots, i) \in \trn^{*}$ with $l(\gr) \geq 3$. Then 
            \begin{align*}
                R^{b} \Gamma \left( 
        \omega_{B(\gr)} \otimes 
        \bigwedge^{a} \cE(\gr)^{\vee} \right) \neq 0 \iff 0 \leq a \leq l(\gr) - 3, b=2.
        \end{align*}
        In particular, $R^{0} \Gamma(\omega_{Y({\gr})})$ is the cokernel of the map
     \begin{align}\label{eqn:E1-map}
        R^{2}\Gamma \left( \omega_{B(\gr)} \otimes 
        \bigwedge^{l(\gr)-4} \cE(\gr)^{\vee} \right) \longrightarrow R^{2}\Gamma \left( \omega_{B(\gr)} \otimes 
        \bigwedge^{l(\gr)-3} \cE(\gr)^{\vee} \right)
    \end{align}
        in the $E_1$ page of the third spectral sequence in Corollary \ref{cor:ss}.

        \end{enumerate}
\end{lemma}
\begin{proof} 
First, note that
\begin{align*}
\bigwedge^{a} \cE = \bigoplus_{\underline{\epsilon}} \left( \bigotimes_{j\in \Z/f\Z} q_j^{*} \O(-1)^{\epsilon_j} \otimes_{\O_{\tilBa_S}} p_{j}^{*} \O(-1)^{\epsilon_{j+1}} \right)
\end{align*}
%\bigoplus_{\substack{\underline{\epsilon} \in \{0,1\}^\Z/f\Z, \\ \sum \epsilon_j = a}}
where the direct sum is over the set $\{\underline{\epsilon} \in \{0,1\}^{\Z/f\Z} \mid \sum \epsilon_j = a\}$.
By Lemma \ref{lem:Rpr-dualizing}, $$R\Gamma \left( q_j^{*} \O(-1)^{\epsilon_j} \otimes_{\O_{\tilBa_S}} p_{j}^{*} \O(-1)^{\epsilon_{j+1}} \right)$$ is concentrated in degree $0$ unless $T_j = 3$ and $\epsilon_j = \epsilon_{j+1} = 1$, when it is concentrated in degrees $0$ and $1$. The first statement then follows from the K\"unneth formula.

When $T_j = 3$ for all $j \in \Z/f\Z$, Lemma \ref{lem:Baj} shows that the dualizing sheaf of $\tilBa_S$ is trivial. Further, Lemma \ref{lem:Rpr-dualizing} shows that for each $j$, the cohomologies of $p_{j}^{*} \cO_{\PP^{1}}$, $p_{j}^{*} \cO_{\PP^{1}}(1)$ and $p_{j}^{*} \cO_{\PP^{1}}(2)$ are all concentrated in degree $0$. The same is therefore true for $\wedge^{a}\cE^{\vee}$ for any $a$. The second point in the statement of the Lemma thus holds true.

Finally, given $\gr$ as in the third point, Lemma \ref{lem:Baj} implies that $$\omega_{B(\gr)} \otimes \bigwedge^{a} \cE(\gr)^{\vee} \cong \bigoplus_{\underline{\epsilon}} \cK(\underline{\epsilon})$$
where 
\begin{align}\label{eqn:K-epsilon-defn} \cK(\underline{\epsilon}) := p_{i-k}^{*} \cO(-1)^{2-\epsilon_{i-k+1}} \otimes 
\left( \bigotimes_{j = i-k+1}^{i-1}  p_j^{*} \O(\epsilon_j + \epsilon_{j+1}) \right)
\otimes q_{i}^{*} \cO(-1)^{2-\epsilon_{i}}\end{align}
and the direct sum is over $\{\underline{\epsilon} = (\epsilon_j)_{j=i-k+1}^{i} \in \{0,1\}^{l(\gr) -1} \mid \sum_{j=i-k+1}^{i} \epsilon_j = a\}$. Since $T_{i-k} \in \{1, 2\}$ and $T_{i} \in \{1, 5\}$, 
\begin{center}$\widetilde{Ba}_{i-k}(\widetilde{z}_{i-k})_{S_{i-k}}^{\mathrm{I}} \cong \PP^{1} \times \GL_2$, and $\widetilde{Ba}_{i}(\widetilde{z}_{i})_{S_{i}}^{\mathrm{II}} \cong Z_{i} \times \PP^{1}$. \end{center}
Therefore, $R \Gamma \left( \cK(\underline{\epsilon}) \right)$ is non--vanishing if and only if $\epsilon_{i-k+1} = \epsilon_{i} = 0$. In particular,
$R\Gamma \left( 
        \omega_{B(\gr)} \otimes 
        \bigwedge^{a} \cE^{\vee} \right)$ is non--vanishing if and only if $a \leq l(\gr)-3$. Assume now that $R \Gamma \left( \cK(\underline{\epsilon}) \right)$ is non--vanishing. 
        Since the cohomologies of $p_{i-k}^{*} \cO(-2)$ and $q_{i}^{*} \cO(-2)$ are both concentrated in degree $1$, and that of $p_j^{*} \O(\epsilon_j + \epsilon_{j+1})$ for $j \in \{i-k+1, \dots, i-1\}$ is concentrated in degree $0$ by Lemma \ref{lem:Rpr-dualizing},
        $R^{b} \Gamma \left( \cK(\underline{\epsilon}) \right) \neq 0$ if and only if $b= 2$. 
\end{proof}   

\begin{cor}\label{cor:support-coh}
\begin{enumerate}
    \item If $i >0$, $R^{i}\Gamma(\cI(\tilz)) \neq 0$ if and only $i=1$ and $T_j = 3$ for each $j \in \Z/f\Z$. In this case, it is supported at $V(N) \subset \Spec \Gamma(\tilBa_S)$ and is free of rank $1$ over $\Gamma(\tilBa_S)/N$.
    \item Suppose $T_j = 3$ for each $j \in \Z/f\Z$. The group $R^{i} \Gamma(\omega_{\tilYz_S})$ vanishes when $i>0$ and has support $$\tilZz_S \subset \Spec \Gamma(\tilBa_S)$$ when $i=0$.
    \item  Suppose $\gr \in \trn^{*}$ and $l(\gr) \geq 3$. The group $R^{i} \Gamma(\omega_{Y(\gr)})$ vanishes when $i>0$ and has support $\cZ(\gr) \subset \Spec \Gamma(B(\gr))$ when $i=0$. 
    
    %$ \subset Z_{\tilB}(\gr)$. 
\end{enumerate}
\end{cor}
\begin{proof}
  The assertions about vanishing of higher cohomologies and support of $R^{1}\Gamma(\cI(\tilz))$ follow immediately from Corollary \ref{cor:ss} and Lemma \ref{lem:coh-L1}.
  
  % and vanishing of higher cohomologies in $R \Gamma(\omega_{\tilYz_S})$ when $T_j =3$ for each $j$, and in $R\Gamma(\omega_{Y(\gr)})$ for $\gr \in \trn^{*}$, follow. 

  By Lemma \ref{lem:birational-T3}, resp. Lemma \ref{lem:bad-sequence}(ii), $\tilZz_S \smallsetminus V(N)$, resp. $\cZ(\gr) \smallsetminus V(N(\gr))$, is isomorphic to its smooth preimage in $\tilYz_S$, resp. in $Y(\gr)$. Therefore, the dualizing sheaf on $\tilZz_S \smallsetminus V(N)$, resp. $\cZ(\gr) \smallsetminus V(N(\gr))$, is non-zero and supported everywhere. Since this dualizing sheaf is obtained by restricting the quasicoherent sheaf associated to the module $$R^{0} \Gamma\left(\omega_{\tilYz_S}\right), \text{ resp. } R^{0} \Gamma(\omega_{Y(\gr)}),$$  we get the desired statement on the support of this module.
\end{proof}

\begin{prop}\label{prop:non-gor}
\begin{enumerate}
    \item Suppose $T_j =3$ for each $j \in \Z/f\Z$. The rank of $R^{0}\Gamma\left(\omega_{\tilYz_S}\right)$ at the point $(\tilz_j)_j \in \tilZz_S$ is $\geq 3$.
    \item Suppose $\gr = (i-k, \dots, i) \in \trn^{*}$ with $l(\gr) \geq 4$. The rank of $R^{0}\Gamma(\omega_{Y(\gr)})$ at every point in $V(N(\gr)) \subset \Spec \Gamma(B(\gr))$ is $\geq 2$.
\end{enumerate}

\end{prop}
\begin{proof}
We first deal with the case when $T_j =3$ for each $j \in \Z/f\Z$. 
In this setting, the domain of the map (\ref{eqn:E1-map-T3}) admits a description as the global sections of the structure sheaf if $f=1$, and of
$$\bigoplus_{j \in \Z/f\Z} \left(\left(\bigotimes_{i \in \{j-1, j\}} p_{i}^{*} \cO_{\PP^1}(1) \right) \otimes \left(\bigotimes_{i \neq j-1, j} p_i^{*} \cO_{\PP^1}(2) \right)\right)$$ 
if $f>1$. The codomain admits a description as the global sections of $\bigotimes_{i \in \Z/f\Z} p_i^{*} \cO_{\PP^1}(2).$ By the explicit description of the Koszul complex, the map (\ref{eqn:E1-map-T3}) is given by mapping the function $1$ to $\mathfrak{s}_0 \in \Gamma(p_0^{*} \cO_{\PP^1}(2))$ if $f=1$, and tensoring a section of 
\begin{align}\label{eqn:domain-summand-E1}
\Gamma\left(\left(\bigotimes_{i \in \{j-1, j\}} p_{i}^{*} \cO_{\PP^1}(1) \right) \otimes \left(\bigotimes_{i \neq j-1, j} p_i^{*} \cO_{\PP^1}(2) \right)\right) \cong \nonumber \\\left(\bigotimes_{i \in \{j-1, j\}} \Gamma(p_{i}^{*} \cO_{\PP^1}(1)) \right) \otimes \left(\bigotimes_{i \neq j-1, j} \Gamma(p_i^{*} \cO_{\PP^1}(2)) \right),\end{align}
with $\pm \mathfrak{s_j} \in \Gamma\left(p_{j-1}^{*}\cO(1) \otimes q_j^{*}\cO(1)\right)$ if $f>1$. If $f>1$, denote by $s_j$ the restriction of the map (\ref{eqn:E1-map-T3}) to (\ref{eqn:domain-summand-E1}) above.

For each $j \in \Z/f\Z$, denote the pullback of the projective coordinates $[x:y]$ on $\Baj_{S_j}$ to $\tilBa_S$ by $[x_j: y_j]$.
In light of Remark \ref{rem:identify-sections}, for each $j \in \Z/f\Z$, we can (and do) describe the space of global sections of $p_j^{*} \cO(1)$ as a module generated by $\{x_j, y_j\}$, and that of $p_j^{*} \cO(2)$ as a module generated by $\{x_j^{2}, x_j y_j, y_j^{2}\}$.

The preimage in $\tilBa_S$ of the point $(\tilz_j)_j$ is cut out by setting $\kappa_j = \id$ and $X_j = \id$ for each $j \in \Z/f\Z$. Consider the map (\ref{eqn:E1-map-T3}) after specialization to $(\tilz_j)_j$, since cokernel commutes with base change. Using that the preimage of $(\tilz_j)_j$ in $\tilYz_S$ is cut out by setting $[x_{j-1}:y_{j-1}] = [x_j:y_j]$ for each $j$, we find that upon restriction to this preimage, the section $\mathfrak{s}_j$ is equivalent to 
\begin{align*}
    &x_0 y_0 - x_0 y_0 
&\text{ if } f=1, &\text{ and} \\
    &x_{j-1} \otimes y_j - y_{j-1} \otimes x_j &\text{ if } f>1.
\end{align*}
Suppose first that $f=1$. We have $\mathfrak{s}_0 \equiv 0$, and so, the rank of the cokernel is the rank of the codomain, which is $3$. Now, suppose $f>1$. For $j \in \Z/f\Z$, after specialization, the space in (\ref{eqn:domain-summand-E1}) and the codomain of (\ref{eqn:E1-map-T3}) can be identified with 
$$\left(\bigotimes_{i \in \{j-1, j\}} \F[x_i, y_i] \right) \otimes \left(\bigotimes_{i \neq j-1, j} \F[x_i^{2}, x_i y_i, y_i^2] \right) \text{ and } \bigotimes_{i \in \Z/f\Z} \F \left[x_{i}^2, x_{ i}y_{ i}, y_{i}^2\right]$$ respectively.
%We can take for a basis of the latter the set $$\cB = \left\{\otimes_{i} a_i \mid a_i\in \{x_i^2, x_i y_i, y_i^2\}\right\}.$$
The image of the map $s_j$ is the span of those $\otimes_i a_i$ for which $a_{j-1} \otimes a_j$ lies in $\{\alpha_j, \beta_j, \gamma_j, \delta_j\}$, where 
\begin{align}\label{eqn:im-s}
&\alpha_j := x_{j-1}^2 \otimes x_j y_j - x_{j-1}y_{j-1} \otimes x_j^2, \nonumber\\
&\beta_j := x_{j-1}^2 \otimes y_j^2 -  x_{j-1} y_{j-1} \otimes x_j y_j, \nonumber \\
&\gamma_j := x_{j-1} y_{j-1} \otimes x_j y_j- y_{j-1}^2 \otimes x_j^2, \nonumber \\
&\delta_j := x_{j-1}y_{j-1} \otimes y_j^2 - y_{j-1}^2 \otimes x_j y_j.\end{align} 
%In particular, linear combinations of sections of the form $\otimes_i a_i$ with $a_{j-1} \otimes a_j \in \{x_{j-1}^{2} \otimes x_j^{2}, x_{j-1}^{2} \otimes y_j^{2},  y_{j-1}^{2} \otimes y_j^{2}\}$ and $a_j \otimes a_{j+1} \in \{x_{j}^{2} \otimes x_{j+1}^{2}, y_{j}^{2} \otimes x_{j+1}^{2},  y_{j}^{2} \otimes y_{j+1}^{2}\}$ do not lie in the images of the maps $s_j$ and $s_{j+1}$ respectively. 
We claim that the $\F$--span of $\{\otimes_i x_i^2, x_0^{2} \otimes y_1^{2} \otimes (\otimes_{i \not\in \{0, 1\}}x_i^2), \otimes_i y_i^{2}\}$, denoted $V$, intersects trivially with the image of (\ref{eqn:E1-map-T3}) (after specialization). Moreover, it remains $3$--dimensional after passing to the cokernel, showing that the cokernel has rank $\geq 3$. To see this, take for a basis of the codomain the set
$$\cB := \{\otimes_i a_i | a_i \in \{x_i^{2}, x_i y_i, y_i^{2}\}\}.$$ For $j \in \Z/f\Z$, we say that $b = \otimes_i a_i \in \cB$ involves cross terms in $j$ if $$(a_{j-1}, a_j) \not\in \{(x_{j-1}^2,x_j^2), (y_{j-1}^2, y_j^2)\}.$$ %Otherwise, we say that $b$ involves pure terms in $j$. 
Note that the image of $s_j$ lies in the span of basis elements with cross terms in $j$.
Suppose $\sum_{b \in \cB} c_b b \in V$ is in the image. For $b \in \{\otimes_i x_i^2, \otimes_i y_i^{2}\}$, $c_b$ is $0$ because $b$ does not involve any cross terms. The only thing to check then is that $v = x_0^{2} \otimes y_1^{2} \otimes (\otimes_{i \not\in \{0, 1\}}x_i^2) \in \cB$ is not in the image. 

Assume otherwise. Suppose $v = u +w$ with $u \in \sum_{j \not\in \{1, 2\}} \text{Im } s_j$ and $w \in \text{Im } s_1 + \text{Im } s_2$. Since $\text{Im } s_j$ is written purely in terms of basis elements with cross terms in $j$ and $v$ does not involve any cross terms in $j \not\in \{1, 2\}$, $u$ can be taken to be $0$. When $f=2$, $\text{Im } s_1 = \text{Im } s_2$, and neither contains $v$, giving rise to a contradiction. When $f \geq 3$, the fact that $v \in \cB \cap  \text{Im } s_1 + \text{Im } s_2$ allows us to simply assume that $f=3$ and $v = x_0^2 \otimes y_1^2 \otimes x_2^2$. The space $\text{Im } s_1 + \text{Im } s_2$ is generated by elements in the set $\cG:= \{\alpha_1, \beta_1, \gamma_1, \delta_1\} \otimes \{x_2^2, x_2 y_2, y_2^2\} \cup \{x_0^2, x_0 y_0, y_0^2\} \otimes \{\alpha_2, \beta_2, \gamma_2, \delta_2\}$. Since $v$ involves cross terms in $j \in \{1, 2\}$, $v$ can be written as a linear combination of the elements of \begin{align*}\cG^{(1)}:= \cG \smallsetminus \{\alpha_1 \otimes x_2^2, \beta_1 \otimes y_2^2, \gamma_1 \otimes x_2^2, \delta_1 \otimes y_2^2, x_0^2 \otimes \alpha_2, x_0^2 \otimes \beta_2, y_0^2 \otimes \gamma_2, y_0^2 \otimes \delta_2\}.\end{align*} 
Using the relation $\beta_1 \otimes x_2^2 = x_0y_0 \otimes \alpha_2 + \alpha_1 \otimes x_2y_2 - x_0^2 \otimes \gamma_2$,
we can further write $v$ as a linear combination of the elements of $\cG^{(2)} := \cG^{(1)} \smallsetminus \{\beta_1 \otimes x_2^2\}$. Since the only element of $\cG^{(2)}$ involving the basis element $v$ is $x_0^2 \otimes \gamma_2$, we can write $v + x_0^2 \otimes \gamma_2 = x_0^2 \otimes x_1y_1 \otimes x_2y_2$ as a linear combination of elements in $\cG^{(3)} := \cG^{(2)} \smallsetminus \{x_0^2 \otimes \gamma_2\}$. The only element of $\cG^{(3)}$ involving the basis element $x_0^2 \otimes x_1y_1 \otimes x_2y_2$ is $\alpha_1 \otimes x_2y_2$. Therefore $x_0^2 \otimes x_1y_1 \otimes x_2y_2 - \alpha_1 \otimes x_2y_2 = x_0y_0 \otimes x_1^2 \otimes x_2y_2$ is a linear combination of elements in $\cG^{(4)} := \cG^{(3)} \smallsetminus \{\alpha_1 \otimes x_2y_2\}$. Repeating the procedure again, we find that $x_0y_0 \otimes x_1^2 \otimes x_2y_2 - x_0y_0 \otimes \alpha_2 = x_0y_0 \otimes x_1y_1 \otimes x_2^2$ is a linear combination of elements in $\cG^{(5)} := \cG^{(4)} \smallsetminus \{x_0y_0 \otimes \alpha_2\}$. But this is impossible as none of the elements in $\cG^{(5)}$ can be expressed as a linear combination of elements in $\cB$ with a non--trivial coefficient of $x_0y_0 \otimes x_1y_1 \otimes x_2^2$, since we have removed $\gamma_1 \otimes x_2^2$, $x_0y_0 \otimes \alpha_2$ and $\beta_1 \otimes x_2^2$ from the allowed set of generators.

Next, we move on to the case when $\gr = (i-k, \dots, i) \in \trn^{*}$ with $l(\gr) \geq 4$. 
The proof in this case is very similar to the previous one with slightly more care needed for describing explicitly the map (\ref{eqn:E1-map}), and for showing that the cokernel has rank $\geq 2$ over all of $V(N(\gr))$ instead of at one finite type point. 

Let $J = \{i-k+2, \dots, i-1\}$. For $l \in J$, define $$\underline{\epsilon}^l = (\epsilon^l_j)_{j=i-k+1}^{i}\in \{0, 1\}^{l(\gr)-1}$$ such that $\epsilon^l_{i-k+1} =\epsilon^l_{i}= \epsilon_l^l = 0$ and for $j \in J \smallsetminus \{l\}$, $\epsilon^l_j = 1$.
    Letting $\cK(\underline{\epsilon}^j)$ be the sheaf defined in (\ref{eqn:K-epsilon-defn}), we have
    \begin{align*}
            &\bigoplus_{j \in J}R\Gamma^{0} \left(\cK(\underline{\epsilon}^j) \otimes \omega_{B(\gr)}^{\vee}\right) \cong R^2\Gamma\left(\omega_{B(\gr)}\right) \otimes \left(\bigoplus_{j \in J}R\Gamma^{0} \left(\cK(\underline{\epsilon}^j) \otimes \omega_{B(\gr)}^{\vee}\right) \right) \\
            & \cong \bigoplus_{j \in J} R^2 \Gamma \left(\cK(\underline{\epsilon}^j)\right) \cong R^{2}\Gamma \left( \omega_{B(\gr)} \otimes 
        \bigwedge^{l(\gr)-4} \cE(\gr)^{\vee} \right),
        \end{align*} where the first isomorphism follows from $R^2\Gamma\left(\omega_{B(\gr)}\right) \cong \Gamma(B(\gr))$, the second by the K\"unneth formula, and the third is induced by the embedding
        $$\bigoplus_{j \in J} \cK(\underline{\epsilon}^j) \hookrightarrow \omega_{B(\gr)} \otimes 
        \bigwedge^{l(\gr)-4} \cE(\gr)^{\vee},$$ and uses the
        fact that all cohomology groups of $\cO(-1)$ on $\PP^{1}$ vanish.
    
 Next, let $\underline{\epsilon} = (\epsilon_j)_{j = i-k+1}^{i}$ be such that $\epsilon_{i-k+1} = \epsilon_{i} = 0$ and for $j \in \{i-k+2, \dots, i-1\}$, $\epsilon_j = 1$. As before, the inclusion $$\cK(\underline{\epsilon}) \hookrightarrow \omega_{B(\gr)} \otimes 
        \bigwedge^{l(\gr)-3} \cE(\gr)^{\vee}$$
and K\"unneth formula induce isomorphisms
\begin{align*}
            R\Gamma^{0} \left(\cK(\underline{\epsilon}) \otimes \omega_{B(\gr)}^{\vee}\right)
            \cong R^2\Gamma\left(\omega_{B(\gr)}\right) \otimes R\Gamma^{0} \left(\cK(\underline{\epsilon}) \otimes \omega_{B(\gr)}^{\vee}\right) \\ \cong R\Gamma^{2} \left(\cK(\underline{\epsilon})\right) \cong R^{2}\Gamma \left( \omega_{B(\gr)} \otimes 
        \bigwedge^{l(\gr)-3} \cE(\gr)^{\vee} \right).
        \end{align*}

By functoriality of the cup product underlying the K\"unneth isomorphism, the map $(\ref{eqn:E1-map})$ corresponds to a map
\begin{align}\label{eqn:E1-map-simplified}
    &\bigoplus_{j \in J}R\Gamma^{0} \left(\cK(\underline{\epsilon}^j) \otimes \omega_{B(\gr)}^{\vee}\right) \xrightarrow{(s_j)_{j \in J}} R\Gamma^{0} \left(\cK(\underline{\epsilon}) \otimes \omega_{B(\gr)}^{\vee}\right)
\end{align} where for each $j \in J$, the map $s_j$ is given, upto a sign, by tensoring with the section $$\mathfrak{s}_j \in \Gamma\left(p_{j-1}^{*}\cO(1) \otimes q_j^{*}\cO(1)\right)$$ considered in Lemma \ref{lem:koszul}.

As before, for each $j \in J \cup \{i-k+1\}$, we denote the pullback of the projective coordinates $[x:y]$ on $\Baj_{S_j}$ to $Ba(\gr)$ by $[x_j: y_j]$, and describe the space of global sections of $p_j^{*} \cO(1)$ as a module generated by $\{x_j, y_j\}$, and that of $p_j^{*} \cO(2)$ as a module generated by $\{x_j^{2}, x_j y_j, y_j^{2}\}$. We now describe a change of coordinates on $\{[x_j: y_j]\}_{j = i-k+1}^{i-1}$.

Let $\kappa_j \in \GL_2(B(\gr))$ be the pullback of the universal point of the copy of $\GL_2$ in $\tilBaj$.
We define matrices $C_j 
\in \GL_2(B(\gr))$ for each $j \in J \cup \{i-k+1\}$ inductively by setting $C_{i-k+1} = \id$ and for $j \in J$, $C_j = \kappa_{j-1}C_{j-1}$. Set $[x_{j}^{\text{new}} : y_j^{\text{new}}] := [x_j:y_j]C_j$. For each $j \in J$, the relation $$[x_{j-1}^{\text{new}} : y_{j-1}^{\text{new}}]C_{j-1}^{-1} \kappa_{j-1}^{-1} = [x_{j-1}: y_{j-1}]\kappa_{j-1}^{-1} = [x_j:y_j] = [x_{j}^{\text{new}} : y_{j}^{\text{new}}]C_{j}^{-1}$$ is thus equivalent to 
\begin{align*}
&[x_{j-1}^{\text{new}} : y_{j-1}^{\text{new}}] = [x_{j}^{\text{new}} : y_{j}^{\text{new}}].
\end{align*} Therefore, for each $j \in J$, we can (and do) take $\mathfrak{s}_j$ to be the section $x_{j-1}^{\text{new}} \otimes y_{j}^{\text{new}} - y_{j-1}^{\text{new}} \otimes x_{j}^{\text{new}}$.

Since cokernel commutes with base change, we consider the map (\ref{eqn:E1-map-simplified}) mod $N(\gr)$. Abusing notation, we rename $x_j^{\text{new}}$ and $y_j^{\text{new}}$ simply as $x_j$ and $y_j$ respectively. From the explicit description of $\cK(\underline{\epsilon})$ and Lemma \ref{lem:Rpr-dualizing}, one sees that the codomain of (\ref{eqn:E1-map-simplified}) mod $N(\gr)$ admits an isomorphism to 
$$\F\left[x_{i-k+1}, y_{i-k+1}\right] \otimes \left( \bigotimes_{j \in J \smallsetminus \{i-1\}} \F \left[x_{j}^2, x_{ j}y_{ j}, y_{j}^2\right] \right) \otimes \F\left[x_{i-1}, y_{i-1}\right].$$

As observed in the proof of (i), 
%tensoring with $\mathfrak{s}_j$ for $j \in \{i-k+2, i-1\}$ will produce tensors that are necessarily described completely in terms of basis elements of pure tensors containing $\dots x_{j-1} \otimes y_{j} \dots$ or $\dots y_{j-1} \otimes x_{j} \dots$. In particular, the image of (\ref{eqn:E1-map-simplified}) mod $N(\gr)$ intersects trivially with the two--dimensional $\F$--subspace generated by \begin{center}
the image of (\ref{eqn:E1-map-simplified}) mod $N(\gr)$ intersects trivially with the two--dimensional $\F$--subspace generated by
\begin{center} $\{x_{i-k+1} \otimes (\bigotimes_{j \in J \smallsetminus \{i-1\}} x_j^{2}) \otimes x_{i-1}, \quad y_{i-k+1} \otimes (\bigotimes_{j \in J \smallsetminus \{i-1\}} y_j^{2}) \otimes y_{i-1}\}$. \end{center}
This finishes the proof.
\end{proof}

\begin{proposition}\label{prop:coh-str-sheaf}(Version of \cite[Prop.~3.3.14]{lhmm}) For $i>0$
    $$R^{i} \pr_{*} \O_{\tilY_S} = 0.$$
\end{proposition}
\begin{proof}
    Since the codomain of $\pr_{\tilB}$ is affine, $R \pr_{\tilB *} \cI(\tilz)$ is the quasicoherent sheaf associated to $R\Gamma( \cI(\tilz))$. The proof is then immediate from Lemma \ref{lem:structure-sh-coh-vs-I} and Corollary \ref{cor:support-coh}.
\end{proof}

\begin{proposition}\label{prop:coh-dualizing-sheaf}
Suppose $\dim \tilYz_S = \dim \tilZz_S$ and that $T_j \neq 3$ for some $j \in \Z/f\Z$.
Then, $R\pr_{*} \omega_{\tilYz_S}$ is concentrated in degree $0$ and is an invertible sheaf on $\tilZz_S$ if and only if for each $\gr \in \trn^{*}$, $l(\gr) = 3$.
The locus in $\tilZz_S$ where the rank of $R^0\pr_{*} \omega_{\tilYz_S}$ is $\geq 2$ is precisely $$\bigcup_{\substack{\{\gr \in \trn^{*} | l(\gr) > 3 \}}} V(N(\gr)).$$
\end{proposition}
\begin{proof}

 Since the codomain of $\pr$ is affine, $R \pr_{*} \omega_{\tilYz_S}$ is the sheaf associated to the module $R \Gamma (\omega_{\tilYz_S})$, with support necessarily contained in the scheme--theoretic image $\tilZz_S$ of $\tilYz_S$.
    By (\ref{eqn:Y-as-prod}) and the K\"unneth formula, we have
$$R^n \Gamma \left(\omega_{\tilYz_S} \right) = \bigoplus_{(m_{\gr})_{\gr}, \sum m_{\gr} = n} \left(\bigotimes_{\gr \in \trn} R^{m_\gr} \Gamma \left(\omega_{Y(\gr)} \right) \right).$$

When $\gr \in \trn \smallsetminus \trn^{*}$, Lemma \ref{lem:bad-sequence} shows that $Y(\gr)$ is isomorphic to its scheme--theoretic image under $\pr(\gr)$. Since $Y(\gr)$ is Gorenstein by Lemma \ref{lem:koszul}, $R \Gamma (\omega_{Y(\gr)})$ is concentrated in degree $0$ and of constant rank $1$ on $\cZ(\gr)$.

On the other hand, when $\gr \in \trn^{*}$, Lemma \ref{lem:bad-sequence} shows that $\dim \tilYz_S = \dim \tilZz_S$ implies that for each $\gr \in \trn^{*}$, $l(\gr) \geq 3$. Corollary \ref{cor:support-coh} thus shows that $R \Gamma(\omega_{Y(\gr)})$ is concentrated in degree $0$. The same argument as in the previous paragraph shows that $R \Gamma (\omega_{Y(\gr)})$ has constant rank $1$ on $\cZ(\gr) \smallsetminus V(N(\gr))$. If $l(\gr) = 3$, Lemmas \ref{lem:Rpr-dualizing} and \ref{lem:coh-L1} imply that $R^{0} \Gamma(\omega_{Y(\gr)})$ is free of rank $1$ over $\Gamma(B(\gr))$. An application of Proposition \ref{prop:non-gor} finishes the proof.
\end{proof}

\begin{prop}\label{prop:non-gor-truncs}
Suppose $\gr = (i-k, \dots, i) \in \trn^{*}$ with $l(\gr) = 3$. Then $\cZ(\gr)$ is a local complete intersection, with singular locus given precisely by $V(N(\gr))$.
\end{prop}
\begin{proof}
  When $l(\gr) = 3$, since $R^{0} \Gamma(\omega_{Y(\gr)})$ is supported on $\Spec \Gamma (B(\gr))$ and the support must be contained in $\cZ(\gr) \subset \Spec \Gamma(B(\gr))$, we obtain an isomorphism
  $$\cZ(\gr) \cong \Spec \Gamma (B(\gr)) \cong \GL_2 \times (Z_{i-1} \times \GL_2) \times Z_i$$ where $Z_i$ is smooth and $Z_{i-1} \cong \F[B, C, D]/(D^2 + BC)$.
  The latter is a local complete intersection, with the sheaf of differentials locally free of rank $2$ away from the vanishing locus of $N(\gr)$ and of rank $3$ upon restriction to $V(N(\gr))$.
\end{proof}

Let $\cZ^{\tau, \mathrm{nm}}_S$ be the normalization of $\cZ^{\tau}$, as described in \cite[Appendix~A]{ascher2018moduli}. Since $\Y_S$ is smooth, the map $\Y_S \to \cZ^{\tau}_S$ factors through $\cZ^{\tau, \mathrm{nm}}_S$ and we have the following Cartesian diagram
\begin{equation}\label{diagram:nm}
\begin{tikzcd}
    \tilY_{S} \arrow[r] \arrow[d] & \tilZnm_S \arrow[d] \arrow[r] & \tilZ_S \arrow[d] \\
    \Y_S \arrow[r] & \cZ^{\tau, \mathrm{nm}}_S \arrow[r] & \cZ^{\tau}_S
\end{tikzcd}
\end{equation}
where the vertical arrows are $\GL_2^{\Z/f\Z}$--torsors and the scheme $\tilZnm_S$ is the normalization of $\tilZz_S$. It admits an open cover by the normalizations $\tilZnmz_S$ of $\tilZz_S$.
\begin{theorem}\label{thm:sing}
    Suppose $\tilY_S \xrightarrow{\widetilde{\pi}_S} \tilZ_S$ is birational. The following are true:
    \begin{enumerate}
        \item The scheme $\tilZz_S$ is normal if and only if there exists $j \in \Z/f\Z$ with $T_j \neq 3$.
        \item The scheme $\tilZnmz_S$ is resolution--rational. It is Gorenstein if and only if 
        \begin{itemize}
        \item there exists $j \in \Z/f\Z$ with $T_j \neq 3$, and
        \item $l(\gr) = 3$ for each $\gr \in \trn^{*}$.
        \end{itemize}
        When these conditions are met, it is in fact a local complete intersection.
    \end{enumerate}
\end{theorem}

% Since $\dim \tilY_S = \dim \tilZ_S$, there exists $\tilz$ such that $\dim \tilYz_S = \dim \tilZz_S$. By Lemma \ref{lem:bad-sequence}, if there exists $j \in \Z/f\Z$ with $T_j \neq 3$, then $\tilY_S$ is birationally isomorphic to $\tilZ_S$. If each $T_j = 3$, then away from the vanishing locus of $I$, we get an isomorphism between $\tilYz_S$ and $\tilZz_S$.
\begin{proof}
By properness of $\widetilde{\pi}_S$, it admits a Stein factorization
\begin{align}\label{eqn:Stein-fac}
\tilY_S \xrightarrow{\widetilde{\pi}^{\mathrm{nm}}} \relSpec{\widetilde{\pi}_{*} \O_{\tilY_S}} \cong \tilZnm_S \to \tilZ_S\end{align} where $\relSpec{(\widetilde{\pi}_S)_{*} \O_{\tilY_S}} \cong \tilZnm_S$ because $\widetilde{\pi}_S$ is birational with smooth domain. Pulling back along the open immersion $$\tilZnmz_S \hookrightarrow \tilZnm_S$$ gives the Stein factorization of $\pr$:
$$\pr: \tilYz_S \xrightarrow{\widetilde{\pi}^{\mathrm{nm}}} \relSpec{\pr_{*} \O_{\tilYz_S}} \cong \tilZnmz_S \to \tilZz_S \xhookrightarrow{\iota} \Uz.$$
The corresponding map on sheaves is 
\begin{align}\label{eqn:map-on-functions}
    \cO_{\Uz} \onto \iota_{*}\cO_{\tilZz_S} \to \pr_{*} \O_{\tilYz_S}
\end{align}
with cokernel $R^{1}\pr_{\tilB*}\cI(\tilz)$ by Lemma \ref{lem:structure-sh-coh-vs-I}(iii). Thus, the map $\tilZnmz_S \to \tilZz_S$ fails to be an isomorphism if and only if $R^{1}\pr_{\tilB*} \cI(\tilz) \neq 0$. An application of Corollary \ref{cor:support-coh}(i) finishes the proof of the first part.

 %By Corollary \ref{cor:ss}, $$R^{1}\cI(\tilz) = R^{f}\Gamma \left(\bigwedge^{f} \cE\right).$$ 
%By Corollary \ref{cor:support-coh}, $R^{1}\cI(\tilz) \neq 0$ if and only if $T_j=3$ for each $j \in \Z/f\Z$, in which case it is isomorphic to $\F$ and supported at the vanishing locus of $N$ in $\tilZz_S \subset \Spec \Gamma (\tilBa_S)$, where the .

For the second statement, the proof proceeds in the same way as that of \cite[Thm.~4.6.6]{lhmm}. By
 Proposition \ref{prop:coh-str-sheaf},
   \begin{align}\label{eqn:pushforward-str}
    R\pr_{*} \O_{\tilYz_S} = \O_{\tilZnmz_S}.\end{align}
    Let $\omega^{\bullet}_{\tilZnmz_S}$ be a dualizing complex of $\tilZnmz_S$. Then, by \cite[\href{https://stacks.math.columbia.edu/tag/0BZL}{Tag 0BZL}]{stacks-project},
    $$\omega^{\bullet}_{\tilYz_S} := (\widetilde{\pi}^{\mathrm{nm}})^{!} \omega^{\bullet}_{\tilZnmz_S}$$ is a dualizing complex of $\tilYz_S$. 
    We have
    \begin{align*}
        R \pr_{*} \omega^{\bullet}_{\tilYz_S} &\cong R\pr_{*} R \sheafHom_{\O_{\tilYz_S}}(\O_{\tilYz_S}, \omega^{\bullet}_{\tilYz_S}) \\
        & \cong R\sheafHom_{\O_{\tilZnmz_S}}(R \pr_{*}\O_{\tilYz_S}, \omega^{\bullet}_{\tilZnmz_S}) \\
        & \cong R\sheafHom_{\O_{\tilZnmz_S}}(\O_{\tilZnmz_S}, \omega^{\bullet}_{\tilZnmz_S}) \\
        & \cong \omega^{\bullet}_{\tilZnmz_S}.
    \end{align*}
The first and last isomorphisms follow from the fact that sheaf homomorphisms from the structure sheaf to any sheaf $\cF$ are isomorphic to $\cF$, the second from Grothendieck duality (see for e.g. \cite[\href{https://stacks.math.columbia.edu/tag/0AU3}{Tag 0AU3}]{stacks-project}) and the third from (\ref{eqn:pushforward-str}). The scheme $\tilZnmz_S$ is thus resolution--rational by Corollary \ref{cor:support-coh} and Proposition \ref{prop:coh-str-sheaf}.

If $T_j=3$ for each $j \in \Z/f\Z$, the cokernel of (\ref{eqn:map-on-functions}) is free of rank $1$ over its support by Corollary \ref{cor:support-coh}(i). Therefore, the rank of the normalization of $\Gamma(\tilZz_S)$ is at most $2$ after specialization to any point of $\tilZz_S$. An identical argument as in the proof of Corollary \ref{cor:support-coh}(ii) shows that the dualizing sheaf of $\tilZnmz_S$ is supported everywhere. A consideration of ranks in Proposition \ref{prop:non-gor}(i) thus implies that $\tilZnmz_S$ is not Gorenstein.

On the other hand, if $T_j \neq 3$ for some $j$, then Lemma \ref{lem:bad-sequence} shows that $$\tilZz_S \smallsetminus \bigcup_{\gr \in \trn^{*}} V(N(\gr))$$ is smooth. Thus, it only remains to prove the desired statements when $\trn^{*} \neq \varnothing$. In this case, Propositions \ref{prop:coh-dualizing-sheaf} and \ref{prop:non-gor-truncs} settle the proof while showing that the singular locus is precisely $\bigcup_{\gr \in \trn^{*}} V(N(\gr))$.
\end{proof}

\begin{corollary}\label{cor:loc-non--normal}
    The versal ring at a point $\rhobar \in \cZ^{\tau}_S(\Fbar)$ is not normal if and only if it admits a lift to a point of $\tilZz_S$ for some $\tilz = (\tilz_j)$ with each $\tilz_j \in \{w_0 t_{\eta}, t_{w_0 (\eta)}\} s_j^{-1} v^{\mu_j}$, such that $T_j = 3$ for each $j \in \Z/f\Z$ and the lift lies in $V(N) \subset \tilZz_S$.
\end{corollary}
\begin{proof}
By the proof of Theorem \ref{thm:sing}, the versal ring at $\rhobar$ is not normal if and only if for some $\tilz = (\tilz_j)$ with each $\tilz_j \in \{w_0 t_{\eta}, t_{w_0 (\eta)}\} s_j^{-1} v^{\mu_j}$, $\rhobar$ lifts to a point of $\tilZz_S$ that lies in the support of $R^{1} \pr_{\tilB*} \cI(\tilz)$. The desired statement then follows from Corollary \ref{cor:support-coh} along with the additional observation that being the cokernel of the map in (\ref{eqn:map-on-functions}), $R^{1}\pr_{\tilB*}\cI(\tilz)$ is necessarily supported in $\tilZz_S \subset \widetilde{U}(\tilz)$.
\end{proof}

\begin{cor}\label{cor:loc-non-CM}
     Suppose $\dim \cZ^{\tau}_S = f$. The versal ring at a point $\rhobar \in \cZ^{\tau}_S(\Fbar)$ is normal but not smooth if and only if it admits a lift to a point of $\tilZz_S$ for some $\tilz = (\tilz_j)$ with each $\tilz_j \in \{w_0 t_{\eta}, t_{w_0 (\eta)}\} s_j^{-1} v^{\mu_j}$, such that $\trn^{*} \neq \varnothing$, and the lift lies in $$\bigcup_{\gr \in \trn^{*}}V(N(\gr)) \subset \tilZz_S.$$ If the lift further lies in $\bigcup_{\{\gr \in \trn^{*} | l(\gr) > 3\}}V(N(\gr))$, then the versal ring is Cohen--Macaulay but not Gorenstein, otherwise it is a local complete intersection.
\end{cor}
\begin{proof}
    Follows immediately from Propositions \ref{prop:coh-dualizing-sheaf} and \ref{prop:non-gor-truncs}, and the proof of Theorem \ref{thm:sing}(ii).
\end{proof}

\begin{cor}\label{cor:R1}
    The stack $\cZ^{\tau}_S$ is either normal or its non--normal locus is a closed substack of codimension $f$ whose preimage in $\Y_S$ also has codimension $f$. In the non--normal case, the complement of the non--normal locus is a smooth open substack isomorphic to its preimage in $\Y_S$.
\end{cor}
\begin{proof}
Suppose $\tilz$ is such that $\tilZz_S$ is not normal.
By Corollary \ref{cor:loc-non--normal}, $T_j=3$ for each $j \in \Z/f\Z$ and the non--normal locus is $V(N) \subset \tilZz_S \subset \Spec \Gamma(\tilBa_S)$. Consider the nilpotent thickening of $V(N)$ in $\Spec \Gamma(\tilBa_S)$ cut out by the functions $B, C$ for each $j$. These functions give a regular sequence of length $2f$ in $\Gamma(\tilBa_S)$, where the latter has dimension $6f$. Therefore, $V(N)$ has dimension $4f$ and its codimension in $\tilZz_S$ is $f$.

The stability of $V(N)$ under shifted-conjugation by $\GL_2^{\Z/f\Z}$ follows from direct computation and implies the descent to a closed substack of $\cZ^{\tau}_S$. Precisely, the locus $V(N) \subset \Spec \Gamma(\tilBa_S) \subset \widetilde{U}(\tilz)$ is characterized by tuples of matrices $(\kappa_j \tilz_j)_j \in \widetilde{U}(\tilz)$ with $$(\kappa_j)_j \in \GL_2^{\Z/f\Z}$$ and $\tilz_j$ in the center of $\LG$ for each $j$. 
% \qquad \text{ and } \qquad  \tilz_j = \begin{pmatrix}
%     v & 0 \\
%     0 & v
% \end{pmatrix}$$ for each $j$. 
Such tuples are evidently stable under $\GL_2^{\Z/f\Z}$--action. 

Next, we note from the explicit description in Section \ref{subsec:classfn} that the vanishing locus of $N$ in $\tilBa_S$ is isomorphic to $$\prod_{j \in \Z/f\Z} \Proj \F[x_j, y_j] \times \prod_{j \in \Z/f\Z} \GL_2.$$ Let $\kappa_j$ denote the universal point of the copy of $\GL_2$ in the $j$-th factor. The preimage of $V(N)$ in $\tilYz_S$ is cut out by setting $$[x_j: y_j]= [x_{j-1}:y_{j-1}]\kappa_{j-1}^{-1}$$ for each $j$. Thus, we can describe this preimage exclusively in terms of $[x_0:y_0]$ and $\{\kappa_j\}_{j}$, and get rid of the variables $[x_j:y_j]$ for $j \neq 0$. More precisely, setting $[x:y] := [x_{0}: y_{0}]$ and $$\kappa := \prod_{j = 0}^{f-1}\kappa_{j}^{-1} = \begin{pmatrix}
    a & b \\
    c & d
\end{pmatrix}, %\bs{\text{or: } \kappa_j' := \k_0^{-1} \kappa_1^{-1}\cdots \kappa_{j-1}^{-1}:= \begin{pmatrix}
%     a_j & b_j \\
%     c_j & d_j
% \end{pmatrix}} 
$$ the preimage of $V(N)$ in $\tilYz_S$ is isomorphic to the closed subscheme of $\Proj \F[x, y] \times \prod_{j} \GL_2$ obtained by setting $[x:y] = [x:y]\kappa = [ax + cy: bx+dy]$. 
%\bs{and $[x_j:y_j] = [x_0: y_0]\kappa_j'$}.
This is the same as setting $bx^{2} + (d-a)xy - cy^{2}=0$, 
%\bs{\text{and } $x_j(b_j x_0 + d_j y_0) - y_j(a_j x_0 + c_j y_0) = 0$ for all $j \in \mathbb{Z}/f\mathbb{Z} \backslash \{0\}$}, 
which cuts out a closed subscheme of pure dimension $4f$. Since the dimension of $\tilYz_S$ is $5f$, the codimension of the preimage of $V(N)$ is $f$. 
Finally, the complement of $V(N)$ is isomorphic to its smooth preimage in $\tilYz_S$ by Lemma \ref{lem:birational-T3}.
\end{proof}

\begin{cor}\label{cor:smooth-away-from-nonCM}
Suppose $\dim \cZ^{\tau}_S = f$, and $\cZ^{\tau}_S$ is normal but not smooth. Writing $\trn$ as $\trn(\tilz)$ to indicate the dependence on $\tilz$, let $$d := \min_{\substack{\tilz \; \mathrm{ s.t.} \\ \trn(\tilz)^{*} \neq \varnothing}} \; \min_{\gr \in \trn(\tilz)^{*}} (l(\gr) - 1).$$ 
The singular locus in $\cZ^{\tau}_S$ is a closed substack of codimension $d$ with preimage in $\Y_S$ of codimension $d - 1$. The smooth locus is isomorphic to its preimage in $\Y_S$.
\end{cor}
\begin{proof}
    Suppose $\tilz$ is such that $\tilZz_S$ is normal but not smooth.
Equivalently, by Corollary \ref{cor:loc-non-CM}, $\trn(\tilz)^{*} \neq \varnothing$. Let $\gr = (i-k, \dots, i) \in \trn(\tilz)^{*}$. We note that if $T_j \in \{1, 2\}$, $\dim \tilBaj_{S_j}^{\mathrm{I}} = \dim \GL_2 + 1$; if $T_j = 3$, $\dim \tilBaj_{S_j} = \dim (Z_j \times \GL_2)$; and if $T_j \in \{1, 5\}$, $\dim \tilBaj_{S_j}^{\mathrm{II}} = \dim Z_j + 1$. Therefore, $$\dim Z_{\tilB}(\gr) = \dim B(\gr)-2.$$
Since a nilpotent thickening of $V(N(\gr))$ in $Z_{\tilB}(\gr)$ is cut out by the functions $B, C$ for each $j \in \{i-k, \dots, i\} \smallsetminus \{i-k, i\}$, and these functions give a regular sequence of length $2(l(\gr) - 2)$ in $\Gamma(Z_{\tilB}(\gr))$, 
$$\dim V(N(\gr)) = \dim B(\gr) -2 - 2(l(\gr) - 2).$$ 

We also note that $$\dim \cZ(\gr) = \dim Y(\gr) = \dim B(\gr) - (l(\gr) - 1),$$ where the first equality follows from the assumption $\dim \tilZz_S = f$, and the second from Lemma \ref{lem:koszul}. Thus, the codimension of $V(N(\gr))$ in $\cZ(\gr)$ is $l(\gr) -1$. Viewing $N(\gr)$ as an ideal in $\Gamma(\tilBa_S)$ instead, the same is thus true for the codimension of $V(N(\gr))$ in $\tilZz_S$. The fiber in $\tilYz_S$ over each point of $V(N(\gr))$ is directly seen to be isomorphic to $\PP^1$ and so, the codimension of the preimage of $V(N(\gr))$ in $\tilYz_S$ is $l(\gr) -2$.

The closed scheme $V(N(\gr))$ descends to a closed substack of $\cZ^{\tau}_S$ of codimension $l(\gr) -1$ by the same argument as in Corollary \ref{cor:R1} and the complement of $$\bigcup_{\substack{\tilz \; \mathrm{ s.t.} \\ \trn(\tilz)^{*} \neq \varnothing}} \bigcup_{\gr \in \trn(\tilz)^{*}}V(N(\gr))$$ is isomorphic to its preimage by Lemma \ref{lem:bad-sequence}(ii). The desired statements follow immediately.
\end{proof}

\section{Combinatorics of Serre weights, tame types and shapes}\label{sec:combinatorics}
Let $\tau = \eta_1 \oplus \eta_2$ be a tame inertial type with implicit, fixed ordering of the two characters $(\eta_1, \eta_2)$. The paper \cite{cegsC} indexes irreducible components of the moduli stack of Breuil--Kisin modules with descent data of tame type $\tau^{\vee}$ by subsets $J \subset \Z/f\Z$. The irreducible component indexed by $J$ is denoted $\cC^{\tau}(J)$. In the following Lemma, we set up a dictionary between the notation of \cite{cegsC} and this article.

\begin{lemma}\label{lem:mu-combinatorics}
Let $\tau^{\vee} = \eta_1 \oplus \eta_2$ be a non--scalar principal series tame inertial type. Suppose $\eta_1 \eta_2^{-1} = \prod_{j \in \Z/f\Z} \omega_{j}^{\gamma_j}$ with $\gamma_j \in [0, p-1]$ for each $j \in \Z/f\Z$. Suppose further that either $\gamma_j \leq (p-1)/2$ for each $j \in \Z/f\Z$, or $\gamma_j < (p-1)/2$ for some $j \in \Z/f\Z$.

Let $M$ be the set, possibly empty, of maximal subsets $\{i-k, \dots, i\} \subset \Z/f\Z$ satisfying 
    \begin{enumerate}[label=(\alph*)]
        \item $\gamma_{i-k} < (p-1)/2$,
        \item $\gamma_{i-k+1}, \dots, \gamma_{i} \geq (p-1)/2$ if $k \neq 1$, and
        \item $\gamma_{i} > (p-1)/2$.
    \end{enumerate}
      Here, maximality is in the sense that the subset is not properly contained in another subset satisfying the three conditions. 
      Suppose either that $0 \in \Z/f\Z$ is not contained in any such subset, or $\gamma_0 < (p-1)/2$.  Then $\tau = \tau(s, \mu)$ where if $A = \{i-k, \dots, i\} \in M$, then $(\langle \mu_j, \alpha^{\vee} \rangle, s_j, s_{\ort, j})$ is
    \begin{align*}
        &\qquad \left(\gamma_j+1, w_0, w_0 \right) &&\text{ if } j=i-k, \nonumber \\
        &\qquad \left(p-1-\gamma_j, \id, w_0 \right) &&\text{ if } j \in A \smallsetminus \{i-k, i\}, \\
        &\qquad \left(p-\gamma_j, w_0, \id \right) &&\text{ if } j = i. \nonumber
    \end{align*}
   On the other hand, if $j$ is not in any subset contained in $M$, then 
    \begin{align*}
    (\langle \mu_j, \alpha^{\vee} \rangle, s_j, s_{\ort, j}) = \left(\gamma_j, \id, \id \right). \end{align*}
    In particular, $$p-2 > \max_{j\in \Z/f\Z} \<\mu_j, \alpha^{\vee}\>.$$
    Furthermore, $C^{\tau^{\vee}}(\Z/f\Z) = \Y_S$, where $S = (S_j)_{j \in \Z/f\Z}$ is as follows: If $A = \{i-k, \dots, i\} \in M$, then $S_{j} = L$ for $j \in A \smallsetminus \{i\}$ while $S_{i} = R$. If $j$ is not in any subset contained in $M$, then $S_j = R$.

\end{lemma}
\begin{proof}
We first make some general observations for $\tau^{\vee} \cong \eta_1 \oplus \eta_2$. Suppose $$\tau \cong \chi \otimes \left(\prod_{j \in \Z/f\Z} \omega_{j}^{a_j} \oplus \prod_{j \in \Z/f\Z}\omega_{j}^{b_j} \right)$$ for some character $\chi = \prod_{j} \omega_j^{c_j}$ such that $a_0 > b_0$ and for each $j$, $\min\{a_j, b_j\} = 0$ and $\max\{a_j, b_j\} \in [0, (p+1)/2]$. Let 
$$\mu_j := (\max\{a_j, b_j\} + c_j, c_j).$$ One verifies immediately that $\tau = \tau(s, \mu)$ where $s_{f-1}, s_{f-2}, \dots, s_0$ are determined (successively and) uniquely by the following constraints:
\begin{align*}
    &s_{f-1}^{-1}s_{f-2}^{-1} \dots s_{f-j}^{-1} (\mu_{f-j}) = (a_{f-j} + c_{f-j}, b_{f-j} + c_{f-j}) \text{ for } 1 \leq j \leq f-1, \\
    &s_j = \id \text{ whenever } \mu_j = (c_j, c_j), \text{ and}\\
    &s_0 s_1 \dots s_{f-1} = \id.
\end{align*}
The permutation $s_{\ort,j}$ is $\id$ if and only if for some $k \in [1, f]$ (taking indices in $\Z$ instead of $\Z/f\Z$), $a_{j+i} = b_{j+i}$ for each $i \in [1, k-1]$ and $a_{j+k} > b_{j+k}$.

% \item $$.
% \item \begin{align*} \prod_{j=0}^{f-1} s_j =\begin{cases} 
% \id &\text{ if } \tau \text{ is a principal series type},\\
% w_0 &\text{ if } \tau \text{ is a cuspidal type}.
% \end{cases}
% \end{align*} 
% \item $s_j = \id$  $\mu_j = (0, 0)$.
% \end{itemize}
  
A criterion for a Breuil--Kisin module $\gM$ with $A$ coefficients to be a point of $\cC^{\tau^{\vee}}(J)$ is given in \cite[Cor.~3.17]{bellovin2024irregular} and is as follows: For each $j$, consider an ordered basis (Zariski locally on $A$) $\{e_j, f_j\}$ of $\gM_j$ with $I(K'/K)$ acting via $\eta_1$ on $e_j$ and via $\eta_2$ on $f_j$. Consider the matrix of the Frobenius map $\Phi_{\gM, j}: \varphi^{*} \gM_{j-1} \to \gM_{j}$ with respect to the basis $\{1 \otimes e_{j-1}, 1 \otimes f_{j-1}\}$ of the domain and $\{e_{j}, f_{j}\}$ of the codomain. Then $\gM \in \cC^{\tau^{\vee}}(J)(A)$ if and only if $v$ divides the top left entry (resp. bottom right entry) of the matrix for $\Phi_{\gM, j}$
whenever $j \in J$ (resp. $j \not\in J$).
Thus, we find that $\Y_S = C^{\tau^{\vee}}(J)$ where $j \in J$ if and only if 
\begin{align*}
    S_j = \begin{cases}
    \begin{aligned}
        L &\quad \text{ if } s_{\ort, j} = \id &\text{ and }\eta_1^{-1} = \chi \otimes \prod_j \omega_j^{a_j}, \\
        R &\quad \text{ if } s_{\ort, j} = w_0 &\text{ and }\eta_1^{-1} = \chi \otimes \prod_j \omega_j^{a_j}, \\
        R &\quad \text{ if } s_{\ort, j} = \id \quad &\text{ and }\eta_1^{-1} = \chi \otimes \prod_j \omega_j^{b_j}, \\
        L &\quad \text{ if } s_{\ort, j} = w_0 &\text{ and } \eta_1^{-1} = \chi \otimes \prod_j \omega_j^{b_j}.
        \end{aligned}
    \end{cases} 
\end{align*}

Now, if $A = \{i-k, \dots, i\} \in M$, let $$A^{\text{i}} \stackrel{\text{def}}{=} \{i-k\}, \qquad A^{\text{o}} \stackrel{\text{def}}{=} \{i\}, \qquad
A^{*} \stackrel{\text{def}}{=} A \smallsetminus \{i-k\},$$ 
and $$M^{\text{i}}\stackrel{\text{def}}{=} \bigcup_{A \in M} A^{\text{i}}, \qquad M^{\text{o}} \stackrel{\text{def}}{=} \bigcup_{A \in M} A^{\text{o}}, \qquad M^{*} \stackrel{\text{def}}{=} \bigcup_{A \in M} A^{*}.$$
By definition, 
\begin{itemize}
    \item if $j \not\in M^{*} \cup M^{\text{i}}$, then $\gamma_j \leq (p-1)/2$;
    \item if $j \in M^{\text{i}}$, then $\gamma_{j}+1 \leq (p-1)/2$;
    %and $j+1 \not\in M^{*}$.
    \item if $j \in M^{*} \smallsetminus M^{\text{o}}$, then $p-1-\gamma_j < (p-1)/2$;
    \item if $j \in M^{\text{o}}$, then $p-\gamma_j \leq (p-1)/2$.
\end{itemize}

Setting
\begin{align}\label{eqn:chi-2}
    \chi^{-1} = \eta_2 \otimes \prod_{j \not\in M^{*}} \omega_j^{\gamma_j} \otimes \prod_{j \in M^{*} \smallsetminus M^{\text{o}}} \omega_{j}^{p-1} \otimes \prod_{j \in M^{\text{o}}} \omega_j^{p},
\end{align} we have $\tau^{\vee} \otimes \chi \cong$ 
$$
\left( \prod_{j \in M^{*} \smallsetminus M^{\text{o}}} \omega_{j}^{-p+1+\gamma_j} \otimes 
\prod_{j \in M^{\text{o}}} \omega_{j}^{-p+\gamma_j} \right)
\oplus  
\left( \prod_{j \not\in M^{*} \cup M^{\text{i}}} \omega_j^{-\gamma_j} \otimes \prod_{j \in M^{\text{i}}} \omega_j^{-\gamma_j-1} \right)
$$
and therefore, $\tau \otimes \chi^{-1} \cong$
$$
\left( \prod_{j \in M^{*} \smallsetminus M^{\text{o}}} \omega_{j}^{p-1-\gamma_j} \otimes 
\prod_{j \in M^{\text{o}}} \omega_{j}^{p-\gamma_j} \right)
\oplus  
\left( \prod_{j \not\in M^{*} \cup M^{\text{i}}} \omega_j^{\gamma_j} \otimes \prod_{j \in M^{\text{i}}} \omega_j^{\gamma_j+1} \right).
$$
Since $0 \not\in M^{*}$ by hypothesis,  \begin{align*}
    &a_j = \gamma_j, &&b_j = 0 &&\text{ for } j \not\in M^{*} \cup M^{\text{i}}, \\
    &a_j = \gamma_j + 1, &&b_j = 0 &&\text{ for } j \in M^{\text{i}}, \\
    &a_j = 0, &&b_j = p-1-\gamma_j &&\text{ for } j \in M^{*} \smallsetminus M^{\text{o}}, \\
    &a_j = 0, &&b_j = p-\gamma_j &&\text{ for } j \in M^{\text{o}}.
\end{align*} In particular, $\eta_1^{-1} = \chi \otimes \prod_{j \in \Z/f\Z} \omega_{j}^{b_j}$. Following through the steps of the algorithm to compute $s_j$, $s_{\ort, j}$ and $S_j$ immediately produces the desired statement. 
\end{proof}

\begin{definition}
    A sequence of integers $(a_{1}, \dots, a_{n})$ is said to satisfy $\heart$ if $a_1 = p-1$, $a_2 = \dots = a_{n-1} = 1$ and $a_n=0$.
\end{definition}
\begin{lemma}\label{lem:type-combinatorics}
    Let $(\gamma_j)_j, s, \mu, S$ be as in Lemma \ref{lem:mu-combinatorics}. Let $$\tilz \in \prod_{j \in \Z/f\Z}\{w_0 t_{\eta}, t_{w_0 (\eta)}\} s_j^{-1} v^{\mu_j}.$$ Let $T=(T_j)_j$ be the class tuple associated to the data of $\tilz, s, \mu, S$. The following are true:
    \begin{enumerate}
\item If $A = \{i-k, \dots, i\} \in M$, then 
        \begin{align*}
            &T_{j} \in \{1,2\} &&\text{ if } j =i-k,\\
            &T_j =2 &&\text{ if } 
        j \in A \smallsetminus \{i-k, i\}, \text{ and} \\
            &T_j \in \{3,4\} &&\text{ if } j =i.
        \end{align*}
        If $j$ is not in any subset contained in $M$, then $T_j \in \{3,4,5\}$. 
\item If $M = \varnothing$ and $\<\mu_j, \alpha^{\vee}\> = 1$ for each $j \in \Z/f\Z$, then for each $T \in \{3, 4\}^{\Z/f\Z}$, there exists $\tilz$ so that the class tuple associated to $\tilBa_S$ is $T$.     
\item  There exists $\tilz$ so that the set $\trn^{*}$ (associated to the data of $\tilz, s, \mu, S$) is non--empty if and only if $(\gamma_{j})_{j \in \Z/f\Z}$ contains some subsequence $(\gamma_{i-k}, \dots, \gamma_{i})$ satisfying $\heart$. The number of such subsequences is $\geq |\trn^{*}|$. Suitable $\tilz$ can be chosen so that equality holds, and the lowest value of $l(\gr)-1$ attained over $\gr \in \trn^{*}$ as $\tilz$ varies is the smallest length of a sequence  $(\gamma_{i-k}, \dots, \gamma_{i})$ satisfying $\heart$.
    \end{enumerate}

\end{lemma}
\begin{proof}
The proofs of (i) and (ii) follow from simple comparison of $(\mu, s, S)$ in the statements of Lemma \ref{lem:mu-combinatorics} with the classification of $\tilBaj_{S_j}$ in Section \ref{subsec:classfn} and Tables \ref{table:shape-L} and \ref{table:shape-R}.

For (iii), we first make the following observation: If $\gr = (i-1, i, \dots, i+l) \in \trn^{*}$, then by definition of $\trn^{*}$ and the constraints on $T$, $i \in M^{\text{o}}$; $T_{j} = 3$ for each $j \in \{i, \dots, i+l-1\}$; and
either $i+l \in M^{\text{i}}$ and $T_{i+l} = 1$, or $i+l \not\in M^{*} \cup M^{\text{i}}$ and $T_{i+l} = 5$.
% \begin{center}
% $T_{j} = 3$ for each $j \in \{i-k, \dots, i-1\}$, and \\
% either $l = 0$ and $T_i = 1$, or $l \neq 0$ and $T_i = 5$. \end{center}

By comparison with Section \ref{subsec:classfn} and Tables \ref{table:shape-L} and \ref{table:shape-R}, $T_{i} = 3$ implies $\gamma_{i} = p-1$; $T_{j} = 3$ for $j \in\{i, \dots, i+l-1\} \smallsetminus \{i\}$ implies $\gamma_j = 1$; and $T_{i+l} \in \{1, 5\}$ implies $\gamma_{i+l} = 0$. On the other hand, if the sequence $(\gamma_{i}, \dots, \gamma_{i+l})$ satisfies $\heart$, then Tables \ref{table:shape-L} and \ref{table:shape-R} demonstrate the existence of suitable $(\tilz_{i}, \dots, \tilz_{i+l})$ so that we get $\gr = (i-1, i, \dots, i+l) \in \trn^{*}$. Note that $l(\gr) -1$ is the length of the sequence $(\gamma_{i}, \dots, \gamma_{i+l})$.
\end{proof}

Now, we are ready to prove the main theorem of this section, which provides an upgrade of \cite[Thm.~5.0.1]{gkksw}.
\begin{theorem}\label{thm:serre-weight-sing}
Let $\sigma = \sigma_{\mathbb{m}, \mathbb{n}}$ be a non--Steinberg Serre weight. The following are true:
\begin{enumerate}
    \item The component $\cX(\sigma)$ is not smooth if and only if one of the following two holds:
    \begin{enumerate}[label=(\alph*)]
        \item For each $j \in \Z/f\Z$, $n_j = p-2$.
        \item There exists a subset $\{i-k, \dots, i\} \subset \Z/f\Z$ with $n_{i-k} = 0$, $n_j = p-2$ if $j \in \{i-k, \dots, i\} \smallsetminus \{i-k, i\}$, and $n_{i} = p-1$.
    \end{enumerate}
    \item If (a) holds, $\cX(\sigma)$ is not normal and its normalization admits a smooth--local cover by a resolution--rational variety that is not Gorenstein. The non--normal locus on $\cX(\sigma)$ has codimension $f$ and its complement is smooth.
    \item If (b) holds, $\cX(\sigma)$ is normal and admits a smooth--local cover by a resolution--rational variety. It is additionally Gorenstein, even lci, if and only if every subset $\{i-k, \dots, i\} \subset \Z/f\Z$ as in (b) has cardinality $2$. The singular locus in $\cX(\sigma)$ has codimension $\geq 2$ and its complement is smooth.
\end{enumerate}
%Product is Gorenstein if and only if each factor is by Stacks OBJL.
%In all cases, the ring of global functions on the normalization of $\cX(\sigma)$ is isomorphic to $\F[x,y][\frac{1}{y}]$.
\end{theorem}
\begin{proof}
When neither (a) nor (b) holds, this is the main result of \cite{gkksw} 
when $n_j \neq 0$ for some $j \in \Z/f\Z$ and follows, for e.g., from the Appendix of loc. cit. when $n_j=0$ for each $j \in \Z/f\Z$.

For the rest of the cases, \cite[Prop.~4.1.2]{gkksw} shows that $\cX(\sigma)$ is the scheme--theoretic image of $\cC^{\tau^{\vee}}(\Z/f\Z)$ for a principal series $\tau^{\vee} = \eta_1 \oplus \eta_2$ with $$\eta_1 \eta_2^{-1} = \prod_{j \in \Z/f\Z} \omega_j^{p-1-n_j}.$$ Thus, $\cX(\sigma)$ is the scheme--theoretic image of $\Y_S$ where, when (a) holds, $s, \mu, S$ are as in Lemma \ref{lem:mu-combinatorics} with the set $M = \varnothing$, and $\langle\mu_j, \alpha^{\vee}\rangle =1$ for each $j$. When (b) holds, then $s, \mu, S$ are as in Lemma \ref{lem:mu-combinatorics} with $M \neq \varnothing$ (after possibly relabelling $0 \in \Z/f\Z$ so as to guarantee that $0 \in M^{\text{o}}$).

An application of Theorem \ref{thm:sing} in conjunction with Lemma \ref{lem:type-combinatorics} shows that $\cX(\sigma) = \cZ^{\tau}_S$ is not smooth when (a) or (b) holds, as well as describes the singularities. The statements about the codimensions of the singular loci follow from Corollaries \ref{cor:R1} and \ref{cor:smooth-away-from-nonCM}.
%By \Cref{diagram:nm}, the global functions on the normalization of $\cZ^{\tau, \text{nm}}$ are the $GL_2^{\Z/f\Z}$--invariant global functions on $\tilZnm_S$. By \Cref{eqn:Stein-fac}, these are the same as the $GL_2^{\Z/f\Z}$--invariant global functions on $\tilY_S$, which in turn are the same as the global functions on $\Y_S = \cC^{\tau^{\vee}}(\Z/f\Z)$. %An application of Proposition \ref{prop:global-fn} finishes the proof.
\end{proof}

For the following theorem, the statement as well as the key argument in the proof was pointed out to us by Toby Gee.
\begin{theorem}\label{thm:steinberg}
    Let $\sigma$ be a Steinberg Serre weight. Then $$\cX(\sigma) \cong \left[\big(\Gm \times \AA^{f+1}\big) / \big(\Gm \times \Gm \big) \right].$$ In particular, $\cX(\sigma)$ is smooth.
\end{theorem}
\begin{proof}
    After twisiting by a $G_K$-character if necessary, we may assume that every finite type point of $\cX(\sigma)$ has crystalline lifts with Hodge--Tate weights $\{-p, 0\}$ in each embedding. Denote by $\cX(\text{triv})$ the irreducible component of $\cX_{2, \text{red}}$ whose finite type points have crystalline lifts with Hodge--Tate weights $\{-1, 0\}$ in each embedding. 
    
    By
    \cite[Sec.~3]{kans22}, for e.g., there exists a scheme 
    $V_{\epsilon, b}$ whose Artinian points parameterize unramified twists of extensions of the trivial character by the mod $p$ cyclotomic character. There also exists an obvious map $f_{\epsilon, b}: V_{\epsilon, b} \to \cX_{2, \text{red}}$ (see \textit{loc. cit.} for details). Combining the description of $V_{\epsilon, b}$ in  \textit{loc. cit.} with \cite[Thm.~5.1.29]{emerton2022moduli} and usual Galois cohomology computations, we find that $$V_{\epsilon, b} \cong \Gm \times \AA^{f+1}.$$

By \cite[(3.12)]{kans22}, there exists a monomorphism 
    \begin{align*} \Gm \times \Gm \times V_{\epsilon, b} \hookrightarrow V_{\epsilon, b} \times_{\cX_{2, \text{red}}} V_{\epsilon, b}.\end{align*}
    % where the codomain is the diagonal of $f_{\epsilon, b}$.
    As in \cite[Lem.~3.15]{kans22}, it is immediate from the explicit description of this map that it induces a bijection on points valued in Artinian rings with finite residue field. Therefore, it is a smooth surjection, and in particular, an isomorphism. This further implies that it induces a monomorphism
\begin{align}\label{eqn:sm-presentation-St}
    [V_{\epsilon, b}/\Gm \times \Gm] \hookrightarrow \cX_{2, \text{red}}.
\end{align}    
Since $V_{\epsilon, b}$ is reduced and irreducible, the same is true for the scheme--theoretic image of the map above. The key point is that by \cite[Lem.~8.6.4]{emerton2022moduli}, the scheme--theoretic image is $\cX(\sigma)$, onto which (\ref{eqn:sm-presentation-St}) surjects.

We claim that (\ref{eqn:sm-presentation-St}) induces the desired isomorphism in the statement of the Theorem.
As before, it suffices to check that $f_{\epsilon, b}$ is formally smooth. This is equivalent to checking the following: if $A$ is an Artinian local ring over $\F$ with residue field $\F'$ finite over $\F$, and $\rho_A \in \cX(\sigma)(A)$ is a $G_K$--representation deforming $\rhobar \in \cX(\sigma)(\F')$, then $\rho_A$ is an extension of the trivial character by the mod $p$ cyclotomic character (base changed to $A$).

Let $R^{\square}_{\rhobar}$ be the universal lifting $\cO$-algebra of $\rhobar$. Let 
    $R^{\mathrm{BT}, \text{ss}}_{\rhobar}$ and $R^{\mathrm{BT}, \text{crys}}_{\rhobar}$ 
be the $\cO$-flat quotients of the universal lifting algebra such that if $B$ is a finite flat $E$-algebra, then an $\cO$-algebra map $$R^{\square}_{\rhobar} \to B$$ factors through $R^{\mathrm{BT}, \text{ss}}_{\rhobar}$, resp. $R^{\mathrm{BT}, \text{crys}}_{\rhobar}$, if and only if the corresponding $G_K$-representation is semistable, resp. crystalline, of Hodge type $\mathrm{BT}$, by which we mean it has Hodge--Tate weights $\{-1, 0\}$ for each embedding $K \hookrightarrow E$. The existence of these flat quotients follows from the main theorem of \cite{kisindefrings}. Moreover, these quotients are equidimensional and of the same dimension whenever non--zero, implying that for any fixed irreducible component in the generic fiber of $R^{\mathrm{BT}, \text{ss}}_{\rhobar}$, the finite flat points over $E$ are either all crystalline, or all semistable but not crystalline. Thus, it makes sense to call an irreducible component in the generic fiber crystalline or non--crystalline. 
%The scheme $\Spec R^{\mathrm{BT}, \text{ss}}_{\rhobar}/\varpi$ has two irreducible components $\Spec R^{\mathrm{BT}, \text{ss}}_{\rhobar}/\varpi \times_{\cX_{2, \text{red}}} \cX(\text{triv}) = \Spec R^{\mathrm{BT}, \text{crys}}_{\rhobar}/\varpi \times_{\cX_{2, \text{red}}} \cX(\text{triv})$ and $\Spec R^{\mathrm{BT}, \text{ss}}_{\rhobar}/\varpi \times_{\cX_{2, \text{red}}} \cX(\sigma)$. 
Additionally, by \cite[Prop.~4.8.10, Thm.~8.6.2]{emerton2022moduli}, there exist versal maps
\begin{align*}
    &\Spec R^{\mathrm{BT}, \text{crys}}_{\rhobar}/\varpi \longrightarrow \cX(\text{triv}), \text{ and} \\
    &\Spec R^{\mathrm{BT}, \text{ss}}_{\rhobar}/\varpi \longrightarrow \cX(\sigma) \cup \cX(\text{triv}),
\end{align*} where $\cX(\sigma) \cup \cX(\text{triv})$ is a reduced union.

By flatness, $$\Spec R^{\mathrm{BT}, \text{ss}}_{\rhobar}/\varpi \times_{\cX_{2, \text{red}}} \cX(\sigma)$$ is in the closure of an irreducible component of the generic fiber of $\Spec R^{\mathrm{BT}, \text{ss}}_{\rhobar}$.
This irreducible component must be non--crystalline; otherwise, we would have an inclusion $$\Spec R^{\mathrm{BT}, \text{ss}}_{\rhobar}/\varpi \times_{\cX_{2, \text{red}}} \cX(\sigma) \subset \Spec R^{\mathrm{BT}, \text{crys}}_{\rhobar}/\varpi = \Spec R^{\mathrm{BT}, \text{ss}}_{\rhobar}/\varpi \times_{\cX_{2, \text{red}}} \cX(\text{triv}),$$ which is an impossibility.
Denote by $R$ the $\cO$-flat quotient of $R^{\mathrm{BT}, \text{ss}}_{\rhobar}$ obtained by taking the closure of this non--crystalline component. 

The closed immersion $$\Spec R^{\mathrm{BT}, \text{ss}}_{\rhobar}/\varpi \times_{\cX_{2, \text{red}}} \cX(\sigma) \hookrightarrow \Spec R/\varpi$$ implies that a quotient of $R/\varpi$ is versal to $\cX(\sigma)$ at $\rhobar$. In particular, by versality, $\rho_A$ lifts to an $A$-point of $R$. 
Thus, \cite[Lem.~4.1.2]{bartlett2020irreducible} shows that $\rho_A$ is obtained by a suitable base change from a semistable but not crystalline $G_K$-representation of Hodge type $\mathrm{BT}$. By the proof of \cite[Thm.~8.6.1]{emerton2022moduli}, all such representations are extensions of the trivial character by the cyclotomic character, proving the desired statement for $\rho_A$.
\end{proof}

\begin{corollary}\label{cor:max-CM-normal}
    Let $f>1$ and $\sigma$ a Serre weight. If $\cF$ is a finite type maximal Cohen--Macaulay sheaf on $\cX(\sigma)$, then $\cF$ is the pushforward of a maximal Cohen--Macaulay sheaf on the normalization of $\cX(\sigma)$. 
\end{corollary}
\begin{proof}
    The codimension of the singular locus is $\geq 2$ by Theorems \ref{thm:serre-weight-sing} and \ref{thm:steinberg}. An application of Lemma \ref{lem:max-CM-end} finishes the proof.
\end{proof}

\begin{lemma}\label{lem:max-CM-end}
    Let $M$ be a finitely generated maximal Cohen--Macaulay module over a noetherian integral domain $A$ that is regular in codimension $1$. Then $M$ is obtained by restriction of scalars from a maximal Cohen--Macaulay module defined over the normalization of $A$.
\end{lemma}
\begin{proof}
    Being maximal Cohen--Macaulay on an integral domain, the only associated prime of $M$ is the zero ideal. Hence, $M$ is torsion-free and admits an injection
    $$M \hookrightarrow M \otimes_A \mathrm{Frac}(A).$$
    We will show that $M$ equals the finitely generated $A$--submodule $M' \subset M \otimes_A \mathrm{Frac}(A)$ generated by the action of the normalization of $A$ on $M$.
    We can assume $A$ is local with maximal ideal $\mathfrak{m}$. 
    We prove the statement by induction on the dimension of $A$. The case when $\dim A \in \{0, 1\}$ is immediate since $A$ is normal in that case. Now, suppose $\dim A \geq 2$. By construction, $M'$ is torsion-free and hence, has depth $\geq 1$. By induction, if $M'/M$ is non--zero, then its support is $\{\mathfrak{m}\}$ and in particular, its depth is $0$. Consider the short exact sequence
    $$0 \to M \to M' \to M'/M \to 0$$ and the long exact sequence obtained from it by applying the functor $\Hom_{A}(A/\mathfrak{m}, \rule{0.2cm}{0.1pt})$. Since $\mathrm{depth} \; M' \geq 1$, nonzero $M'/M$ implies that $\mathrm{depth} \; M = 1$, a contradiction.
\end{proof}

\begin{cor}\label{cor:line-bundle-pushfwd}
    Let $f>1$, $\sigma  = \sigma_{\mathbb{m},\mathbb{n}}$ a Serre weight, and $\iota: \cU \hookrightarrow \cX(\sigma)$ the smooth open locus in $\cX(\sigma)$. Suppose $\cF$ is a finite type maximal Cohen--Macaulay sheaf on $\cX(\sigma)$ generically of rank $1$. The following are true:
    \begin{enumerate}
        \item The sheaf $\cF$ is isomorphic to the pushforward along $\iota$ of the invertible sheaf  $\iota^{*} \cF$ on $\cU$.
        \item If $\sigma$ is non--Steinberg and there does not exist $i$ such that $(n_{i-1}, n_i) = (0, p-1)$, then $\cF$ is the pushforward under $\pi_S$ of a unique invertible sheaf on $\Y_S$ for some suitable $\tau, S$.
    \end{enumerate}
\end{cor}
\begin{proof} When $\cX(\sigma)$ is smooth, (i) is trivial and (ii) follows from \cite[Prop.~4.2.1]{gkksw} when $n_j \neq 0$ for some $j \in \Z/f\Z$ and from the Appendix of \text{loc. cit.} when $n_j = 0$ for each $j \in \Z/f\Z$. Otherwise, fix $\tau$ and $S$ such that $\cX(\sigma)= \cZ^{\tau}_S$. 

%By Theorem \ref{thm:serre-weight-sing}, the singular locus on $\cZ^{\tau}_{S}$ has codimension $\geq 2$. Therefore, Lemma \ref{lem:max-CM-end} below shows that $\cF$ is obtained as the pushforward of a maximal Cohen--Macaulay sheaf on the normalization of $\cZ^{\tau}_S$. Thus, after replacing $\cZ^{\tau}_S$ by its normalization if necessary, we may assume that $\cZ^{\tau}_{S}$ is normal. 

Denote by $F$ the pullback of $\cF$ to $\tilZ_S$ and let $j:U \hookrightarrow \tilZ_S$ be the pullback of $\iota$ along the map $\tilZ_S \to \cZ^{\tau}_S$. 
By Theorem \ref{thm:serre-weight-sing}, $\tilZ_S \smallsetminus j(U)$ has codimension $\geq 2$ in $\tilZ_S$. 
Applying \cite[Thm.~3.5]{hassett2004reflexive} to the structure map $\tilZ_S \to \Spec \F$, we infer the existence of an isomorphism $$F \xrightarrow{\sim} j_{*} j^{*} F.$$ 

By Corollaries \ref{cor:R1} and \ref{cor:smooth-away-from-nonCM}, $\widetilde{\pi}_S^{-1}U$ is isomorphic to $U$ under $\widetilde{\pi}_S$. We identify the two, and denote by $$j_Y: U \hookrightarrow \tilY_S$$ the pullback of $j$ along $\widetilde{\pi}_S$. We thus have $j = \widetilde{\pi}_S \circ j_Y. $ Since $U$ is smooth, the Auslander--Buchsbaum formula implies that $j^{*}F$ is invertible. Therefore, there exists a Weil divisor $D \subset \tilY_S$ such that $j^{*}F \cong j_Y^{*} \cO(D)$, where $\cO(D)$ is the (invertible) sheaf associated to $D$. 
 %Using Corollary \ref{cor:smooth-away-from-nonCM} and Lemma \ref{lem:type-combinatorics}, the hypotheses in (ii) guarantees that the complement of $j_Y(U)$ has codimension $\geq 2$ in $\tilY_S$, and so, 

 Now, suppose (ii) holds. Equivalently, the complement of $j_Y(U)$ in $\tilY_S$ has codimension $\geq 2$ by Corollaries \ref{cor:R1} and \ref{cor:smooth-away-from-nonCM}. By the algebraic Hartog's Lemma,  $\cO(D)$ is the unique invertible sheaf restricting to $j^{*} F$ on $j_{Y}(U)$. Additionally,
 $(j_Y)_{*} j^{*} F \cong \cO(D).$

Descent data on $F$ corresponding to the sheaf $\cF$ restricts to descent data on $j^{*}F$. Since $j(U)$ is a dense open subscheme of $\tilZ_S$ and $F$ has no embedded primes being maximal Cohen--Macaulay, descent data on $j^{*}F$ uniquely extends to descent data on $j_{*} j^{*} F$. When the complement of $j_Y(U)$ has codimension $2$, the descent data on $j^{*}F$ also extends uniquely to descent data on $\cO(D)$. Therefore, by descending, we finish the proofs of parts (i) and (ii).
\end{proof} 

\begin{remark}
    To be precise, the pushforward functors considered in Corollary \ref{cor:line-bundle-pushfwd} are the quasicoherent pushforwards defined in \cite[\href{https://stacks.math.columbia.edu/tag/077A}{Tag 077A}]{stacks-project}.
\end{remark}

\begin{theorem}\label{thm:non--normal-rep}
 Let $\sigma = \sigma_{\mathbb{m}, \mathbb{n}}$ be a Serre weight. The versal ring at $\rhobar \in \cX(\sigma)(\Fbar)$ is not normal if and only if $n_j = p-2$ for each $j$ and, viewing $\rhobar$ as a $G_K$--representation, 
 $$\rhobar \otimes \left( \prod_{j \in \Z/f\Z} \omega_j^{(1-m_j)} \right)$$ is unramified.
%  $ is of the form
%  $$\left( \prod_{j \in \Z/f\Z} \omega_j^{(m_{j} - 1)} \right) \otimes \begin{pmatrix}
%     \ur_{\lambda'} & * \\
%     0 & \ur_{\lambda''}
% \end{pmatrix}$$
% where $\lambda'$ and $\lambda''$ are arbitrary units in $\Fbar$, and
% \begin{itemize}
%     \item $*$ is vanishing if $\lambda' \neq \lambda''$, and
%     \item $*$ lies in the $1$--dimensional space of extension classes that vanish after restriction to $I_K$ if $\lambda'=\lambda''$.
% \end{itemize}
\end{theorem}
\begin{proof}
As noted in the proof of Theorem \ref{thm:serre-weight-sing}, $\cX(\sigma)$ is the scheme--theoretic image of $\cC^{\tau^{\vee}}(\Z/f\Z) = Y_{S}^{\eta, \tau}$, where $\tau^{\vee} \cong \eta_1 \oplus \eta_2$ with $$\eta_1 = \prod_{j \in \Z/f\Z} \omega_j^{m_j} \text{ and } \qquad \eta_2 = \prod_{j \in \Z/f\Z} \omega_j^{m_j - 1},$$
and
$s, \mu, S$ are as in Lemma \ref{lem:mu-combinatorics}.  
%In particular, we note by \cite[Prop.~4.1.2]{gkksw} that  
Therefore $\mu_j = (-m_j +1, -m_j)$ for each $j \in \Z/f\Z$.

By Corollary \ref{cor:loc-non--normal}, the versal ring at $\rhobar$ is not smooth, if and only if it is not normal, if and only if it lifts to a point in the vanishing locus of $N$ in $\tilZz_S$ for the unique (by Table \ref{table:shape-R}) $$\tilz \in \prod_{j \in \Z/f\Z}\{w_0 t_{\eta}, t_{w_0 (\eta)}\} s_j^{-1} v^{\mu_j}$$ such that the class tuple $(T_j)_j$ associated to $\tilBa_S$ satisfies $T_j = 3$ for each $j \in \Z/f\Z$. By the proof of Lemma \ref{lemma:shape-conditions}, $V(N)$ is the image of the closed locus in $\tilYz_S$ cut out by setting $X_j =\id$ for each $j$. 

Let $\gM$ be a Breuil--Kisin module with $\Fbar$--coefficients in the image of this locus under the map $\tilYz_S \to \Y_S$. The matrix $A_{\gM, \beta}^{(j)}$ %(for $\beta$ ordered with respect to the ordering $(\eta_2, \eta_1)$ of eigenvalues by the proof of Lemma \ref{lem:mu-combinatorics})  
is given by $W_j$ described in (\ref{eqn:W_j}). 
By \cite[Table~3]{lhmm}, $l_j \kappa_j = r_j$ for each $j$, or equivalently, $\widetilde{l}_j \kappa_j \widetilde{r}_j^{-1} \in B(\Fbar)$. 
Therefore, $$A_{\gM, \beta}^{(j)} \in 
B(\Fbar) \begin{pmatrix}
    1 & 0 \\
    0 & v
\end{pmatrix}.$$ 
On the other hand, one verifies immediately that for any $(b_j)_j \in B(\Fbar)^{\Z/f\Z}$,
$$([0:1], b_j, \id, [0:1])_j \in \tilBa_S$$ gives a point of  $\tilYz_S$ in the fiber over $V(N)$. Equivalently, the finite type points in the non--normal locus are precisely the \'etale $\varphi$--modules admitting Breuil--Kisin models (with descent data of type $\tau$) $\gM$ with $$A_{\gM, \beta}^{(j)} =  \begin{pmatrix}
    \lambda'_j & x_j v \\
    0 & \lambda''_j v
\end{pmatrix}$$ 
% \text{ equivalently } C_{\gM, \beta}^{(j)} =  \begin{pmatrix}
%     \lambda_j & x_j v u^{\sum_{i=0}^{f-1} p^{i}} \\
%     0 & \lambda''_j v
% \end{pmatrix},$$
where $\lambda'_j, \lambda''_j$ are invertible while $x_j$ is an arbitrary scalar for each $j \in \Z/f\Z$. Following the classification of rank $1$ Breuil--Kisin modules with descent data in \cite[Lem.~4.1.1]{cegsC}, such Breuil--Kisin modules are certain extensions of 
$\gM'' \stackrel{\text{def}}{=} \gM(r'', (\lambda''_j)_j, c'')$ by $\gM' \stackrel{\text{def}}{=} \gM(r', (\lambda'_j)_j, c')$, where for each $j \in \Z/f\Z$, $r''_j = p^{f}-1$, $r'_j = 0$, $c''_j = \sum_{i \in \Z/f\Z} p^{i} m_{j-i}$, $c'_j = \sum_{i \in \Z/f\Z} p^{i} (m_{j-i} - 1)$. We claim that the extension class is $0$ whenever $\prod_j \lambda'_j \neq \prod_j \lambda''_j$. Otherwise, the space of extensions is $1$--dimensional. To see this, we simplify the matrices $A_{\gM, \beta}^{(j)}$ in the following way: First, by scaling the elements of $\beta$ appropriately, we can assume that $\lambda'_j = \lambda''_j =1$ for all $j \neq 0$. Let $$g_j = \begin{pmatrix}
    1 & \alpha_j \\
    0 & 1
\end{pmatrix}\in B(\Fbar).$$ The matrix $g_j A_{\gM, \beta}^{(j)} \left( \Ad\; s_j^{-1} v^{\mu_j} (g_{j-1}^{-1}) \right)$ is given by 
$$ \begin{pmatrix}
    \lambda'_j & v(x_j - \alpha_{j-1}\lambda'_j + \alpha_j \lambda''_j) \\
    0 & \lambda''_j v
\end{pmatrix}.$$ 
When $f=1$, if $\lambda'_0 \neq \lambda''_0$, take $\alpha_0 = x_0/(\lambda''_0-\lambda'_0)$ to kill $x_0$. On the other hand, if $\lambda'_0 = \lambda''_0$, then no choice of $\alpha_0$ can alter $x_0$. When $f \geq 2$, the extension class is seen to be vanishing if we can find $(\alpha_j)_j$ so that 
$$\underbrace{\begin{bmatrix}
   -1 & 1 & 0 & 0 & \dots& 0 & 0 & 0\\
   0 & -1 & 1 & 0 & \dots& 0 & 0 & 0\\
   \vdots & & & & \ddots& & & \vdots \\
   0 & 0 & 0 & 0 & \dots& -1 & 1 & 0\\
   0 & 0 & 0 & 0 & \dots& 0 &-1 & 1\\
   \lambda''_0 & 0 & 0 & 0 & \dots&0 &0 & -\lambda'_0
\end{bmatrix}}_{C} \begin{bmatrix}
    \alpha_{0} \\
    \alpha_1\\
    \vdots\\
    \alpha_{f-3} \\
    \alpha_{f-2}\\
    \alpha_{f-1}
\end{bmatrix} = - \begin{bmatrix}
    x_{1} \\
    x_2\\
    \vdots\\
    x_{f-2} \\
    x_{f-1}\\
    x_{0}
\end{bmatrix}.$$
Using cofactor expansion along the bottom row, we see that the determinant is $\pm(\lambda'_0 - \lambda''_0)$. Therefore, the extension class is vanishing when $\lambda'_0 \neq \lambda''_0$. Otherwise, we can choose any $\alpha_0$ and make successive choices of $\alpha_1, \dots, \alpha_{f-1}$ so that $x_j - \alpha_{j-1}\lambda'_j + \alpha_j \lambda''_j$ is $0$ for each $j \in \{1, \dots, f-1\}$. 

We need not consider any other types of change--of--basis matrices $(g_j)_j \in \Iw^{\Z/f\Z}$ because if each $x_j$ has to be killed, then using that $\< \mu_j, \alpha^{\vee} \> \leq p-2$, one checks immediately that it has to be killed by the constant part of $(g_j)_j$. Further, diagonal matrices cannot kill $x_j$. Following \cite[Defn.~4.2.4]{cegsC}, we additionally note that $(\gM'', \gM')$ has refined profile $(\Z/f\Z, r'')$. Therefore, by \cite[Prop.~5.1.8]{cegsC}, the extension class of $\gM[1/u]$ is non--vanishing if and only if that of $\gM$ is non--vanishing.

Computing the $G_{K}$--representations associated to $\gM'$ and $\gM''$ using \cite[Lem.~4.1.4]{cegsC}, we find that upto twist by a character, $\rhobar$ is an extension of $\ur_{\lambda''_0}$ by $\ur_{\lambda'_0}$ as desired, where the extension class is in a specific one-dimensional space of extensions when $\lambda'_0 = \lambda''_0$, and is $0$ when $\lambda'_0 \neq \lambda''_0$. (Note that the formula actually gives $G_{K_{\infty}}$--representations for $K_{\infty}$ a wildly ramified extension of $K$, but by \cite[Prop.~2.2.6]{cegsC}, this is enough.)

Finally, to characterize the $1$--dimensional extension class that survives when $\lambda'_0 = \lambda''_0$, let $\cM = \varepsilon_{\tau}(\gM)$. By Proposition \ref{prop:phi-module-frob}, for each $j$, there exists a basis of $\cM_j$ with respect to which the matrix of $\varphi_{\cM}^{(j)}: \cM_{j-1} \to \cM_j$  is given by $$A_{\gM, \beta}^{(j)} s_j^{-1} v^{\mu_j} = A_{\gM, \beta}^{(j)} \begin{pmatrix}
v & 0 \\
0 & 1
\end{pmatrix} v^{-m_j}.$$ We claim that the $G_{K_{\infty}}$--representation $T(\cM)$ splits after restriction to $G_{L_{\infty}}$, where $L$ is the unramified extension of $K$ of degree $p$ with residue field $l = \F_{p^{pf}}$ and $L_{\infty} = LK_{\infty}$. 
 Indeed, one can check easily that the \etale $\varphi$--module $\cM_L$ corresponding to $T(\cM)|_{L_{\infty}}$ is given by $l((v)) \otimes_{k((v))} \cM$. Thus, for $j \in \Z/pf\Z$, the matrix $F_j$ for the Frobenius map $\varphi_{\cM_L}^{(j)}: (\cM_{L})_{j-1} \to (\cM_L)_j$ with respect to the obvious basis is given by 

 \begin{align*}
    &\begin{pmatrix}
    1 & 0 \\
    0 &  1
    \end{pmatrix} v^{1-m_{j \text{ mod } f}} &&\text{ if } j \not\equiv 0 \mod f, \vspace{0.1cm}\\
    &\begin{pmatrix}
    \lambda'_0 & x_0 \\
    0 & \lambda'_0
    \end{pmatrix} v^{1-m_{j \text{ mod } f}} &&\text{ if } j \equiv 0 \mod f.
\end{align*}

Let
\begin{align*}
    g_j =
      \begin{pmatrix}
          1 & -d_j x_0/\lambda'_0 \\
          0 & 1
      \end{pmatrix}  
\end{align*}
where $d_j \in [0, p-1]$ is such that $j - d_j f \in \{0, 1, \dots, f-1\}$ mod $pf$. Then the change--of--basis that maps $F_j$ to $g_j F_j g_{j-1}^{-1}$ exhibits $\cM_L$ as a split extension of two rank one \etale $\varphi$--modules, as desired. Therefore, $\rho$ is split after restriction to $G_L$, and in particular to $I_L$.
%, by \cite[Lem.~5.4.2]{gls13}. 
By \cite[Cor.~3.2]{DDR}, the space of extensions that vanish after restriction to $I_K = I_L$ is $0$ when $\lambda'_0 \neq \lambda''_0$, and is $1$--dimensional when $\lambda'_0 = \lambda''_0$. We thus get the desired description of the $G_K$ representation associated to $\cM$.
%. Thus, by comparison of dimension, we get the desired description of the extension class that survives when $\lambda'_0 = \lambda''_0$. Moreover, 
\end{proof}

\begin{theorem}\label{thm:non-CM-rep}
    Let $\sigma = \sigma_{\mathbb{m}, \mathbb{n}}$ be a Serre weight with $n_{0} = p-1$, $n_{1} = 0$, 
and $n_j=p-2$ for $j \not\in \{0, 1\}$. The versal ring at $\rhobar \in \cX(\sigma)(\Fbar)$ is not smooth if and only if, viewing $\rhobar$ as a $G_K$--representation, 
 $$\rhobar \otimes \left(\omega_1^{-m_1} \otimes \prod_{j \neq 1} \omega_j^{1-m_j}\right)$$ is unramified.
% as a $G_K$--representation, $\rhobar$ is of the form
%  $$\left(\omega_1^{m_1} \otimes \prod_{j \neq 1} \omega_j^{m_j - 1}\right) \otimes \begin{pmatrix}
%     \ur_{\lambda'} & * \\
%     0 & \ur_{\lambda''}
% \end{pmatrix}$$
% where $\lambda$ and $\lambda''$ are arbitrary units in $\Fbar$, and
% \begin{itemize}
%     \item $*$ is vanishing if $\lambda' \neq \lambda''$, and
%     \item $*$ lies in the $1$--dimensional space of extension classes that vanish after restriction to $I_K$ if $\lambda'=\lambda''$.
% \end{itemize}
In this case, the versal ring is normal and Cohen--Macaulay. It is Gorenstein, even lci, if and only if $f=2$.
\end{theorem}

\begin{proof}
    As noted in the proof of Theorem \ref{thm:serre-weight-sing}, $\cX(\sigma)$ is the scheme--theoretic image of $Y_{S}^{\eta, \tau}$ where, by Lemma \ref{lem:mu-combinatorics}, $(\<\mu_j, \alpha^{\vee}\>,  s_j, s_{\text{or}, j},  S_j)$ equals
    \begin{align*}
        &(1, w_0, w_0, L) &&\text{ if } j = 0, \\
         &(1, w_0, \id, R) &&\text{ if } j = 1, \\
       &(1, \id, \id, R) &&\text{ if } j \not\in \{0, 1\}
    \end{align*} and by \cite[Prop.~4.1.2]{gkksw}, $\tau^{\vee} \cong \eta_1 \oplus \eta_2$ where $$\eta_1 = \prod_{j \in \Z/f\Z} \omega_j^{m_j}, \qquad \eta_2 = \prod_{j \in \Z/f\Z} \omega_j^{m_j + n_j}.$$ Thus, $\mu_j = (c_j + 1, c_j)$ for each $j$ where, by (\ref{eqn:chi-2}), $c_j = -m_j$ if $j \neq 1$ and $c_{1} = -m_{1} - 1$.
    
By Corollary \ref{cor:loc-non-CM} and the proof of Lemma \ref{lem:type-combinatorics}
we are looking for $\rhobar$ admitting a lift to $\tilZz_S$ for some $\tilz = (\tilw_j s_j^{-1} v^{\mu_j})_j$ with $ \tilw_j \in \{w_0 t_{\eta}, t_{w_0 (\eta)}\} $ such that the associated type class tuple $(T_j)_j$ satisfies
$T_{1} = \dots = T_{f-1} = 3$ and $T_{0} = 1$, so that the lift lies in the vanishing locus of $N(\gr)$ where $\gr = (0, 1, \dots, f-1, 0) \in \trn^{*}$. By comparing $s, \mu$ to Tables \ref{table:shape-L} and \ref{table:shape-R}, we find that $\tilw_0 = \tilw_{1} = w_0 t_{\eta}$ while $\tilw_j = t_{w_0 (\eta)}$ for $j \not\in \{0, 1\}$.

The proof of Lemma \ref{lemma:shape-conditions} shows that $V(N(\gr))$ is the 
image of the closed locus in $\tilYz_S$ cut out by setting $X_j = \id$ for each $j \neq 0$, and that $X_0$ is necessarily $\id$.
Let $\gM$ be a Breuil--Kisin module with $\Fbar$--coefficients in the image of this locus under the map $\tilYz_S \to \Y_S$. The matrix $A_{\gM, \beta}^{(j)}$ %(for $\beta$ ordered with respect to the ordering $(\eta_2, \eta_1)$ of eigenvalues by the proof of Lemma \ref{lem:mu-combinatorics})  
is given by $W_j$ described in (\ref{eqn:W_j}). 
By \cite[Table~3]{lhmm}, $l_j \kappa_j = r_j$,  equivalently $\widetilde{l}_j \kappa_j \widetilde{r}_j^{-1} \in B(\Fbar)$, for each $j \neq 0$. 
Therefore, there exist $b_j \in B(\Fbar)$ for $j \neq 0$ and $b_{0} \in \GL_2(\Fbar)$ such that
\begin{align}\label{eqn:non-CM}
    A_{\gM, \beta}^{(j)} =
    \begin{cases}
    b_{0} \begin{pmatrix}
    0 & 1 \\
    v & 0
\end{pmatrix} &\text{ if } j = 0,
\vspace{0.1cm}\\
        b_1 \begin{pmatrix}
    0 & 1 \\
    v & 0
\end{pmatrix} &\text{ if } j = 1,
\vspace{0.1cm}\\
b_j \begin{pmatrix}
    1 & 0 \\
    0 & v
\end{pmatrix} &\text{ if } j \not\in \{0, 1\}.
    \end{cases}
\end{align}
On the other hand, one verifies immediately that for any $(b_j)_j$ with $b_j \in B(\Fbar)$ when $j \neq 0$ and $b_{0} \in \GL_2(\Fbar)$,
$$([0:1], b_j, \id, [0:1])_j \in \tilBa_S$$ gives a point of  $\tilYz_S$ in the fiber over $V(N(\gr))$. In other words, the finite type points in the singular locus are precisely the \'etale $\varphi$--modules admitting Breuil--Kisin models (with descent data of type $\tau$) $\gM$ with $A_{\gM, \beta}^{(j)}$ satisfying (\ref{eqn:non-CM}) for each $j$.
Let $$b_0 = \begin{pmatrix}
    U & V \\
    W & Z
\end{pmatrix} \in \GL_2(\Fbar)$$ and for each $j \neq 0$, $$b_j = \begin{pmatrix}
    \lambda'_j & x'_j \\
    0 & \lambda''_j
\end{pmatrix}.$$ Therefore, 
\begin{align}\label{eqn:C-matrices}
    C_{\gM, \beta}^{(j)} =
    \begin{cases}
\begin{pmatrix}
    W & Zu^{p^{f}-1-l_0} \\
    Uu^{l_0} & Vv 
\end{pmatrix}&\text{ if } j = 0,
\vspace{0.1cm}\\
 \begin{pmatrix}
    x_1 v & \lambda'_1 u^{l_1}\\
    \lambda''_1 u^{p^{f}-1-l_1} & 0
\end{pmatrix} &\text{ if } j = 1,
\vspace{0.1cm}\\
\begin{pmatrix}
    \lambda'_j & x'_j u^{l_j} \\
    0 & \lambda''_j v
\end{pmatrix} &\text{ if } j \not\in \{0,1\},
    \end{cases}
\end{align}
for suitable integers $l_j \in (0, p^f-1)$. Specifically, $l_0 = p^{f-1} - \sum_{i=0}^{f-2}p^{i} $ and $l_1 = -1 + \sum_{i=1}^{f-1} p^{i}$. 
Let $\beta_j = (e_j, f_j)$ (note that it follows from the proof of Lemma \ref{lem:mu-combinatorics} that $\beta_j$ is ordered with respect to the ordering $(\eta_2, \eta_1)$ of the eigenvalues). 

 Suppose first that $b_0 \in B(\Fbar)$, equivalently $W=0$.
We observe from the matrices that $f_0 + e_1 + \dots + e_{f-1}$ gives a free generator of a sub-Breuil Kisin module $\gM'$ of $\gM$ and $e_0 + f_1 + \dots + f_{f-1}$ lifts a basis of the quotient $\gM''$. 
Thus, in the sense of \cite[Sec.~4.2]{cegsC}, $\gM$ is an extension of $\gM'' \cong \gM(r'', (\lambda''_j)_j, c'')$ by $\gM' \cong \gM(r', (\lambda'_j)_j, c')$
where $\lambda'_0 = U$, $\lambda''_0 = Z$,
\begin{align*}
    &r''_j = p^f - 1 -l_0, \quad &&r'_j = l_0 \quad &&\text{ if } j = 0,\\
    &r''_j = p^{f}-1-l_1, \quad &&r'_j = l_1  \quad &&\text{ if } j = 1,\\
    &r''_j = p^{f}-1, \quad && r'_j = 0 \quad  &&\text{ if } j \not\in \{0, 1\},
\end{align*}
and $\{c'_j\}_j, \{c''_j\}_j$ satisfy
 $$\omega_j^{c''_j} = \eta_1, \quad \omega_j^{c'_j} = \eta_2 \text{ when } j \neq 0,$$ while $\omega_0^{c''_0} = \eta_2$, $\omega_0^{c'_0} = \eta_1$. Note that $(\gM'', \gM')$ has refined profile $(\Z/f\Z \smallsetminus \{0\}, r'')$.
 Using exactly the same arguments as in the proof of Theorem \ref{thm:non--normal-rep}, we find that the Galois representations associated to such $\gM[1/u]$ are precisely those of the form
described in the statement of the theorem.

Next, suppose $W \in \Fbar^{\times}$. Consider a change of basis on $\gM$ that transforms $b_j = A_{\gM, \beta}^{(j)} \tilw_j^{-1}$ to $b'_j = g_j A_{\gM, \beta}^{(j)} \left( \Ad \; s_j^{-1} v^{\mu_j} (g_{j-1}^{-1}) \right) \tilw_j^{-1}$ with $g_j \in B(\Fbar)$ for each $j$. We claim that by doing a careful change of basis, we can arrange for $V = x_1 = \dots = x_{f-1}=0$. First, one sees immediately that by taking $g_j = \id$ 
for $j \not\in \{f-1, 0\}$ and suitable $g_{f-1}, g_{0} \in B(\Fbar)$, we can arrange $V = 0$. Similarly, it is easy to see that after doing a change of basis with $g_{f-1} = g_0 = \id$ and suitable $g_j \in B(\Fbar)$ for each $j \not\in \{0, f-1\}$, we can take $x_j =0$ whenever $j \not\in \{0, 1\}$ and further, $\lambda'_j = \lambda''_j = 1$ when $j \neq 0$. Next, take $$g_j = \begin{pmatrix}
    1 & \alpha_j \\
    0 & 1
\end{pmatrix}$$ where $\alpha_0 \in \Fbar$ satisfies $$W \alpha_0^{2} + (U - Z - W x_1)\alpha_0 - U x_1 = 0,$$ and $\alpha_1 = \dots = \alpha_{f-1} = \alpha_0 - x_1$.
Thus
\begin{align*}
    b'_j = \begin{cases}
        \begin{pmatrix}
            U + \alpha_0 W & -U\alpha_{f-1} - W\alpha_0 \alpha_{f-1} + Z\alpha_0  \\
        W & Z
        \end{pmatrix} &\text{ if } j = 0, \\
        \begin{pmatrix}
            1 & x_1 - \alpha_{0} + \alpha_1 \\
        0 & 1
        \end{pmatrix} &\text{ if } j = 1, \\
        \begin{pmatrix}
            1 & - \alpha_{j-1} + \alpha_j \\
        0 & 1
        \end{pmatrix} &\text{ if } j \not\in \{0, 1\}.
    \end{cases}
\end{align*}
By choice of $\alpha_j$, the top right entry of $b'_j$ vanishes for each $j$.

Now, assume $V = x_1 = \dots = x_{f-1} = 0$ and let $\cM = \varepsilon_{\tau}(\gM)$. By Proposition \ref{prop:phi-module-frob}, one can choose a basis for each $j$ so that the matrix of $\varphi_{\cM}^{(j)}: \cM_{j-1} \to \cM_j$ is given by $A_{\gM, \beta}^{(j)} s_j^{-1}v^{\mu}$. We have
\begin{align*}
    A_{\gM, \beta}^{(j)} s_j^{-1}v^{\mu} =
    \begin{cases}
\begin{pmatrix}
    U & 0  \\
    W & Z
\end{pmatrix}v^{-m_0 + 1} &\text{ if } j = 0,
\vspace{0.1cm}\\
\begin{pmatrix}
    \lambda'_j & 0 \\
    0 & \lambda''_j
\end{pmatrix} v^{-m_1} &\text{ if } j =1,
\vspace{0.1cm}\\
 \begin{pmatrix}
    \lambda'_j & 0\\
    0 & \lambda''_j
\end{pmatrix} v^{-m_j +1} &\text{ if } j \not\in \{0,1\}.
    \end{cases}
\end{align*}

After a change of basis involving permuting the two basis elements of $\cM_j$ for each $j$, we find that $\cM$ could also have been obtained via the first case, that is, when we assumed $b_0 \in B(\Fbar)$. Therefore, we get no new representations upon assuming that $W$ is a unit.
This finishes the proof.
\end{proof}

\section*{Declarations} This work was done while K.K. was supported by the National Science Foundation under Grant No. DMS--1926686. On behalf of all authors, the corresponding author states that there is no conflict of interest. Further, this study has no associated data.

\bibliographystyle{math}
\bibliography{nongeneric}

\newcommand{\etalchar}[1]{$^{#1}$}
\renewcommand{\MR}[1]{}\providecommand{\noopsort}[1]{}\renewcommand{\MR}[1]{}\renewcommand{\MR}[1]{}
\begin{thebibliography}{LLHLM}

\bibitem[ABI]{ascher2018moduli}
Kenneth Ascher, Dori Bejleri, and Giovanni Inchiostro.
\newblock {Moduli of weighted stable elliptic surfaces and invariance of log
  plurigenera}, 2018.

\bibitem[Bar]{bartlett2020irreducible}
Robin Bartlett.
\newblock {On the irreducible components of some crystalline deformation
  rings}.
\newblock In {\em Forum of Mathematics, Sigma}, volume~8, page e22. Cambridge
  University Press, 2020.

\bibitem[BBH{\etalchar{+}}]{bellovin2024irregular}
Rebecca Bellovin, Neelima Borade, Anton Hilado, Kalyani Kansal, Heejong Lee,
  Brandon Levin, David Savitt, and Hanneke Wiersema.
\newblock {Irregular loci in the Emerton--Gee stack for $\mathrm{GL}_2$}.
\newblock {\em Journal f{\"u}r die reine und angewandte Mathematik (Crelles
  Journal)} {\bf 2024}(2024), 9--46.

\bibitem[Bre]{breuil2007representations}
Christophe Breuil.
\newblock {Representations of Galois and of GL2 in characteristic p}, 2007.

\bibitem[CEGSa]{cegsA}
{\noopsort{CaraianiEmertonGeeSavittA}}{Ana Caraiani, Matthew Emerton, Toby Gee,
  and David Savitt}.
\newblock {The geometric {B}reuil--{M}\'ezard conjecture for two-dimensional
  potentially {B}arsotti-{T}ate {G}alois representions}.
\newblock Algebra \& Number Theory (to appear), 2022.

\bibitem[CEGSb]{cegsB}
{\noopsort{CaraianiEmertonGeeSavittB}}{Ana Caraiani, Matthew Emerton, Toby Gee,
  and David Savitt}.
\newblock {Local geometry of moduli stacks of two-dimensional {G}alois
  representations}.
\newblock Proceedings of the International Colloquium on `Arithmetic Geometry',
  TIFR Mumbai, Jan. 6-10, 2020 (to appear), 2022.

\bibitem[CEGSc]{cegsC}
{\noopsort{CaraianiEmertonGeeSavittC}}{Ana Caraiani, Matthew Emerton, Toby Gee,
  and David Savitt}.
\newblock {Components of moduli stacks of two-dimensional Galois
  representations}.
\newblock In {\em Forum of Mathematics, Sigma}, volume~12, page e31. Cambridge
  University Press, 2024.

\bibitem[DDR]{DDR}
Lassina Demb\'el\'e, Fred Diamond, and David~P. Roberts.
\newblock {Serre weights and wild ramification in two-dimensional {G}alois
  representations}.
\newblock {\em Forum Math. Sigma} {\bf 4}(2016), e33, 49.

\bibitem[EG]{emerton2022moduli}
Matthew Emerton and Toby Gee.
\newblock {\em Moduli Stacks of {\'E}tale ($\varphi$, $\Gamma$)-Modules and the
  Existence of Crystalline Lifts:(AMS-215)}.
\newblock Number 215 in Annals of Mathematics Studies. Princeton University
  Press, 2022.

\bibitem[EGH]{EGHcategorical}
Matthew Emerton, Toby Gee, and Eugen Hellmann.
\newblock {An introduction to the categorical $p$-adic {L}anglands program}.
\newblock Notes from the I.H.E.S.\ Summer School on the {L}anglands program,
  2022.

\bibitem[GKK{\etalchar{+}}]{gkksw}
Anthony Guzman, Kalyani Kansal, Iason Kountouridis, Ben Savoie, and Xiyuan
  Wang.
\newblock {Smoothness of components of the Emerton-Gee stack for
  $\mathrm{GL}_2$}.
\newblock {\em Transactions of the American Mathematical Society} (2024).

\bibitem[HK]{hassett2004reflexive}
Brendan Hassett and S{\'a}ndor~J Kov{\'a}cs.
\newblock {Reflexive pull-backs and base extension}.
\newblock {\em Journal of Algebraic Geometry} {\bf 13}(2004), 233--248.

\bibitem[Kan]{kans22}
Kalyani Kansal.
\newblock {Codimension one intersections between components of the
  {E}merton-{G}ee stack for ${\mathrm{GL}}_2$}.
\newblock \url{https://arxiv.org/abs/2210.05002}, 2022.

\bibitem[Kis]{kisindefrings}
Mark Kisin.
\newblock {Potentially semi-stable deformation rings}.
\newblock {\em J. Amer. Math. Soc.} {\bf 21}(2008), 513--546.

\bibitem[Kov]{kovacs2017rational}
S{\'a}ndor~J Kov{\'a}cs.
\newblock {Rational singularities}, 2022.

\bibitem[LLHLM]{LLLM-models}
Daniel Le, Bao~V. Le~Hung, Brandon Levin, and Stefano Morra.
\newblock {Local models for {G}alois deformation rings and applications}.
\newblock {\em Invent. Math.} {\bf 231}(2023), 1277--1488.

\bibitem[LHMM]{lhmm}
Bao~V. Le~Hung, Ariane M{\'e}zard, and Stefano Morra.
\newblock {Local model theory for non-generic tame potentially Barsotti--Tate
  deformation rings}, 2024.

\bibitem[San]{sandermultiplicities}
Fabian Sander.
\newblock {Hilbert-{S}amuel multiplicities of certain deformation rings}.
\newblock {\em Math. Res. Lett.} {\bf 21}(2014), 605--615.

\bibitem[{Sta}]{stacks-project}
The {Stacks project authors}.
\newblock {The Stacks project}.
\newblock \url{http://stacks.math.columbia.edu}, 2018.

\bibitem[Zhu]{zhu2016introduction}
Xinwen Zhu.
\newblock {An introduction to affine Grassmannians and the geometric Satake
  equivalence}, 2016.

\end{thebibliography}
\end{document}